\documentclass[11pt]{article}
\usepackage{fullpage,amsfonts,amssymb,amsmath,epsfig,graphicx,color}
\usepackage{comment}
\usepackage{authblk}
\usepackage{appendix}

\usepackage{amsthm}

\usepackage{subcaption}

\usepackage{hyperref}

\usepackage{todonotes}
\newcommand{\WM}[1]{\todo[inline,color=orange]{WM: #1}} % Martin
\newcommand{\IL}[1]{\todo[inline,color=yellow]{IL: #1}} % Ian
\newcommand{\BH}[1]{\todo[inline,color=green]{BH: #1}} % Bron

\newtheorem{proposition}{Proposition}[section]

\newtheorem{lemma}{Lemma}[section]
\newtheorem{definition}{Definition}

\newtheorem{remark}{Remark}[section]

    \title{A geometric singular perturbation analysis of generalised \\shock selection rules in reaction-nonlinear diffusion models}

\author[1]{Bronwyn H Bradshaw-Hajek}
\author[2]{Ian Lizarraga}
\author[2]{Robert Marangell}
\author[2]{Martin Wechselberger}
\affil[1]{UniSA STEM, University of South Australia, Mawson Lakes, South Australia, Australia}
\affil[2]{School of Mathematics and Statistics, University of Sydney, NSW, Australia}

\date{}

\begin{document}

\maketitle

\begin{abstract}
    
Reaction-nonlinear diffusion (RND) partial differential equations are a fruitful playground to model the formation of sharp travelling fronts, a fundamental pattern in nature. In this work, we demonstrate the utility and scope of regularisation as a technique to investigate shock-fronted solutions of RND PDEs, using geometric singular perturbation theory (GSPT) as the mathematical framework. In particular, we show that {\it composite} regularisations can be used to construct families of monotone shock-fronted travelling waves sweeping out distinct generalised area rules, which interpolate between the equal area and extremal area (i.e. algebraic decay) rules that are well-known in the shockwave literature. We further demonstrate that our RND PDE supports other kinds of shock-fronted solutions, namely, nonmonotone shockwaves as well as shockwaves containing slow tails in the aggregation (negative diffusion) regime.  Our analysis blends Melnikov methods---in both smooth and piecewise-smooth settings---with GSPT techniques applied to the PDE over distinct spatiotemporal scales.

We also consider the spectral stability of these new interpolated shockwaves. Using techniques from geometric spectral stability theory, we determine that our RND PDE admits spectrally stable shock-fronted travelling waves. The multiple-scale nature of the regularised RND PDE continues to play an important role in the analysis of the spatial eigenvalue problem. 
\end{abstract}

\section{Introduction}

Continuum transport models of coupled systems of cell populations have for the most part used standard linear diffusion to model population spread \cite{keenersneyd,murray02}.

Reaction-diffusion equations, such as the extensively studied Fisher equation \cite{fisher37}, are used to model population growth dynamics combined with a simple linear Fickian diffusion process, and are typically capable of exhibiting travelling wave solutions.\\

In cell migration, advection (or transport) is another source of pattern formation. It may represent, e.g., tactically-driven movement, where cells migrate in a directed manner in response to a concentration gradient \cite{Keller1971,murray02}. 
Such a concentration gradient develops, for example, in a soluble fluid (chemotaxis) or as a gradient of cellular adhesion sites or of substrate-bound chemoattractants (haptotaxis).  Well-studied examples of individual cells exhibiting directed motion in response to a chemical gradient include bacteria chemotactically migrating towards a food source. Wound healing, angiogenesis or malignant tumor invasion are just a few examples of chemotactic and/or haptotactic cell movement where the migrating cells form part of a dense population of cells as may be found in tissues.  Such migrating cell populations not only form travelling waves but may also develop sharp interfaces in the wave form   \cite{hvhmpw13,hvhmpw14,Hillen2009,landmanetal03,landmanetal08,marchantetal01,wechpet10}.  \\

Another important experimental observation is that motility varies with population density \cite{Browning2020,landetal10,P2007}. Such density-dependent nonlinear diffusion processes are also implicated in the formation of sharp interfaces. In the context of population dynamics, living cells make informed decisions through, e.g., sensing the local cell density, and they perform a `biased walk.' This could lead to, e.g., the tendency to cluster or aggregate with other nearby cells--think of flocking or swarming. Such aggregation mechanisms can be achieved through, e.g., density dependent negative (or backward) diffusion, and such nonlinear diffusion models with subdomains of backward diffusion are known to admit shock-type solutions \cite{hollig1983existence}. Our goal in this paper is to consider the emergence of shock-fronted travelling waves in {\it reaction-nonlinear diffusion} (RND) models that arise as continuum limits of discrete motile processes made by aggregations of cells or other biological agents \cite{Browning2020,Johnston2017,penington11}. \\

 In general, shocks are problematic because as the wave front steepens (and a shock forms) the solution becomes multivalued and physically nonsensical. The model breaks down and it becomes impossible to compute the temporal evolution of the solution \cite{mainietal07,mainietal10}. 
To deal with such ill-posed problems, shock solutions of PDEs are mathematically formalised as weak solutions in the appropriate function space, and it is generally understood that such solutions are nonunique---this is related to the issue of where exactly the shock discontinuity develops in the domain. Nonunique families of weak solutions are known to arise in models of nonlinear diffusion processes. In physical and chemical problems, shock selection is typically enforced by physical constraints, such as conservation laws, entropy and energy conditions, just to name a few. For example, in advection-reaction models they may represent hyperbolic balance laws, i.e., hyperbolic conservation laws with source terms, where the formation of shock fronts is well-known. These shock selection rules are usually referred to as admissibility criteria. We refer the reader to \cite{smoller83,whitham1999} for a comprehensive development of shock structure theory from this point-of-view.\\

Our approach  to deal with ill-posed problems and shock formation in this paper is to employ {\em regularisation}: we add small perturbative higher order terms to these models to give rise to well-posed problems and, hence, introduce smoothing effects. In the context of hyperbolic conservation/balance laws, these are usually small viscous (diffusive) regularisations, e.g., the well-known viscous Burgers equation. Due to dissipative mechanisms, these physical shocks are observed as narrow transition regions with steep gradients of field variables. Mathematically, questions of existence and uniqueness of such viscous shock profiles are fundamental.\footnote{Another option is dispersive regularisation, e.g., the Korteweg-de Vries (KdV) equation. Note that both regularisations (viscous and dispersive) deal with the same problem (inviscid Burgers equation) but create very different outcomes.}\\

Regularisation techniques have been employed in physical and chemical problems having nonlinear diffusion processes. These regularised models are usually referred to as phase separation problems \cite{Fife00}. The Cahn-Hilliard equation modelling nucleation in a binary alloy is probably the most famous of these phase separation problems \cite{cahn61,hilliard70,pego89}, while Sobolev regularisation of phase separation models is another technique \cite{novpego91}. These regularisation techniques are not so well-known \cite{CohenMurray81,landetal10,padron03,penington11} within the population dynamics modelling community.\\

Possible shock formation in such regularised RND models is the main focus of this article, and we will use tools from geometric singular perturbation theory and dynamical systems theory to tackle this problem. 
Again, questions of existence and uniqueness of such regularised shock profiles are fundamental. In particular, we focus on the shock selection criteria based on different composite regularisations. It is worth noting that Witelski considered shock formation in regularised advective nonlinear diffusion models in \cite{WITELSKI199527,Witelski96} by means of singular perturbation theory, i.e., he also included advection or transport phenomena in his study. We, on the other hand, focus on nonlinear diffusion as the sole shock formation mechanism. Furthermore, we also consider the spectral stability of these regularised shock waves, and we even construct new kinds of regularised waves, including nonmonotone shockwaves as well as shockwaves containing slow passage through regions of negative diffusion---all of which extends Witelski's approach significantly.\\

The manuscript is structured as follows: in section 2 we introduce our RND model and its composite regularisations. In sections 3 and 4 we show the existence of travelling and standing waves in such models using  GSPT machinery. In section 5 we then show spectral stability results for monotone regularised shock waves, and we conclude in section 6. We summarise the Melnikov theory we use---in particular, a new piecewise-smooth adaptation of the Melnikov integral---in the Appendix.

\section{The setup for RND Models}

We start by considering a dimensionless reaction--nonlinear diffusion model of the form
\begin{equation}\label{eq:rnd}
u_t =  ( D(u) u_x)_x   + f(u) = \Phi(u)_{xx}    + f(u)
\end{equation}
where $x\in\mathbb{R}$ denotes the spatial domain, $t\in\mathbb{R}_+$ denotes the time domain, $u(x,t)\in \mathbb{R}_+$ denotes a (population/agent) density, $D(u)$ models a (population/agent) density dependent diffusivity. $\Phi(u)$ is an anti-derivative of $D(u)$, i.e. $\Phi'(u)=D(u)$, referred to as the \emph{potential}.

The (dimensionless) population/agent density $u$ is scaled such that  $u\in[0,1]$ forms the domain of interest where $u=1$ is the carrying capacity of the population/agent density. This domain of interest is also reflected in the reaction term $f(u)$ which is often modelled either as {\em logistic growth}

or as {\em bistable growth}. We consider the latter in this paper:
\begin{equation}\label{f:bistable}
f(u)=\kappa u(u-\alpha)(1-u)\,,\,\, \kappa>0,\,0<\alpha<1.    
\end{equation}

We focus on RND models where not only diffusion is present but also {\em aggregation} (or {\em backward diffusion}) \cite{cahn61, landetal10, hilliard70,  murray02, padron03}.  

By imposing different motility rates for agents that are isolated compared to other agents, one obtains density dependent nonlinear diffusion \cite{Johnston2017}. Aggregation will manifest itself in these models in sign changes of the density dependent diffusion coefficient $D(u)$.
The simplest density-dependent nonlinear diffusion coefficient model that we consider is of the polynomial form

\begin{equation}\label{eq:phi1}
D(u)=\beta(u-\gamma_1)(u-\gamma_2)
\end{equation}
with $0<\gamma_1<\gamma_2<1$, i.e., we model diffusion-aggregation-diffusion (DAD) in the domain of interest.

\begin{comment}
\BH{does it matter if $\gamma_1+\gamma_2>1$ or $<1$? Tom has found that for some things like stability, this makes a difference. But perhaps it doesn't for the things presented here.} 
\IL{Good question---I have no idea. How does the destabilization happen in your \& Tom's case? For the parameter set considered in the paper (Fig. \ref{fig:fullhet}, i.e. $\gamma_1 + \gamma_2 > 1$) it looks like things are nicely stable, and I vaguely remember from considering a couple other parameter sets that there were no surprises (but I don't remember what $\gamma_1 + \gamma_2$ was in those other scenarios).}
\WM{A priori there is no apparent reason to restrict $\gamma_1+\gamma_2<1$ since we only require a quadratic form with two zeros within $(0,1)$. If Tom found something of interest, please let us know.}
\BH{In Tom's case things become unstable for $\gamma_1 + \gamma_2 > 1$, but on a second look I think this is because Tom's case is very specific in that he has the special relationship b/w $D$ and $R$. In his case, the stability depends on the sign of $R'(u)$, which depends on whether $\gamma_1 + \gamma_2 > 1, <1, =1$. I don't think that would carry forward to the more general case here.}
\WM{Thx for clarifying Bron. I don't think there is any update needed here.}
\end{comment}

\begin{figure}[t]
    \centering
    \includegraphics[width=9cm]{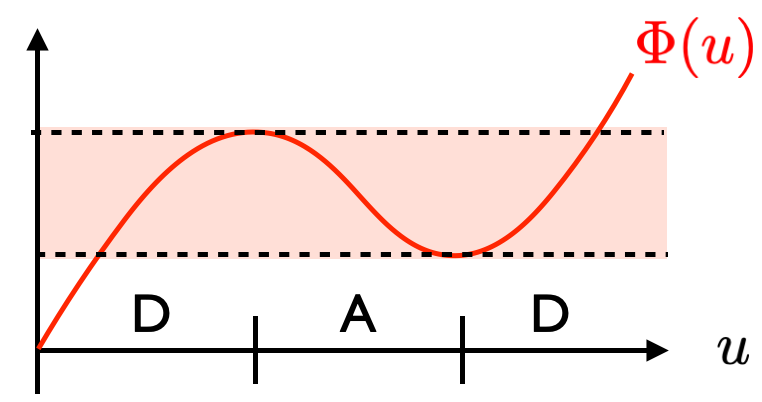}
    \caption{The graph of the potential $\Phi$ and the admissible jump zone (shaded) that allows for possible shock connections. We denote the graph $(u,\Phi(u))$ by $S=S_s^l\cup F_l\cup S_m \cup F_r \cup S_s^r$ which is referred to as a critical manifold; see Sec. \ref{sec:layerproblem} for details. Possible shocks are confined to a fixed potential value $\Phi(u)=const$, i.e.  jumps must occur over the middle branch of the potential, $S_m$, connecting the outer two branches of the potential, $S_s^r$ and $S_s^l$ (shaded region). }
    \label{fig:jump-zone}
\end{figure}

For sparse population density diffusive behaviour is assumed, while for intermediate population density aggregation will happen. For larger population densities (close to the carrying capacity), diffusive behaviour occurs again. 

This DAD model assumption leads to a non-monotone cubic potential $\Phi(u)$ as sketched in Figure~\ref{fig:jump-zone}. It is this non-monotonicity of the potential $\Phi$ which creates a bistability zone of diffusive states which can lead to {\em phase-separation}, i.e., shock formation.

\subsection{RND dynamics and shock formation in travelling wave coordinates}
\label{sec:formalshocks}
Let us look for one of the simplest possible coherent structures in such RND models \eqref{eq:rnd}, travelling waves with wave speed $c\in\mathbb{R}$ that connect the asymptotic end states $u_{-}=1\to  u_{+}=0$ or vice versa, i.e., population/agents invade or evade the unoccupied domain with constant speed. A travelling wave analysis introduces a co-moving frame $z=x-ct$ in \eqref{eq:rnd}, $c\in\mathbb{R}$. Stationary solutions, i.e., $u_t=0$, in this co-moving frame include travelling waves/fronts, and they are found as special (heteroclinic) solutions of the corresponding ODE problem
\begin{equation}
-cu_z-(D(u)u_z)_z = f(u)\,.
\end{equation}
Define the variable $v:=-cu-D(u)u_z\,$ to obtain the corresponding  2D dynamical system
\begin{equation}\label{eq:red-1}
\begin{aligned}
%w_z =
D(u)u_z&= -(v+cu)\\
v_z &= f(u)\,.\\
\end{aligned}
\end{equation}
Note that this dynamical system is singular where $D(u)=0$, i.e., wherever the diffusion-aggregation transition happens. 
To be able to study this problem \eqref{eq:red-1} on the whole domain of interest including these transition zones near $D(u)=0$, we make an auxiliary state-dependent transformation $dz =D(u)d\zeta$ which gives the so-called {\em desingularised problem}
\begin{equation}\label{eq:desing-1}
\begin{aligned}
u_\zeta&= -(v+cu)\\
v_\zeta &= D(u)f(u)\,.
\end{aligned}
\end{equation}
This problem is topologically equivalent to \eqref{eq:red-1} in the diffusion regime $D(u)>0$

while one has to reverse the orientation in the aggregation regime $D(u)<0$

to obtain the equivalent flow.

\begin{remark}
We emphasize that the auxiliary system is only a proxy system to study the problem near $D(u)=0$. To completely understand the original flow near $D(u)=0$, one has to use additional techniques such as the blow-up method; see, e.g., \cite{szmwex2001}.
\end{remark}

\begin{remark} \label{rem:hamiltonian}
The desingularised system \eqref{eq:desing-1} is Hamiltonian when $c = 0$, with generating function
\begin{equation}
    \tilde H(u,v) = -\frac{v^2}{2}-\int D(u)f(u)\,du.
\end{equation} 
In this case, local segments of the (un)stable manifolds of the saddle points at $u=0$ and $u=1$ lie inside contours of $\tilde H(u,v)$.
\end{remark}
The asymptotic end states of the travelling waves form equilibrium states of the desingularised (and the original) problem defined by $f(u_\pm)=0$, and $v_\pm = -c u_\pm$. Our focus is on these asymptotic end states given by the equilibria
\begin{equation}
p_- := (u_-,v_-)=(1,-c)\,, \qquad  p_+: = (u_+,v_+)=(0,0) 
\end{equation}

\begin{comment}
%
\begin{table}[h]
\centering
\begin{tabular}{|c | c | c | c | } 
 \hline
 $D(u)$ & $f(u)$ & $(u_\mp,v_\mp)$ & $(u_\pm,v_\pm)$ \\
 \hline
 DAD & logistic/bistable   & $S_s^r$  & $S_s^l$  \\
%DA & logistic/bistable &$S_m$ & $S_s^+$  \\  
 \hline
\end{tabular}
\caption{Location of asymptotic end states on critical manifold $S$}
\label{table:location-eq1}
\end{table}
%
\end{comment}

\begin{remark}
    In the case of a bistable reaction term \eqref{f:bistable}, there exists an additional equilibrium in the domain of interest defined by $f(u_{b}=\alpha)=0$ which gives
$(u_{b},v_{b})=(\alpha,-c\alpha)$. Its importance/relevance  will be discussed later on.

\end{remark}

In our setup of the RND model, travelling wave solutions connecting $u_-$ and $u_+$ (if they exist) allow for nonsmooth solutions, because the zeroes of the diffusion coefficient $D(u)$ in the relevant domain of interest $u\in [0,1]$ define singularities in this problem.  Discontinuous jumps (shocks) can occur anywhere within the admissible jump zone; see Figure~\ref{fig:jump-zone}. In the absence of an obvious integral conservation law, our approach is to define {\em geometric} criteria for shock selection. \\

One strategy is to formally select a shock height from the admissible jump zone in Figure \ref{fig:jump-zone}, i.e. we specify the endpoints of the shock $u = u_{l}$ and $u = u_{r}$ such that $\Phi(u_{l}) = \Phi(u_{r})$. For system \eqref{eq:desing-1}, let us denote by $W^u(p_-,c)$ the unstable manifold of $p_-$ and by $W^s(p_+,c)$ the stable manifold of $p_+$. Let $v_-(c)$ denote the $v$-coordinate of the first intersection of $W^u(p_-,c)$ with the section $\{(u,v):u=u_{r}\}$, and similarly denote by $v_+(c)$ the $v$-coordinate of the first intersection of $W^s(p_+,c)$ with the section $\{(u,v):u=u_{l}\}$.  We can then attempt to locate a wavespeed $c = c_*$ such that 
\begin{equation}
v_+(c_*) = v_-(c_*).
\end{equation}

If such a wavespeed exists, then we are able to construct a {\it formal} (nonsmooth) shock connecting $u_-$ and $u_+$ according to a given shock selection rule. \\

Shock selection can also be enforced from other geometric conditions. Let us consider a symmetric setup with $\gamma_2 = 1 - \gamma_1$ and $\alpha = 1/2$, i.e., the roots of the diffusion $D(u)$ are placed symmetrically about the midpoint $u=1/2$ in the interval $u \in (0,1)$ and the middle root of the reaction term $f(u)$ is placed exactly in the middle. In view of Remark \ref{rem:hamiltonian} and the symmetry, the (un)stable manifolds of the saddle points at $u_\pm=0,1$ happen to lie on the same contour of $\tilde H(1,0) = \tilde H(0,0)=0$ when $c = 0$, i.e., 

\begin{equation}
\tilde H(1,0)=\int_0^1 D(u)f(u)=0\,.
\end{equation}

We find that there is exactly one pair of $u$-values $u_{l/r}$ at which the (un)stable manifolds of the saddle points in the region $D(u)>0$ can be formally connected by a shock, such that the potential $\Phi(u)$ remains constant; see Figure \ref{fig:symmstanding}. In other words, the correct admissible standing shock has been {\it decided} by the symmetry.

\begin{remark}
    As a consequence of the symmetry, the shock selection values $u=u_{l/r}$ in this case are given by the well-known \textit{equal-area rule}; see Figure~\ref{fig:symmstanding}(b) and the corresponding formula \eqref{def:equal-area}.
\end{remark}
 
Under continuous variation of the shock selection height and the other model parameters, we expect an entire continuum of such shock connections to persist near to this family of symmetric standing shocks, with continuously varying wavespeeds.\footnote{We revisit this assertion in more detail in Sec. \ref{sec:concat}.} Each member of this continuum will constitute a formal shock solution connecting the states $u_l$ and $u_r$ in the desingularised system \eqref{eq:desing-1}.

\begin{figure}[t]
\begin{subfigure}{0.5\textwidth}
\includegraphics[width=8.75cm]{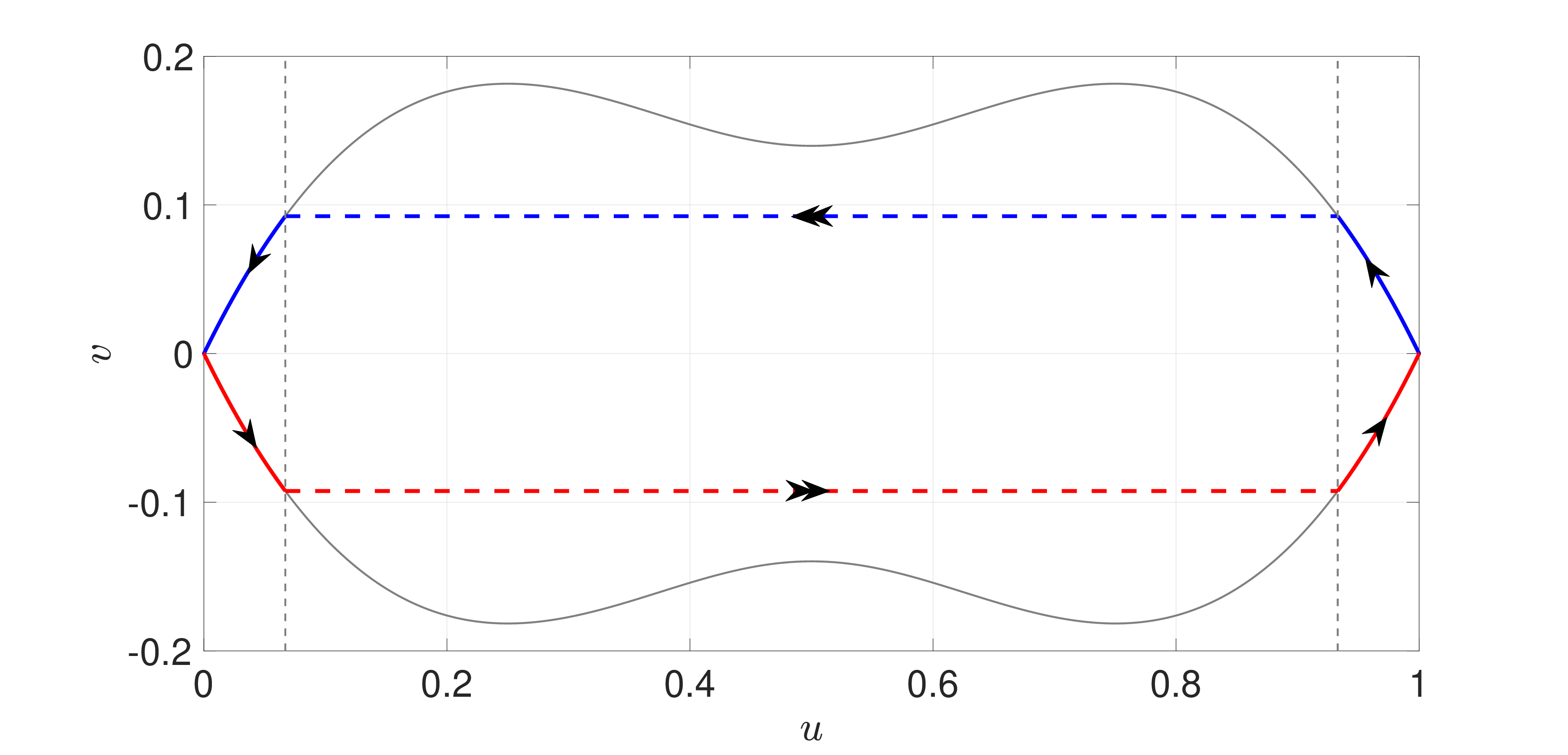} 
\caption{}
\end{subfigure}
\begin{subfigure}{0.5\textwidth}
\includegraphics[width=8.75cm]{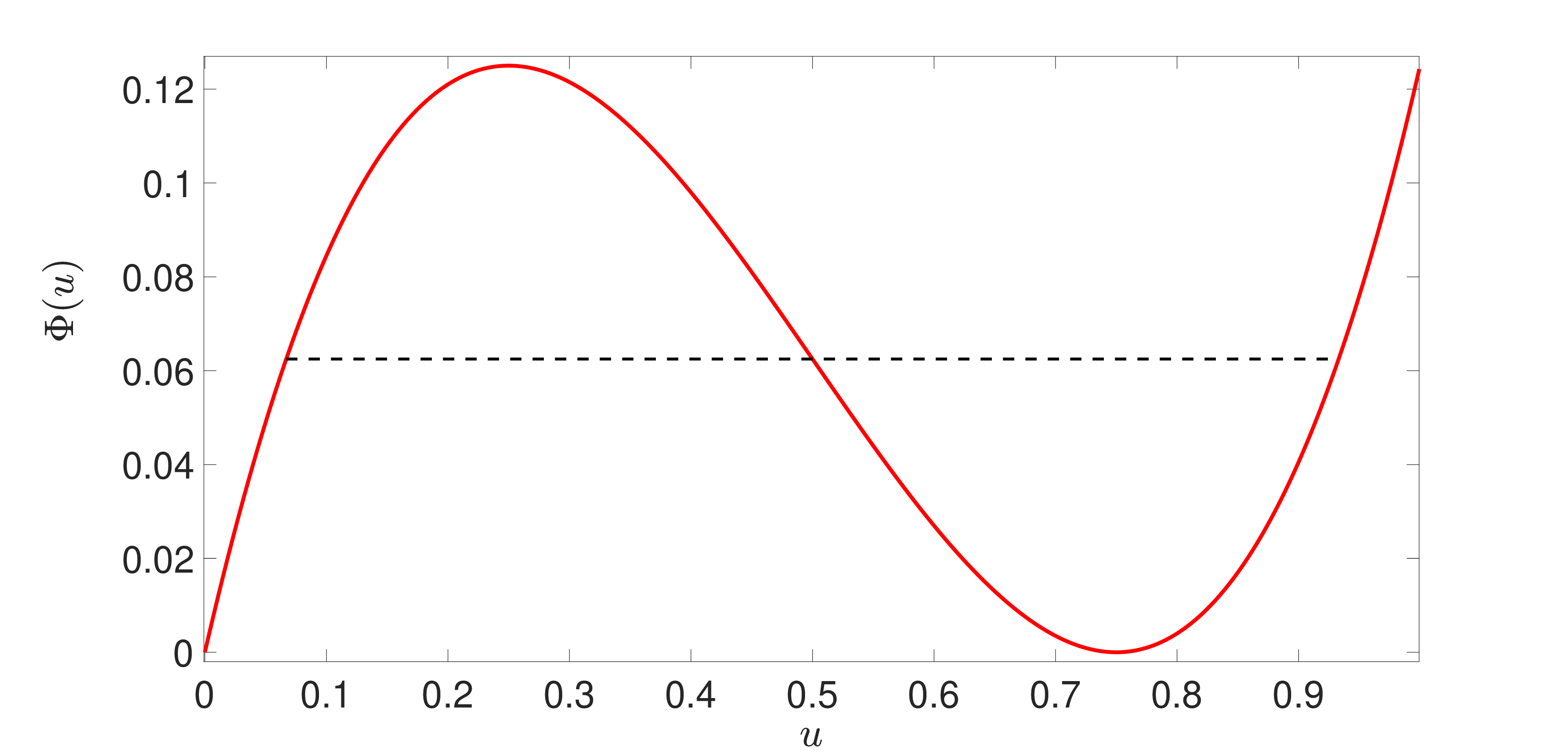}
\caption{}
\end{subfigure}
    \caption{A pair of symmetric standing waves (solid curve segments connected by dashed lines representing shock discontinuities of the solutions, red and blue online) of the desingularised travelling wave equations \eqref{eq:desing-1}, subject to the equal area shock selection rule. Vertical dashed lines (gray online) denote shock selection values $u = u_{l/r}$. Corresponding contour $\tilde H(u,v) = \tilde H(0,0)=\tilde H(1,0)=0$ underlaid (gray solid curve). (b) Corresponding shock selection with $u_l \approx 0.06699$ and $u_r \approx 0.93301$.  Parameter set: $\gamma_1 = 1/4,\,\gamma_2 = 1-\gamma_1 = 3/4,\,\alpha = 1/2,\,\,c=0,\,\kappa=5,\,\beta = 6$.}
    \label{fig:symmstanding}
\end{figure}

\begin{comment}
\begin{figure}
\centering
 (a)  \includegraphics[width=7.5cm]{symmetricstanding2.eps}
 (b)  \includegraphics[width=7.5cm]{equalareasketch.eps}

   % (b) \includegraphics[width=7.5cm]{interp-slowconcat.eps}
    \caption{A pair of symmetric standing waves (solid curve segments connected by dashed lines representing shock discontinuities of the solutions, red and blue online) of the desingularised travelling wave equations \eqref{eq:desing-1}, subject to the equal area shock selection rule. Vertical dashed lines (gray online) denote shock selection values $u = u_{l/r}$. Corresponding contour $\tilde H(u,v) = \tilde H(0,0)=\tilde H(1,0)=0$ underlaid (gray solid curve). (b) Corresponding shock selection with $u_l \approx 0.06699$ and $u_r \approx 0.93301$.  Parameter set: $\gamma_1 = 1/4,\,\gamma_2 = 1-\gamma_1 = 3/4,\,\alpha = 1/2,\,\,c=0,\,\kappa=5,\,\beta = 6$.}
    \label{fig:symmstanding}
\end{figure}
\end{comment}

\subsection{Regularisations of RND models} \label{sec:regularisations}

 Following the geometric approach in the previous section, our goal is to describe  conditions under which particular shock criteria are {\it uniquely} selected from within the admissible jump zone. The related issues of nonuniqueness and lack of smoothness suggest that we should consider a `nearby' system in which locally unique, smoothed shock-fronted solutions are available. Our approach is to add small perturbative high-order regularisation terms. \\

Regularisation of RND models is typically considered in one of two ways  \cite{padron03, pego89}. 
The first method of regularisation accounts for \emph{viscous relaxation} by adding a small temporal change in the diffusivity: 

\begin{equation}\label{eq:varnad}
 u _t  =  ( \Phi(u) + \varepsilon u_{t})_{xx}  +f(u), \qquad 0 \leq \varepsilon \ll 1.
\end{equation}
This is usually referred to as a Sobolev regularisation.
The second of these involves adding a small change in the potential to account for interfacial effects, leading to: 
\begin{equation}\label{eq:charnad}
 u_t  =  ( \Phi(u) - \varepsilon^2 u_{xx} )_{xx}  +f(u), \qquad 0 \leq \varepsilon \ll 1.
\end{equation}
\noindent
Both regularisation techniques can be viewed as higher order viscous regularisations. They have been widely employed in models of chemical phase-separation, though they have gone relatively unnoticed in biological models until very recently.\\

Here, we study the possible effects of \emph{both} regularisations in a single RND model, i.e.,
\begin{equation}\label{eq:vcharnad}
 u _t  =  ( \Phi(u) + \varepsilon a u_{t} - \varepsilon^2 u_{xx} )_{xx}  +f(u), 
 \qquad 0 \leq \varepsilon \ll 1, a\ge 0.
\end{equation}

Since we only consider small perturbative regularisations $0<\varepsilon\ll 1$, these models are so-called {\em singularly perturbed systems} and, as a consequence, the powerful machinery of geometric singular perturbation theory (GSPT) is applicable \cite{fenichel79, jonesgsp95, wechpet10}, as we shall explain.

\begin{comment}
\begin{remark}

This regularised RND model \eqref{eq:vcharnad} can be derived from the history dependent energy functional
$$
 E(u) =\int_\Omega \left( F(u) +\varepsilon a \int_0^t u_s^2 ds + \frac{\varepsilon^2}{2} |u_x|^2\right)dx\,,
$$
where $F(u)=\int \Phi(u) du$ is the free energy density function of the homogeneous state. The interfacial energy, $\frac{\varepsilon^2}{2} |u_x|^2$, introduces smoothing effects in regions with large gradients, and so does the memory term, $\varepsilon a\int_0^t u_s^2 ds$,  which can be interpreted as visco-elastic potential energy; see, e.g., \cite{Witelski96}.
\end{remark}
\end{comment}

\begin{remark}
Continuum macroscale models can also be derived from lattice-based microscale models; see \cite{Johnston2017} for leading order RND models and \cite{Anguige_2008} for (more complicated) regularised RND models.
\end{remark}

\section{The GSPT setup for the regularised RND model \eqref{eq:vcharnad}} \label{sec:dynamics}

We derive conditions based on the specific functions $D(u)$ and $f(u)$ that lead to travelling waves with sharp interfaces (shocks) in one spatial dimension.
% and to show that physically/biologically relevant functions such as those shown in Table 2 fulfill these requirements.
We introduce a travelling wave coordinate $z=x-ct$ for waves with speed $c\in\mathbb{R}$ and ask for stationary states of the PDE in the co-moving frame. This transforms the regularised RND model \eqref{eq:vcharnad} into a fourth order ordinary differential equation
\begin{equation}\label{eq:rnd-tw-ch}
-c u_z = \Phi(u)_{zz} -\varepsilon a c u_{zzz}-\varepsilon^2  u_{zzzz} +f(u)\,,
\end{equation}
which we can recast as a \emph{singularly perturbed dynamical system} in standard form
\begin{equation}\label{eq:rnd-sys-1-slow}
\begin{aligned}
\varepsilon u_{z} &= \hat{u}\\
\varepsilon \hat{u}_{z} &=w+\Phi(u)-\delta \hat{u}\\ 
v_z &= f(u)\\
w_z &= v+cu\,.
\end{aligned}
\end{equation}
where $(u,\hat{u})\in\mathbb{R}^2$ are `fast' variables, $(v,w)\in\mathbb{R}^2$ are `slow' variables, $\varepsilon\ll 1$ is the singular perturbation parameter, and $\delta:=ac$ is an additional lumped system parameter incorporating the wave speed $c$ and the relative contribution of the viscous relaxation $a$.\\

\begin{comment}
\WM{Ian: do you prefer to rewrite the dynamical system as follows (pls check)?
$$
\begin{aligned}
\varepsilon u_{z} &= \hat{u}\\
\varepsilon \hat{u}_{z} &= - w+\Phi(u)-\delta \hat{u}\\ 
v_z &= f(u)\\
w_z &= -(v+cu)\,.
\end{aligned}
$$
}

\IL{Please don't do this to me :) All the numerics are with the original system. PS your notation is consistent with the earlier Yifei paper so it's fine.}

\WM{Good. Saves me work/time as well!-) So, dismissed.}
\end{comment}
%%%%%%%%%
\begin{comment}
rewrite as
\begin{equation}
\{\varepsilon^2  u_{zzz}+\varepsilon a c u_{zz}-cu-\Phi(u)_z\}_z = f(u)\,.
\end{equation}
%
We define
\begin{equation}
v:=\varepsilon^2  u_{zzz}+\varepsilon a c u_{zz}-cu-\Phi(u)_z\,,
\end{equation}
which gives
\begin{equation}
\begin{aligned}
\{\varepsilon^2  u_{zz}+\varepsilon a c u_{z}-\Phi(u)\}_z &=v+cu\\
v_z &= f(u)\,.
\end{aligned}
\end{equation}
Next, we define
\begin{equation}
w:=\varepsilon^2  u_{zz}+\varepsilon a c u_{z}-\Phi(u)\,,
\end{equation}
which gives
\begin{equation}
\begin{aligned}
%\varepsilon\{\varepsilon  u_{z}+ a c u\}_z &=w+\Phi(u)\\
\varepsilon^2  u_{zz}+ \varepsilon a c u_z &=w+\Phi(u)\\
v_z &= f(u)\\
w_z &= v+cu\,.
\end{aligned}
\end{equation}
Finally, we set %$\hat{u} := \varepsilon  u_{z}$ and $\delta:= a c$
\begin{equation}
\hat{u} := \varepsilon  u_{z}\,,\quad \delta:= a c\,,
\end{equation}
to obtain 
\end{comment}
%%%%%%%%%%%%%

Rescaling the `slow' independent travelling wave variable $dz=\varepsilon dy$ in \eqref{eq:rnd-sys-1-slow} gives the equivalent fast system
\begin{equation}\label{eq:rnd-sys-1-fast}
\begin{aligned}
u_{y} &= \hat{u}\\
\hat{u}_{y} &=w+\Phi(u)-\delta \hat u\\ 
v_y &= \varepsilon f(u)\\
w_y &= \varepsilon (v+cu)\,,
\end{aligned}
\end{equation}

with the `fast'  independent travelling wave variable $y$.

These equivalent dynamical systems \eqref{eq:rnd-sys-1-slow} respectively \eqref{eq:rnd-sys-1-fast} have a symmetry
\begin{equation}
(\hat{u},v,c,y) \leftrightarrow (-\hat{u},-v,-c,-y),\quad\mbox{respectively}\quad
(\hat{u},v,c,z) \leftrightarrow (-\hat{u},-v,-c,-z)\,.
\end{equation}

We will focus on heteroclinic connections made between the two equilibria
\begin{equation}
p_- := (u_-,\hat{u}_-,v_-,w_-)=(1,0,-c,-\Phi(1))\,, \qquad  p_+: = (u_+,\hat{u}_+,v_+,w_+)=(0,0,0,-\Phi(0))
\end{equation}
corresponding to steady-states of the density variable $u$ at $u = 1$ and $u=0$; note the slight abuse in notation in using $p_{\pm}$ to refer to the corresponding equilibria of both the 2D desingularised problem \eqref{eq:desing-1} and the (regularised) 4D system \eqref{eq:rnd-sys-1-fast}. An eigenvalue calculation using the linearisation of \eqref{eq:rnd-sys-1-fast} determines that both equilibria are (hyperbolic) saddle points having two-dimensional stable and unstable manifolds for each $\varepsilon>0$. \\

The aim is to use methods from GSPT to analyse the travelling wave problem in its `slow' respectively `fast' singular limit system, i.e., $\varepsilon\to 0$ in \eqref{eq:rnd-sys-1-slow} respectively \eqref{eq:rnd-sys-1-fast}, and to infer results on the existence (and stability) of shock-fronted travelling waves in the full regularised RND problem for $\varepsilon\neq 0$.

\begin{comment}
Based on our model assumption on $D(u)$ in \eqref{eq:phi1} we calculate an anti-derivative
$$
\Phi(u) = \beta u\left(\frac {{u}^{2}}{3} - \frac { \left( \gamma_1 +\gamma_2 \right) u}{2}
+\gamma_1\gamma_2\right)\,.
$$

\WM{We have two free parameters in this problem, $a$ and $c$, which are related through $$\delta(a,c)=ac\,.$$ 
We need to distinguish small limits of these parameters, i.e.,
\begin{equation*}
\lim_{c\to 0} \delta(a,c)=
\left\{
\begin{aligned}
0 & \\
\hat\delta>0
\end{aligned}
\right.
%\end{equation*}
\qquad\mbox{and}\qquad
\lim_{a\to 0} \delta(a,c)=
\left\{
\begin{aligned}
0 & \\
\tilde\delta>0
\end{aligned}
\right.
\end{equation*}
Then, Figure 12 makes sense, I guess!
}
\end{comment}

%%%%%%%%%%%%%%%%%%%%%%%%%%%%%%%
\subsection{The limit on the fast scale - the layer problem} \label{sec:layerproblem}

We begin with the `fast' system \eqref{eq:rnd-sys-1-fast}.
Here the limit $\varepsilon\to 0$ gives the {\em layer problem}

\begin{equation}\label{eq:rnd-layer-1}
\begin{aligned}
u_{y} &= \hat{u}\\
\hat{u}_{y} &=w+\Phi(u)-\delta \hat u\\ 
v_y =w_y &=0\,,
\end{aligned}
\end{equation}
i.e., $(v,w)$ are considered parameters. 
%or, in other words, the quantities $w=-\Phi(u)$ and $v=-(cu+\Phi(u)_z)$ are conserved in the layer problem. 
Hence, the flow is along two-dimensional fast fibers 
${\cal L} := \{(u,\hat{u},v,w)\in\mathbb{R}^4: (v,w)=const\}$. 
The set of equilibria of the layer problem,  
\begin{equation}\label{eq:critical-2c}
S:=\{(u,\hat{u},v,w)\in\mathbb{R}^4: \hat{u}= \hat{u}(u,v)=0,\, w=w(u,v)=-\Phi(u)\}\,,
\end{equation}
forms the two-dimensional {\em critical manifold} of the problem which is a graph over $(u,v)$-space.
In the assumed diffusion-aggregation-diffusion (DAD) setup \eqref{eq:phi1}, we have 
a sign change in the diffusivity along the set
\begin{equation}
F:=\{(u,\hat{u},v,w)\in S: D(u)=0\}\,,
\end{equation}
where
$F=F_l\cup F_r =\{(u,\hat{u},v,w)\in S: u=\gamma_1\}\cup\{(u,\hat{u},v,w)\in S: u=\gamma_2\}$ consists of two disjoint one-dimensional lines. Thus we have a splitting of the critical manifold $S=S_s^l\cup F_l\cup S_m \cup F_r \cup S_s^r$
where 
\begin{equation}
\begin{aligned}
    S_s^{l}:= & \{(u,\hat{u},v,w)\in\mathbb{R}^4: \hat{u}= \hat{u}(u,v)=0,\,w=w(u,v)=-\Phi(u),\, u<\gamma_1\}\\
    S_s^{r}:= &\{(u,\hat{u},v,w)\in\mathbb{R}^4: \hat{u}= \hat{u}(u,v)=0,\,w=w(u,v)=-\Phi(u),\, u>\gamma_2\}\\
    S_m:= & \{(u,\hat{u},v,w)\in\mathbb{R}^4: \hat{u}= \hat{u}(u,v)=0,\,w=w(u,v)=-\Phi(u),\, \gamma_1<u<\gamma_2\}\,,
\end{aligned}
\end{equation}
see Figure~\ref{fig:manifold}.
%
%\begin{comment}
\begin{figure}
    \centering
    \includegraphics[width=10cm]{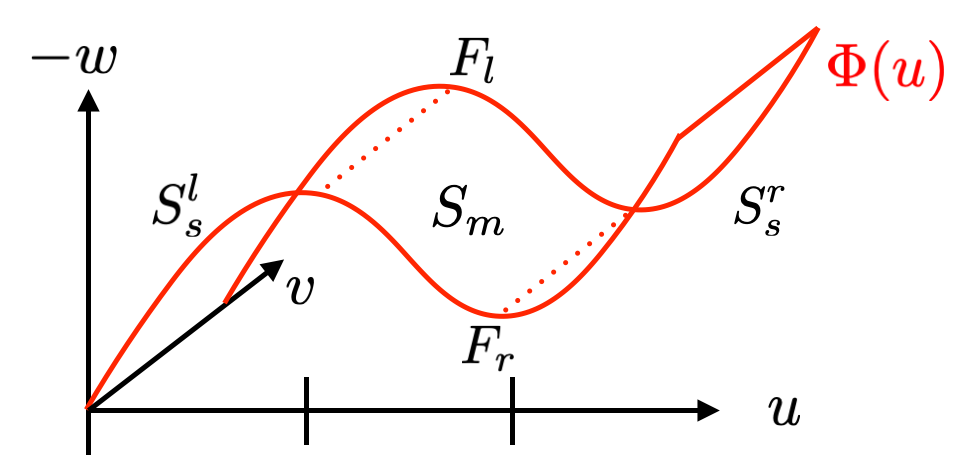}
    \caption{sketch of the two-dimensional critical manifold $S$ projected onto $(u,v,w)$-space.}
    \label{fig:manifold}
\end{figure}
%
%\end{comment}
%
The stability property of this set of equilibria $S$ is determined by the two non-trivial eigenvalues of the layer problem, i.e., the eigenvalues of the  Jacobian evaluated along $S$,
\begin{equation}
J=
\begin{pmatrix}
0 & 1\\
D(u) & -\delta
\end{pmatrix}
\,.
\end{equation}
This matrix has $\mbox{tr}\, J=-\delta$ and $\det J=-D(u)$. Hence, for $D(u)>0$ the outer branches $S_s^{l/r}$ are normally-hyperbolic and of saddle-type (S-type). For $\delta\neq 0$ and $D(u)<0$ the middle-branch $S_m$ is also normally-hyperbolic, focus/node-type (FN-type), while for $\delta= 0$ and $D(u)<0$ the middle-branch $S_m$ loses normal-hyperbolicity and is of centre-type (C-type).

Loss of normal hyperbolicity happens also along the set $F=F_l\cup F_r$ where $\det J=0$ independent of $\delta$.

%\newpage
\subsection{The limit on the slow scale - the reduced problem}

For the slow system \eqref{eq:rnd-sys-1-slow}, the limit $\varepsilon\to 0$ gives the {\em reduced problem}

\begin{equation}\label{eq:rnd-reduced-1}
\begin{aligned}
0 &= \hat{u}\\
0 &=w+\Phi(u)-\delta \hat u\\ 
v_z &= f(u)\\
w_z &= v+cu\,.
\end{aligned}
\end{equation}

It describes the `evolution' of the slow variables $(v,w)$ constrained to the 2D critical manifold $S$ \eqref{eq:critical-2c}
%Here, the critical manifold $S$ 
which is given as a graph over the $(u,v)$-coordinate chart. We denote the corresponding embedding 
$\psi: \mathbb{R}^2\to\mathbb{R}^4$, i.e., $S=\psi(u,v)$.
Therefore, we aim to study the corresponding reduced flow on $S$ in this $(u, v)$-coordinate chart. By definition, the main requirement on the reduced vector field $R(u,v)\in T\mathbb{R}^2$ is that, when mapped onto the tangent bundle $TS$ via the linear transformation $D\psi$ it has to correspond to the (leading order) slow component of the full four-dimensional vector field constraint to $TS$, i.e., 

\begin{equation}
D\psi(u,v) R(u,v)= \Pi^S G(\psi(u,v))=\left(\frac{v+cu}{-D(u)},0,f(u),v+cu\right)^\top
\end{equation}
where $\Pi^S G(\psi(u,v))$ is the projection\footnote{In general, such a projection operator $\Pi^S$ is oblique; see, e.g., \cite{wex20}. Here it is orthogonal due to the standard form of system \eqref{eq:rnd-sys-1-slow}.} of the vector field $G=(0,0,f(u),v+cu)^\top$ onto the tangent bundle $TS$ of the critical manifold $S$ along fast fibres $\mathcal L$ spanned by $\{(1,0,0,0)^\top,(0,1,0,0)^\top\}$.

Thus the reduced vector field $R(u,v)$ in the $(u,v)$-coordinate chart is given by the right-hand side of \eqref{eq:red-1}.

\begin{remark}
  The reduced problem is independent of the two choices of regularisation as expected, since it represents the TWP of the original RND model. With the geometric approach, we have the extra information that this reduced flow is constrained to the critical manifold $S$ embedded in the full four-dimensional phase-space.
\end{remark}

We classify all singularities of \eqref{eq:red-1} by analysing the auxiliary system, i.e., the desingularised problem \eqref{eq:desing-1}.
The equilibria of the reduced problem \eqref{eq:red-1} respectively desingularised problem \eqref{eq:desing-1} and their stability properties are summarised in Table~\ref{table:type-eq}. Additionally we know that the two asymptotic end-states $(u_\pm,v_\pm)$ are located on opposite outer branches $S_s^{l/r}$ of the critical manifold while the location of the additional equilibrium state $(u_b,v_b)$ varies under the variation of the system parameters. 

\begin{comment}
%    
This extra equilibrium is located as follows:
%
%\begin{comment}
\begin{table}[h]
\centering
\begin{tabular}{|c | c | c | c | c |} 
 \hline
 $D(u)$ & $f(u)$ & $\alpha>\gamma_2$ & $\gamma_1<\alpha<\gamma_2$  &  $\alpha<\gamma_1$ \\
 \hline
DAD & bistable  & $S_s^r$  & $S_m$ & $S_s^l$   \\  
%DA & bistable & - & $S_m$ & $S_s^+$ \\
 \hline
\end{tabular}
\caption{Location of additional equilibrium  $(u_{b},v_{b})$ on critical manifold $S$ in the bistable case}
\label{table:location-eq2}
\end{table}
%
\end{comment}
%

The Jacobian evaluated at any of these equilibria $(u_{\pm,b},v_{\pm,b})$ is given by 
\begin{equation}
J=
\begin{pmatrix}
-c & -1 \\
D(u_{\pm,b})f'(u_{\pm,b}) & 0
\end{pmatrix}
\end{equation}
which has $\mbox{tr}\,J=-c$ and $\det J= D(u_{\pm,b})f'(u_{\pm,b})$. 
The types of equilibria are summarised in Table~\ref{table:type-eq}.
The distinction between Node and Focus (NF)  depends on the sign of the discriminant 
$\mathcal{D}:=c^2- 4 D(u_{\pm,b})f'(u_{\pm,b})$, 
$\mathcal{D}>0$ (Node) or $\mathcal{D}<0$ (Focus). The distinction between stable $c>0$ and unstable $c<0$ depends on the sign of the wave speed.\footnote{In case of a standing wave $c=0$, any NF becomes a Centre.}
\begin{table}[h]
\centering
\begin{tabular}{|c | c | c | c | c | c | c |} 
 \hline
 $D(u)$ & $f(u)$ & $(u_-,v_-)$ & $(u_+,v_+)$ & $(u_{b},v_{b})$,  & $(u_{b},v_{b})$,  &  $(u_{b},v_{b})$, \\
  & &  & &  $\alpha<\gamma_1$ & $\gamma_1<\alpha<\gamma_2$  &  $\gamma_2<\alpha$ \\
 \hline
%DAD & logistic  & Saddle & (un)stable NF  & -  & - & -   \\  
%DA & logistic & unstable NF & stable NF  & - & - & - \\
DAD & bistable  & Saddle & Saddle  &  (un)stable NF  & Saddle  & (un)stable NF    \\  
%DA & bistable & unstable NF  & Saddle  & - & Saddle  & stable NF  \\
 \hline
\end{tabular}
\caption{ Type of equilibria on critical manifold $S$.}
\label{table:type-eq}
\end{table}
%

\begin{comment}
\begin{table}[h]
\centering
\begin{tabular}{|c | c | c | c | c  |} 
 \hline
 $D(u)$ & $f(u)$ & $(u_-,v_-)$ & $(u_+,v_+)$  & $(u_{b},v_{b})$, $\gamma_1<\alpha<\gamma_2$     \\
 \hline
DAD & logistic  & Saddle & stable NF  & -     \\  
%DA & logistic & unstable NF & stable NF  & - & - & - \\
DAD & bistable  & Saddle & Saddle   & Saddle     \\  
%DA & bistable & unstable NF  & Saddle  & - & Saddle  & stable NF  \\
 \hline
\end{tabular}
\caption{ Type of equilibria on critical manifold $S$}
\label{table:type-eq}
\end{table}
\end{comment}

%%%
\begin{comment}
DAD & logistic  & Saddle $S_r^-$ & stable NF $S_r^+$ & -  & - & -   \\  
DA & logistic & unstable NF $S_a$ & stable NF $S_r^+$ & - & - & - \\
DAD & bistable  & Saddle $S_r^-$ & Saddle $S_r^+$ &  stable NF $S_r^-$ & Saddle $S_a$ & stable NF $S_r^+$   \\  
DA & bistable & unstable NF $S_a$ & Saddle $S_r^+$ & - & Saddle $S_a$ & stable NF $S_r^+$ \\
\end{comment}
%%%

\begin{comment}
\begin{remark}
The desingularised system \eqref{eq:desing-1} defines another type of singularities for the original problem through $D(u)=0$ which are known as {\em folded singularities}; see, e.g., \cite{wechpet10}. 
These will be discussed in more detail in section~\ref{sec:fs}.
%If time permits, I will briefly discuss this in my presentation.
\end{remark}
\end{comment}

\subsection{Folded singularities}\label{sec:fs}

The desingularised system \eqref{eq:desing-1} defines another type of singularity for the reduced problem through $D(u)=0$ which exists on the fold lines $F_{l/r}$ and are known as {\em folded singularities}. In our problem, these folded singularities are given by $v_{f_{l/r}}=-c u_{f_{l/r}}$ where $u_{f_{l/r}}=\gamma_{1/2}$\,, i.e.,
\begin{equation}
(u_{f_{l/r}},u_{f_{l/r}})=(\gamma_{1/2},-c\, \gamma_{1/2})\,.
\end{equation}

The Jacobian of the desingularised problem evaluated at such a folded singularity %$(u_{f_{l/r}},v_{f_{l/r}})$ 
is given by 
\begin{equation}
J=
\begin{pmatrix}
-c & -1 \\
D'(u_{f_{l/r}})f(u_{f_{l/r}}) & 0
\end{pmatrix}
\end{equation}
which has $\mbox{tr }J=-c$, $\det J= D'(u_{f_{l/r}})f(u_{f_{l/r}})$ and $\mathcal{D}=c^2- 4 D'(u_{f_{l/r}})f(u_{f_{l/r}})$.
Hence we have %dealing 
for $\det J<0$ a folded saddle (FS), and for $\det J>0$ a folded node (FN) or a folded focus (FF) depending on the discriminant $\mathcal{D}$ being positive or negative.

\begin{comment}
 %   
\begin{table}[h]
\centering
\begin{tabular}{|c | c | c | c | } 
 \hline
 $\phi(u)$ & $f(u)$ & $(u_{f-},v_{f-})$ & $(u_{f+},v_{f+})$ \\
 \hline
 DAD & logistic   & 'stable' FN/FF   & FS  \\
%DA & logistic & - & FS  \\  
 \hline
\end{tabular}
\caption{Type of folded singularities on critical manifold $S$, logistic case}
\label{table:type-fs-log}
\end{table}
%
\end{comment}

\begin{table}[h]
\centering
\begin{tabular}{|c | c | c | c | c | c |} 
 \hline
 $\phi(u)$ & $f(u)$ & folded singularity & $\alpha<\gamma_1$ & $\gamma_1<\alpha<\gamma_2$  &  $\gamma_2<\alpha$ \\
 \hline
DAD & bistable & $(u_{f_{l}},v_{f_{l}})$  &  FS &  `stable' FN/FF &  `stable' FN/FF   \\  
DAD & bistable & $(u_{f_{r}},v_{f_{r}})$  &   `stable' FN/FF &  `stable' FN/FF & FS   \\  
%DA & bistable &  $(u_{f+},v_{f+})$  & - &  'stable' FN/FF & FS \\
 \hline
\end{tabular}
\caption{Type of folded singularities on fold lines $F_l \cup F_r=F$.}%critical manifold $S$.}
\label{table:type-fs-bi}
\end{table}

%\end{comment}

\begin{remark}
    The change in type of a folded singularity (FS/FN/FF) under parameter variation coincides with the crossing of the additional equilibrium $(u_b,v_b)$ through the corresponding folded singularity (see Table~\ref{table:type-eq}). This codimension-one phenomenon is known in the GSPT literature as a folded saddle-node type II (FSN II); see. e.g., \cite{szmwex2001}.

    The term `stability' indicates only stability properties for the auxiliary system, i.e., the desingularised problem \eqref{eq:desing-1}. Folded singularities have no associated stability property since corresponding special solutions known as {\em canards} pass through them in finite time, i.e., these canards represent transient phenomena.
\end{remark}

\subsection{Singular heteroclinic orbits} \label{sec:heteroclinics}

The shock-fronted travelling waves that we seek are found as heteroclinic orbits of the four-dimensional dynamical system \eqref{eq:rnd-sys-1-slow} connecting the saddle equilibrium end states $(u_\pm,0,v_\pm,w_\pm)\to (u_\mp,0, v_\mp,w_\mp)$, i.e., we are seeking evasion fronts $u_-=0\to u_+=1$ or invasion fronts $u_-=1\to u_+=0$ of the original travelling front problem.
%
%at $u=0$ to the one at $u=1$. 
From the point-of-view of GSPT, a key observation is that solutions of the (fast) layer problem \eqref{eq:rnd-layer-1} and the (slow) reduced problem \eqref{eq:red-1} can be concatenated to form singular heteroclinic orbits,
\begin{equation}\label{def:SingHet}
    \Gamma_{het}^\pm = \Gamma_{l/r} \cup \Gamma_\pm \cup \Gamma_{r/l}
\end{equation}
where $\Gamma_{het}^+$ denotes a singular heteroclinic evasion front connecting $u_-=0\to u_+=1$, i.e., $\Gamma_l$ is a slow segment of the unstable manifold of the saddle equilibrium at $u_-=0$ connecting to $u_l$, $\Gamma_+$ is a fast jump connecting $u_l \to u_r$, and  $\Gamma_r$ is a slow segment of the stable manifold of the saddle equilibrium at $u_+=1$ connecting to $u_r$, while $\Gamma_{het}^-$ denotes a singular heteroclinic invasion front connecting $u_-=1\to u_+=0$, i.e., $\Gamma_r$ is a slow segment of the unstable manifold of the saddle equilibrium at $u_-=1$ connecting to $u_r$, $\Gamma_-$ is a fast jump connecting $u_r \to u_l$, and  $\Gamma_l$,  is a segment of the stable manifold of the saddle equilibrium at $u_+=0$ connecting to $u_l$.\\

In Section \ref{sec:singhetconstruction} we will construct such singular heteroclinic orbits, and then describe how heteroclinic orbits of the regularised system \eqref{eq:rnd-sys-1-slow} arise as a codimension-one family of perturbations of these singular connections for $0 < \varepsilon \ll 1$. See \cite{li2021} for details of this construction in the setting of both `pure' viscous relaxation and `pure' Cahn-Hilliard-type regularisation \eqref{eq:varnad}, respectively \eqref{eq:charnad}. Our objective is to synthesize and greatly extend this previous GSPT analysis to the more general PDE model with composite regularisation \eqref{eq:vcharnad}. Another new feature of this extension is a description of how the family of monotone waves terminates via global singular bifurcations. %\eqref{eq:rnd}.

\section{Shock selection rules} \label{sec:singhetconstruction}

%\footnote{
The asymptotic end states $(u_\pm,0,v_\pm,w_\pm)$ of the 4D composite regularised RND problem \eqref{eq:rnd-sys-1-slow}
are located on saddle branches of the critical manifold (2D layer problem) and the equilibria are also saddles for the 2D reduced problem. This hyperbolic structure persists for the full 4D model, i.e., the fixed asymptotic end-states are saddles with 2 stable and 2 unstable directions.\\

We extend formally the regularised RND problem \eqref{eq:rnd-sys-1-slow} by the dummy wavespeed equation $c_y=0$ and seek 1D intersections of the corresponding 3D centre-stable and the 3D centre-unstable manifold of the two asymptotic end states $(u_\pm,0,v_\pm,w_\pm,c)$ in 5D phase space, which is generic. In particular, we identify how the composite regularisation parameter $a$ picks a unique shock location and wave speed $c$. 
%} 

\subsection{Singular fast fronts $\Gamma_\pm$ and generalised shock selection rules}
\label{sec:layer-het}

Recall that we introduced the lumped system parameter $\delta=ac$ in the layer problem \eqref{eq:rnd-layer-1} which takes the composite regularisation parameter $a$ and the wave speed $c$ into account. 

\subsubsection{The $\delta=0$ case} \label{sec:delta0layer}
In this case, the layer problem \eqref{eq:rnd-layer-1},
\begin{equation}\label{eq:rnd-layer-0}
\begin{aligned}
u_{y} &= \hat{u}\\
\hat{u}_{y} &=w+\Phi(u)
\end{aligned}
\end{equation}
is Hamiltonian with 
\begin{equation}\label{eq:hamiltonian}
H(u,\hat{u})=\frac{\hat{u}^2}{2}-\int (w+\Phi(u))du\,.
\end{equation}

Trajectories of this layer problem are confined to level sets of the Hamiltonian \eqref{eq:hamiltonian}, i.e., $H(u,\hat{u})=k$. Possible  trajectories that are able to connect equilibrium points on different branches of the critical manifold $S$ are confined to the saddle branches $S_s^{l/r}$ including the boundaries $F_{l/r}$. 

The corresponding equilibrium points $p_{l/r}=(u_{l/r},0,v_{l/r},-\Phi(u_{l/r}))\in S_s^{l/r} \cup F_{l/r}$ of such connections must fulfill $v_{l}=v_{r}$ and $\Phi(u_{l})=\Phi(u_{r})$
since $v$ and $w$ are constant. 

\begin{remark}\label{rem:w}
This creates a bound on possible $w$-values, $w\in [-\Phi(u_{f-}),-\Phi(u_{f+})]$ where $D(u_{f\mp})=0$, i.e., confined to the region between the local extrema of $\Phi$; see Figure~\ref{fig:jump-zone}.
\end{remark}

Without loss of generality, set $H(u_{l},\hat{u}=0)=0$, i.e., 
$H(u,\hat{u})=\frac{\hat{u}^2}{2}-\int_{u_{l}}^u (w+\Phi(u))du$.
Then $H(u_{r},\hat{u}=0)$ must be equal zero as well
for the existence of a layer connection between these two points. 
This constraint leads to the well-known {\em `equal area rule'} (see, e.g. \cite{pego89}),
\begin{equation}\label{def:equal-area}
\boxed{
    \int_{u_{l}}^{u_{r}} (w_h+\Phi(u))du=0}\,.   
\end{equation}
This rule allows for $S_s^{l/r}$ to $S_s^{r/l}$ connections, but not to the boundaries $F_{l/r}$ or the centre-type middle branch $S_m$. Due to the symmetry $(\hat{u},y)\leftrightarrow (-\hat{u},-y)$ in \eqref{eq:rnd-layer-0}, there exists automatically a pair of such heteroclinic connections for fixed $w=w_h$, i.e., $\Gamma_+(w_h,0):p_{l}\to p_{r}$ and $\Gamma_-(w_h,0):p_{r}\to p_{l}$; see Figure~\ref{fig:het_delta_zero}.

%\begin{comment}
\begin{figure}[t]
    \centering

    \includegraphics[width=9.0cm]{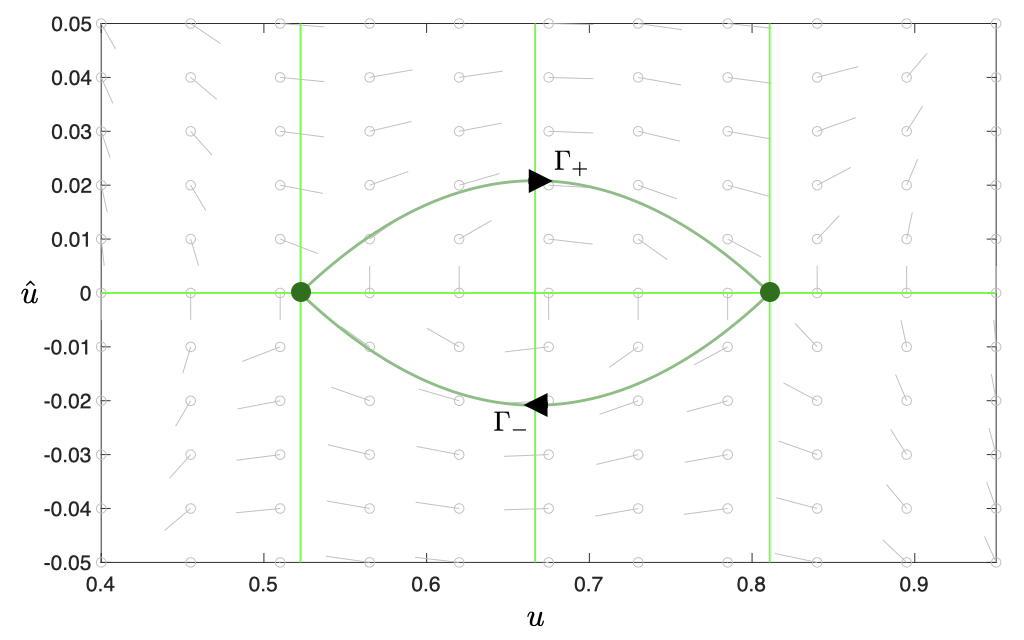}
    \caption{two heteroclinics $\Gamma_+:p_l\to p_r$ and $\Gamma_-:p_r\to p_l$ for $\delta=0$ and $w=w_h\approx -0.5648$ in $(u,\hat u)$-space; other parameter values: $\beta=6$, $\gamma_1=7/12$, $\gamma_2=3/4$\,.} 

    \label{fig:het_delta_zero}
\end{figure}
%\end{comment}

\begin{remark} \label{rem:standingwaves}
The equal area rule \eqref{def:equal-area} determines the value $w=w_h$ for which this integral vanishes. Since $\delta=ac=0$ there are two possible cases:
for $a=0$, it is independent of the possible wave speed $c\in \mathbb{R}$.
On the other hand, for $c=0$ it is independent of the viscous relaxation regularisation contribution $a\in \mathbb{R}$. Hence, the only shock-fronted standing waves that our regularised model can produce are those satisfying the equal area rule.
\end{remark}

\subsubsection{The small $|\delta|$ case}

%\WM{update/ work in progress}

Here we apply {\it Melnikov theory} (see, e.g., \cite{Vanderbauwhede_1992,Wechselberger_2002}) to establish heteroclinic connections in the layer problem \eqref{eq:rnd-layer-1} for sufficiently small $|\delta|>0$. 
%This complements the analysis in section~\ref{sec:layer-het}. 

\begin{comment}
\WM{moved from section 4}
For sufficiently small $|\delta|>0$, we show that nearby heteroclinic connections to the same asymptotic end states still exist. This is done via a {\it Melnikov-type} argument; see, e.g., \cite{Vanderbauwhede_1992,Wechselberger_2002}:\\

\begin{itemize}
    \item We measure the distance $\Delta$ between stable and unstable manifolds emanating from the saddle-equilibria in  a cross section $\Sigma$ in between these saddle points. This distance function depends on the system parameters, i.e., $\Delta=\Delta(w,\delta)$.
    \item In the previous section we established $\Delta(w_h,0)=0$, i.e. we have found a heteroclinic connection (actually two).
    \item We want to solve $\Delta(w,\delta)=0$ near $(w,\delta)=(w_h,0)$. If we can show that $D_w \Delta (w_h,0)\neq 0$ then by the implicit function theorem we have
    $w=w_h(\delta)=w_h + b \delta +O(\delta^2)$ solves $\Delta(w(\delta),\delta)=0$ for $\delta\in (-\delta_0,+\delta_0)$. The leading order parameter is given by 
    $$b=-\frac{D_\delta \Delta (w_h,0)}{D_w \Delta (w_h,0)}$$ 
    \item The first-order derivative terms in the $b$ formula are known as first-order \emph{Melnikov integrals}.
\end{itemize}

\WM{Can we provide proper Melnikov argument here? Yes}
\end{comment}

Define $x=(u,\hat{u})^\top$ and $h(x;w,\delta)=(\hat{u},w+\Phi(u)-\delta \hat u)^\top$ such that the layer problem is given in vector form by
\begin{equation}
x'=h(x;w,\delta)\,,\qquad x\in\mathbb{R}^2\,.
\end{equation}

%\IL{The variable $f$ is already used for the reaction term.}
As shown in section \ref{sec:delta0layer}, this system possesses heteroclinic orbits $\Gamma_\pm(y)$ for $w=w_h$ and $\delta=0$, i.e., $\Gamma_\pm'=h(\Gamma_\pm;w_h,0)$. Let $x=\Gamma_\pm + X$, $X\in\mathbb{R}^2$ which transforms the layer problem to the non-autonomous problem
\begin{equation}
X'=A(y)X +g(X,y;w,\delta)
\end{equation}
with the non-autonomous matrix
$A(y):= D_x h(\Gamma_\pm;w_h,0)$
and the nonlinear remainder
$g(X,y;w,\delta)= h(\Gamma_\pm+X;w,\delta)-h(\Gamma_\pm;w_h,0) - A(y)X$.

Without loss of generality, we define the splitting of the vector space at $y=0$ by 
\begin{equation}
\mathbb{R}^2=\mbox{span}\,\{h(\Gamma_\pm(0);w_h,0)\} \oplus W
\end{equation}
where $W$ is spanned by a solution of the adjoint equation 
$\psi' +A^\top(y)\psi=0$
that decays exponentially for $y\to\pm\infty$; here, this space is one-dimensional and we denote the corresponding solution by $\psi(y)=(\psi_1(y),\psi_2(y))^\top$.

We measure the distance $\Delta\in\mathbb{R}$ between the one-dimensional stable and unstable manifolds emanating from the saddle-equilibria $p_{l/r}$ in a suitable cross section $\Sigma=W$. 
%We denote these manifold segments by $X_\pm$.
%Based on our setup, we choose $\Sigma=W$.
This distance function $\Delta=\Delta(w,\delta)$ depends on the system parameters, and we have $\Delta(w_h,0)=0$. 
%
%\begin{comment}
Melnikov theory establishes the following distance function formula (see Appendix \ref{appendix:smooth})
%(see, e.g., \cite{Vanderbauwhede_1992,Wechselberger_2002}):
%
\begin{equation}\label{equ:distance-fct}
\begin{aligned}
\Delta(w,\delta) & =\int_{-\infty}^0(\psi(s)^\top g(X_-(w,\delta)(y),y;w,\delta))ds -
\int_{\infty}^0
(\psi(s)^\top g(X_+(w,\delta)(y),y;w,\delta))ds\\
&=\int_{-\infty}^\infty(\psi(s)^\top g(X(w,\delta)(y),y;w,\delta))ds 
\end{aligned}
\end{equation}
where $X_\pm(w,\delta)(y)$ denotes the corresponding (un)stable manifold segments of the saddle equilibria $p_{l/r}$ from the corresponding saddle equilibrium to the cross section $\Sigma$, and $X$ is the representative of these sets in the relevant domain. 

If, e.g., $D_w \Delta (w_h,0)\neq 0$ then $w=w_h(\delta)=w_h + b \delta +O(\delta^2)$ solves $\Delta(w_h(\delta),\delta)=0$ for $\delta\in (-\delta_0,+\delta_0)$,  $\delta_0>0$. The leading order expansion parameter $b$ is then given by 
\begin{equation}
b=-\frac{D_\delta \Delta (w_h,0)}{D_w \Delta (w_h,0)}\,,
\end{equation}
%and the corresponding 
and these first-order expansion terms of the distance function $\Delta$ are known as first-order Melnikov integrals.

\begin{proposition}
    The first order Melnikov integrals $D_w \Delta (w_h,0)$ and $D_\delta \Delta (w_h,0)$ are nonzero.
\end{proposition}

%can be calculated as follows: 
%(see, e.g., \cite{Vanderbauwhede_1992,Wechselberger_2002}):
\begin{proof}
    The first order Melnikov integrals can be calculated as follows:
\begin{equation}
\begin{aligned}
D_w \Delta (w_h,0) &=
\int_{-\infty}^\infty (\psi(s)^\top D_w h(\Gamma_\pm(s);w_h,0))ds\,, \\ 
D_\delta \Delta (w_h,0) &=
\int_{-\infty}^\infty (\psi(s)^\top D_\delta h(\Gamma_\pm(s);w_h,0))ds\,.
\end{aligned}
\end{equation}
We have $D_w h(\Gamma_\pm(0);w_h,0)=(0,1)^\top$ and, hence,
\begin{equation}
D_w \Delta (w_h,0)=
\int_{-\infty}^\infty (\psi(s)^\top D_w h(\Gamma_\pm(s);w_h,0))ds =
\int_{-\infty}^\infty \psi_2(s)ds\neq 0,
\end{equation}
based on the observation that the $\psi_2$-component does not change sign along $\Gamma_\pm$, i.e., it is a monotone function along $\Gamma_\pm$. Furthermore, the integral is well-defined since $\psi_2(y)$ is decaying exponentially for $y\to\pm\infty$. Hence by the implicit function theorem, $w=w_h(\delta)=w_h + b \delta +O(\delta^2)$ solves $\Delta(w(\delta),\delta)=0$ for $\delta\in (-\delta_0,+\delta_0)$.

We also have 
$D_\delta h(\Gamma_\pm(0);w_h,0)=(0,-\hat u(y))^\top$ and, hence,
\begin{equation}
D_\delta \Delta (w_h,0)=
\int_{-\infty}^\infty (\psi(s)^\top D_\delta h(\Gamma_\pm(s);w_h,0))ds =
- \int_{-\infty}^\infty \hat{u}(s) \psi_2(s)ds\neq 0\,,
%- \int_{-\infty}^\infty u(s) \psi_1(s)ds\,,
\end{equation}

%$$
based on a similar observation as above, i.e., both terms do not change sign under the variation along $\Gamma_\pm$. 
Hence, 
\begin{equation}
b=\frac{D_\delta \Delta (w_h,0)}{D_w \Delta (w_h,0)}=
-\frac{\int_{-\infty}^\infty \hat{u}(s) \psi_2(s)ds}{\int_{-\infty}^\infty \psi_2(s)ds} \neq 0\,,
\end{equation} 
and we have a leading order affine solution $w(\delta)$ to $\Delta (w,\delta)=0$ near $(w_h,0)$.
\end{proof}

This result confirms the transverse crossing of the heteroclinic branches near $(0,w_h)$ as shown in Figure~\ref{fig:bif-complete}.

\subsubsection{The general $\delta\neq 0$ case}

%In general, 
Heteroclinic orbits $\Gamma^\pm$ of the 2D layer problem \eqref{eq:rnd-layer-1} connecting $S_s^{l/r}$ to $S_s^{r/l}$ are confined to the upper ($\Gamma_+$) or lower ($\Gamma_-$) half-plane in $(u,\hat u)$-space. In these half-planes, the $u$-dynamics is monotone. Hence, all heteroclinics $\Gamma_\pm$ are graphs over the $u$-coordinate chart in $(u,\hat u)$-space, i.e., $\Gamma_\pm : \hat u(u):\, u\in(u_{l},u_{r})$. We consider $\Gamma_+$ here (the same works for $\Gamma_-$). Such a heteroclinic orbit $\hat u(u)$ must fulfill
\begin{equation}
\begin{aligned}
\frac{d\hat u}{du} &=\frac{w+\Phi(u)-\delta \hat u}{\hat u}\,,\quad\forall u\in(u_{l},u_{r})\\
%\implies \hat u\,\frac{d\hat u}{du} &=w+\Phi(u)-\delta \hat u \,,\quad\forall u\in(u_{l},u_{r}) \\
\implies \frac{d}{du}(\frac{\hat u^2}{2}) &=
\frac{d}{du}\int(w+\Phi(u)-\delta \hat u)du\,,\quad\forall u\in(u_{l},u_{r})\\
\implies \frac{\hat u^2}{2} &=
\int_{u_l}^u (w+\Phi(u)-\delta \hat u)du\,,\quad\forall u\in(u_{l},u_{r})\,.
\end{aligned}
\end{equation}

For $u\to u_l$, the last line is fulfilled since $\hat u(u_l)=0$. For $u\to u_r$, where $\hat u(u_r)=0$, we obtain a condition for the existence of a heteroclinic orbit,

\begin{equation}\label{eq:general-equal-area}
%\hat\Delta(w,\delta)=
\boxed{
\int_{u_l}^{u_r} (w+\Phi(u))du=
\delta \int_{u_l}^{u_r}  \hat u(u)\, du \,,
}
\end{equation}

which, for $\delta=0$, gives the equal area rule as established previously. For $\delta\neq 0$ this formula provides a generalised `equal area rule', i.e., the left hand side must move away from its `equal area' position given for $w=w_h(0)$ to counteract the right hand side contribution. This gives $w=w_h(\delta)$ for heteroclinic connections defined by the generalised equal-area rule \eqref{eq:general-equal-area}. Figure \ref{fig:het_delta_08}~(a) shows an example of a heteroclinic orbit for $\delta\neq 0$.\\ 

%Figure~\ref{fig:het_delta_03} shows heteroclinic connections for $\delta=0.3$\,.\\

\begin{comment}
\begin{figure}
    \centering
    \includegraphics[width=9.0cm]{het-orbit-delta01.png}
    %\includegraphics[width=7.5cm]{het-orbit-delta03.png}
    %\includegraphics[width=7.5cm]{het-orbit-delta03-2.png}
    \caption{heteroclinic $\Gamma_-$ for $\delta=0.1$ and $w=w_h(\delta)\approx -0.5661$ 
    %$w=w_h(\delta)\approx -0.1789$ (left) respectively $w=w_h(\delta)\approx -0.2542$ (right) 
    shown in $(u,\hat u)$-space; other parameter values: $\beta=6$, $\gamma_1=7/12$, $\gamma_2=3/4$.} 
    %$\delta=0.3$.\,\WM{do we need these figures?}}
    \label{fig:het_delta_03}
\end{figure}

\end{comment}

For sufficiently large $|\delta|\ge \delta_{m}$, $w$ will necessarily reach its limit $w_{sn}$ where one of the saddle equilibria $p_{l/r}$ goes through a saddle-node bifurcation. Until then, the heteroclinic connection is along the hyperbolic direction, but afterwards it will be along the centre direction which is non-unique and, hence, replaces the codimension-one role of the $w$-variation. Thus, for fixed $w=w_{sn}$ and for sufficiently large $|\delta|>\delta_{m}$, there always exists a heteroclinic orbit located at the boundary of the admissible jump zone. Figure \ref{fig:het_delta_08}~(b) shows an example of a heteroclinic orbit for $\delta\approx\delta_m$ connecting $S_s^r$ to $F_l$.

\begin{figure}[t]
\begin{subfigure}{0.5\textwidth}
\includegraphics[width=8.4cm]{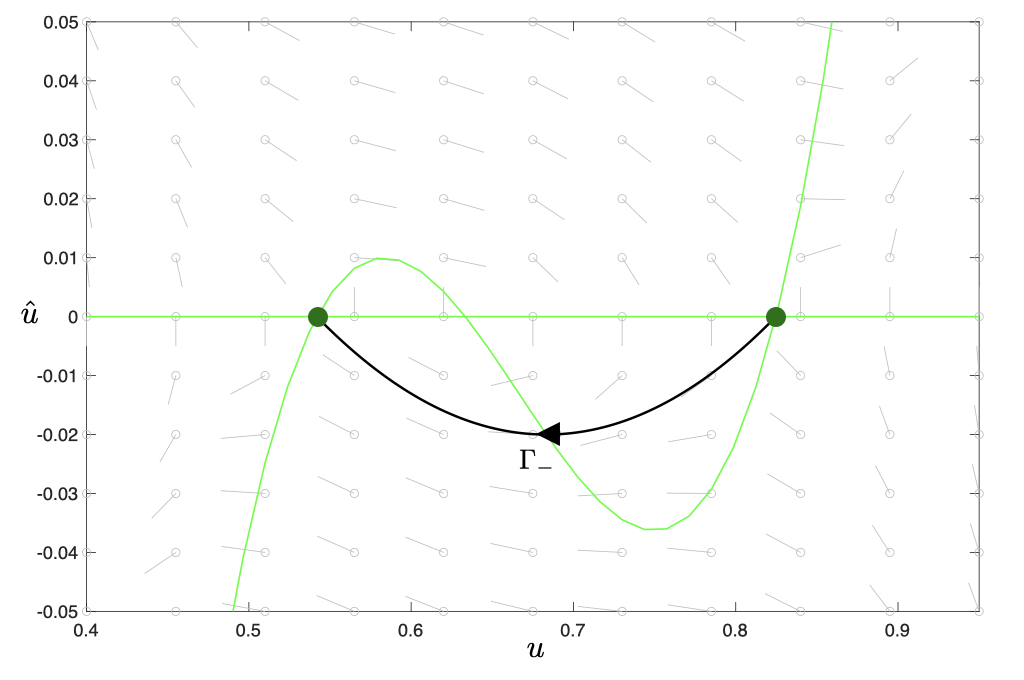} 
\caption{}
\end{subfigure}
\begin{subfigure}{0.5\textwidth}
\includegraphics[width=8.75cm]{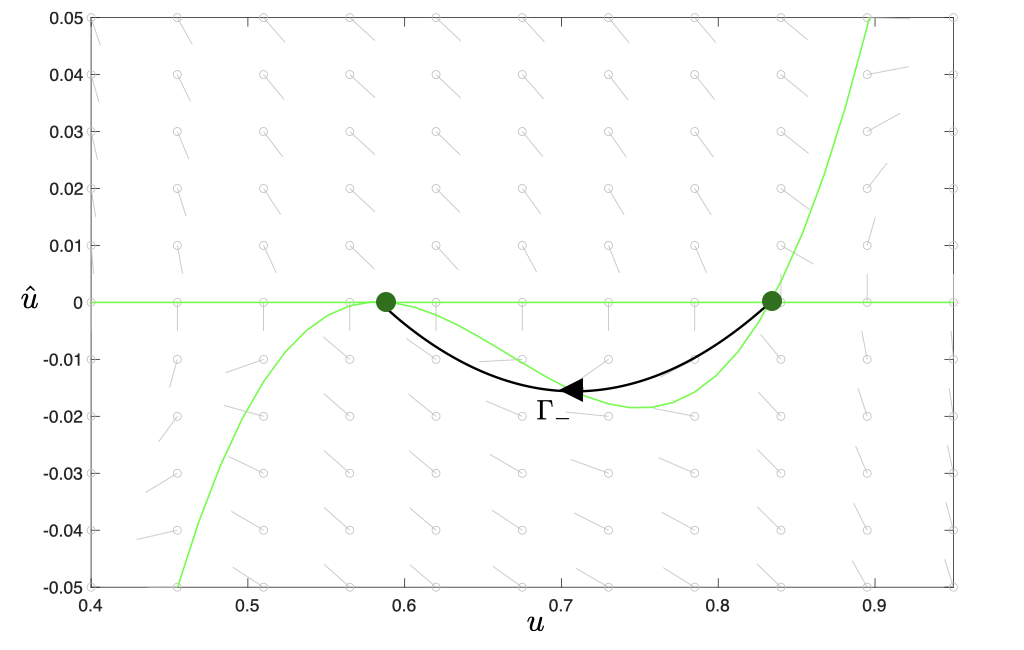}
\caption{}
\end{subfigure}
    \caption{(a) heteroclinic $\Gamma_-$ for $\delta=0.1$ and $w=w_h(\delta)\approx -0.5661$, (b) border case heteroclinic $\Gamma_-$ for $\delta=\delta_m\approx 0.248$ and $w=w_{sn}\approx -0.5671$; %both shown in $(u,\hat u)$-space. 
    %(Right) heteroclinic $\Gamma_+$  for $\delta=1.0 > \delta_{m}$ and $w=w_{sn}\approx -0.1408$ shown in $(u,\hat u)$-space; 
    other parameter values: $\beta=6$, $\gamma_1=7/12$, $\gamma_2=3/4$\,.}
    %\WM{do we need these figures?}}
    \label{fig:het_delta_08}
\end{figure}

\begin{comment}
\begin{figure}[t]
\centering
    \includegraphics[width=7.7cm]{het-orbit-delta01.png}
    \includegraphics[width=8.0cm]{het-orbit-delta0248.png}

    \caption{(left) heteroclinic $\Gamma_-$ for $\delta=0.1$ and $w=w_h(\delta)\approx -0.5661$, (right) border case heteroclinic $\Gamma_-$ for $\delta=\delta_m\approx 0.248$ and $w=w_{sn}\approx -0.5671$; %both shown in $(u,\hat u)$-space. 
    other parameter values: $\beta=6$, $\gamma_1=7/12$, $\gamma_2=3/4$\,.}

    \label{fig:het_delta_08}
\end{figure}
\end{comment}

\begin{remark}
For $w=w_{sn}$, the left-hand side of the generalised equal-area rule \eqref{eq:general-equal-area} is fixed. One concludes that for sufficiently large $|\delta|>\delta_{m}$, there exists a $\hat u(u)$ that fulfills the generalised equal are rule, i.e., $\hat u(u)$ fixes the right hand side $\delta\int \hat u\, du$ to the correct/desired value.  

\end{remark}

\begin{comment}
\begin{figure}
\centering
    \includegraphics[width=8cm]{het-orbit-delta1.png}

    \caption{heteroclinic $\Gamma_+$  for $\delta> \delta_{m}$ and $w=w_{sn}\approx -0.1408$ shown in $(u,\hat u)$-space; other parameter values: $\beta=10$, $\gamma_1=0.3$, $\gamma_2=0.75$, $\delta=1.0$.\,.}
    \label{fig:het_delta_1}
\end{figure}
\end{comment}

\begin{figure}[t]
    \centering
    \includegraphics[width=9.0cm]{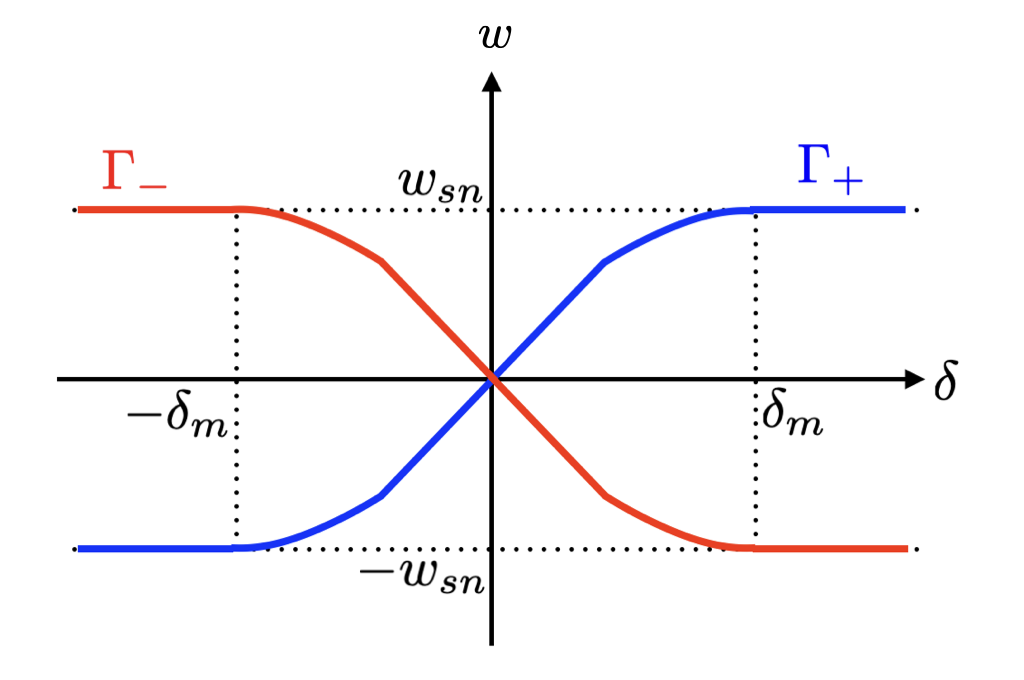}
    \caption{Sketch of complete bifurcation diagram for heteroclinic connections $\Gamma_\pm$ in $(\delta,w)$-space centered at $(w_h(0),0)$. }
%    \IL{Centered at $w_h(0)$, specifically (check axis label)}}
    \label{fig:bif-complete}
\end{figure}

Figure~\ref{fig:bif-complete} summarizes our results on the existence of shocks in the regularised RND model, i.e., the solution branches of $\Delta(w,\delta)=0$. 

Under $\delta$-variation, the shock connection varies continuously in the `jump zone' between the height specified by $w=w_h(0) = -\Phi(u_{infl})$, where $u_{infl}$ denotes the inflection point of the cubic (corresponding to the equal area rule), and a height specified by $w = w_{sn}= -\Phi(F_{l,r})$, where $F_{l,r}$ denote the $u$-values of the fold points of the critical manifold. One important insight here is that viscous relaxation is the dominant regularising effect for $|\delta|>|\delta_m|$
for shock location selection (in the layer problem).

\subsection{Concatenating singular heteroclinic orbits $\Gamma_{het}^\pm$}\label{sec:concat}
%for $\varepsilon = 0$} 
%
According to our analysis in the layer problem, we find a curve $w=w_h(\delta)$ in $(\delta,w )$ space for which there exists singular fast jumps $\Gamma_\pm$ connecting the outer two branches of the critical manifold; see 
Figure~\ref{fig:bif-complete}. 
In the following we aim to concatenate these fast jump segments $\Gamma_\pm$ with slow segments $\Gamma_{l/r}$ of the reduced problem that connect to the given end states $u_\pm$ and, thus, obtain singular heteroclinic orbits $\Gamma_{het}^\pm$ \eqref{def:SingHet} representing singular shock fronted travelling (or standing) waves of our composite regularised RND problem \eqref{eq:vcharnad}.\\

We will start with the construction of singular heteroclinic connections in the `Cahn-Hilliard'-type regularisation limit  $a=0$. Then, we will vary $a$ to produce new families of shock-fronted travelling waves under composite regularisation $a\neq 0$. The system parameters $\beta,\kappa > 0$ are always regarded as fixed for simplicity.

\subsubsection{Singular standing and slowly-moving shocks: the $a = 0$ case} \label{sec:melnikovreduced}

As discussed in Section~\ref{sec:formalshocks}, the locus of fully symmetric standing waves is  specified in $(\gamma_1,\gamma_2,\alpha,c)$-parameter space  by the line segment
\begin{equation}
\mathcal{L} = \{(\gamma_1,1-\gamma_1,1/2,0): \gamma_1 \in (0,1/2)\}.
\end{equation}
We apply a piecewise-smooth variant of the Melnikov method (see Appendix~\ref{sec:pwmelnikov}) to show rigorously that $\mathcal{L}$ locally separates (singular) invasion and evasion fronts. We consider a parameter variation in $(\alpha,c)$-space, where we remind the reader that $\alpha$ determines the location of the middle root of the reaction term $f(u)$ given by \eqref{f:bistable}. We now prove the following.

\begin{proposition} \label{prop:singhet}
For each $\gamma_1 \in (0,1/2)$, there exists $h > 0$ so that a one-parameter family of singular heteroclinic orbits is given locally by a  path $(\gamma_1,\gamma_2,\alpha,c) = (\gamma_1,1-\gamma_1,\alpha,c(\alpha))$ in parameter space, where $\alpha \in (1/2-h,1/2+h)$, $c(1/2) = 0$, and $c'(1/2) < 0$. 
\end{proposition}

To set up the piecewise-smooth Melnikov calculation, we first write down the piecewise-defined planar system of the form \eqref{eq:piecewiseeqns}, whose trajectories are topologically conjugate to those of \eqref{eq:desing-1} subject to the equal area rule shock condition. Denoting $x  = (u,v)$ and $\mu = (\mu_1,\mu_2) = (\alpha,c)$, we have
\begin{equation} \label{eq:pwstanding}
\begin{aligned}
\dot{x} &= \begin{cases}
h_-(x;\alpha,c) := g(x; \mu)& \text{ if } u < u_l\\
h_+(x;\alpha,c) := g(x+(u_r-u_l,0)^T;\mu)& \text{ if } u > u_l.
\end{cases}
\end{aligned}
\end{equation}

\begin{figure}
\centering
 \includegraphics[width=12cm]{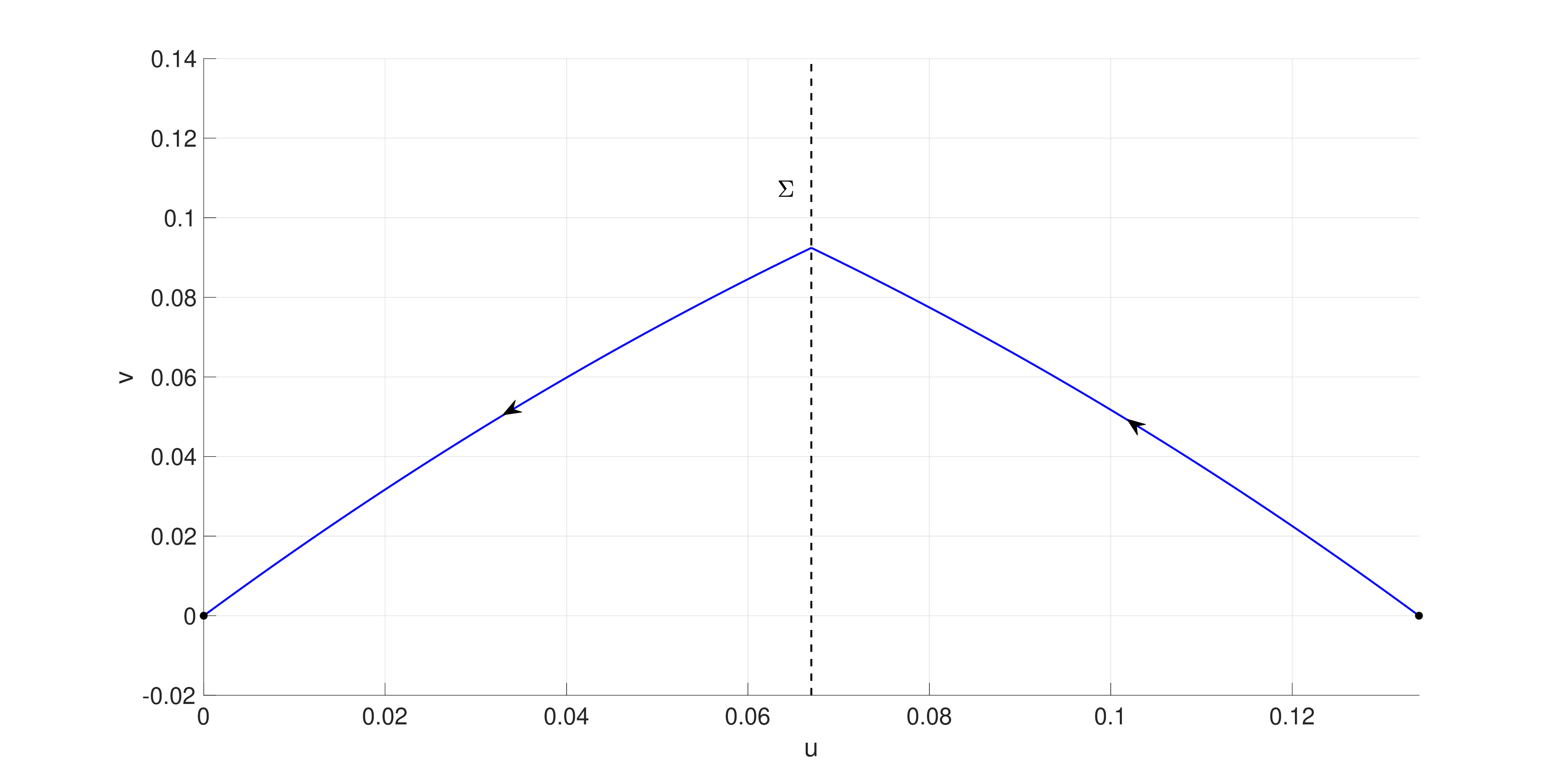}
 % (b)  \includegraphics[width=2.0cm]{reduced-melnikov.png}
\caption{Piecewise-smooth heteroclinic connection obtained by shifting the segment of the corresponding singular heteroclinic orbit (see Figure \ref{fig:symmstanding}) on $S^r_s$ to the left by $u_r - u_l$.}
    \label{fig:symmstanding-melnikov}
\end{figure}

We define $\Sigma$ as a suitable compact segment of $\{u = u_L\}$ intersecting the piecewise-smooth heteroclinic orbit $\gamma(t)$, which is constructed by shifting the right portion of the singular heteroclinic orbit in Figure \ref{fig:symmstanding} to the left by $(u_r-u_l)$ and then selecting the phase so that $\gamma(0)\in \Sigma$; see Figure \ref{fig:symmstanding-melnikov}. 
%As before, 
 Define a function $\Delta(\mu)$ measuring the distance between the (un)stable manifolds of the corresponding saddle points on $\Sigma$. 
Letting $\mu_0 = (\alpha_0,c_0) := (1/2,0)$, we have $\Delta(\mu_0) = 0.$
We seek to apply the implicit function theorem to obtain a local path 
\begin{equation}
c=c(\alpha) = b(\alpha-1/2)+ \mathcal{O}((\alpha-1/2)^2)
\end{equation}
so that there exists $h>0$ so that $\Delta(\alpha,c(\alpha)) = 0$ for $\alpha \in (1/2-h,1/2+h)$, i.e., it suffices to show that $D_c \Delta(\mu_0) \neq 0$. The corresponding leading-order coefficient is then given by
\begin{equation} \label{eq:pwleadingordercoeff}
\begin{aligned}
b &= -\frac{D_{\alpha} G(\mu_0)}{D_{c}G(\mu_0)},
\end{aligned}
\end{equation}

where $G(\mu)$ denotes the component of $\Sigma(\mu) \in \mathbb{R}^2$ evaluated in the $+x_2$ direction. Our goal is now to evaluate the partial derivative terms on the right-hand side of \eqref{eq:pwleadingordercoeff}, which we accomplish by using the following lemma.

\begin{lemma} \label{lem:pwmelnikov}
The leading-order terms $D_{\mu_i} G(\mu_0)$ in \eqref{eq:pwleadingordercoeff} are computed according to \eqref{eq:melnikovintegralspw2}, i.e. we have the corresponding formulas for the (leading-order) Melnikov integrals:

\begin{equation}
\label{eq:melnikovintegralspw2body}
\begin{aligned}
D_{\mu_i}G(\mu_0) &= \frac{1}{v_1^-}\int_{-\infty}^0 \psi_-(s)^{\top}  \frac{\partial h_-}{\partial \mu_i}(\gamma_-(s),s;\mu_0)\,ds
+ \frac{1}{v_1^+}\int_{0}^{\infty} \psi_+(s)^{\top}  \frac{\partial h_+}{\partial \mu_i}(\gamma_+(s),s;\mu_0)\,ds,
\end{aligned}
\end{equation}

where $\psi_{\pm}(s)$ denote a pair of adjoint solutions defined for $t \leq 0$ resp. $t \geq 0$ on the corresponding segment of the piecewise-smooth heteroclinic orbit $\gamma_0(t)$ with normalised initial conditions $\psi_{\pm}(0)$, and $v_{1,\pm}$ denote the $x_1$ components of the normalised vector fields $(v_{1,\pm},v_{2,\pm})$ evaluated at the intersection point $x = x_* \in \Sigma$ of the piecewise-smooth heteroclinic orbit $\gamma_0$ at $t = 0$, i.e.

\begin{equation} \label{eq:normalisedvfields}
\begin{aligned}
(v_{1,-},v_{2,-}) &= h_-(x_*,\mu_0)/||h_-(x_*,\mu_0)|| \text{~and}\\
(v_{1,+},v_{2,+}) &= h_+(x_*,\mu_0)/||h_+(x_*,\mu_0)||.
\end{aligned}
\end{equation}
\end{lemma}

\begin{proof}[Proof of Lemma \ref{lem:pwmelnikov}]
See Appendix \ref{sec:pwmelnikov}.
\end{proof}

\begin{proof}[Proof of Prop. \ref{prop:singhet}]
The Melnikov integrals in \eqref{eq:melnikovintegralspw2body} involve the partial derivatives $D_{\alpha} g(u,v)$ and $D_{c} g(u,v)$. For $u<u_L$, we have

\begin{equation}
\begin{aligned}
D_{\alpha} g(u,v) &= (0,\kappa u(u-1)D(u))\\
D_{c} g(u,v) &= (-u,0),
\end{aligned}
\end{equation}

and similar expressions can be derived for $u > u_l$ using the piecewise-smooth problem \eqref{eq:pwstanding}. For $(\alpha,c) = (1/2,0)$, we define a pair of exponentially decaying solutions to the relevant adjoint problem $\psi_{\pm}$ on either segment $\{u<u_L\}$ and $\{u>u_l\}$, with the orientation at $t = 0$ chosen so that the first components satisfy

\begin{equation*}
    \begin{aligned}
        \psi_{-,1}(s) >0 &\text{ for } s < 0 \text{ and}\\
        \psi_{+,1}(s) <0 &\text{ for } s \geq 0,
    \end{aligned}
\end{equation*}

i.e., the adjoint solution points to the right on $S^{r}_s$ and to the left on $S^{l}_s$. We highlight that

\begin{equation*}
    \begin{aligned}
        \psi_{-,1}(s) = - \psi_{+,1}(-s) 
    \end{aligned}
\end{equation*}

for each $s < 0$; this property is inherited from the vertical reflection symmetry 

\begin{equation*}
h_{-,2}(u,v;1/2,0) = -h_{+,2}(2u_l-u,v;1/2,0)
\end{equation*}

satisfied by the vector fields for $u<u_l$.\\

We now compute the Melnikov integrals defined by \eqref{eq:melnikovintegralspw2body}. Observe that $v_1^{-} = v_1^{+}$ at the intersection point $\gamma(0)$, and hence it suffices to compute the `classical' Melnikov integrals to obtain the leading-order coefficient \eqref{eq:pwleadingordercoeff} since the common prefactor $(1/v_1^{\pm})$ cancels upon evaluating the ratio.\\

Let us denote by $\tilde{u} = u + (u_r-u_l)$ the value of $u$ shifted by the shock segment. Dropping the common prefactor, the corresponding Melnikov integral for $D_{\alpha}\Delta(\mu_0)$ is

\begin{equation*}
    \begin{aligned}
        & \int_{-\infty}^0 \psi_-(s)^\top D_{\alpha}  h_-(\gamma_-(s);1/2,0)\,ds + \int_{0}^{\infty} \psi_+(s)^\top D_{\alpha}  h_+(\gamma_+(s);1/2,0)\,ds\\
         =&         \int_{-\infty}^0 \psi_{-,2}(s) \kappa u(s)(u(s)-1)D(u(s))\,ds + \int_{0}^{\infty} \psi_{+,2}(s) \kappa \tilde{u}(s)(\tilde{u}(s)-1)D(\tilde{u}(s))\,ds < 0
    \end{aligned}
\end{equation*}

since $\psi_{\pm,2}(s)>0$, $\kappa u(s)(u(s)-1) D(u(s))<0$, and $\kappa \tilde{u}(s)(\tilde{u}(s)-1)D(\tilde{u}(s))<0$ along the heteroclinic connection.\\

The remaining Melnikov integral for $D_{c} \Delta (\mu_0)$ is

\begin{equation*}
    \begin{aligned}
        & \int_{-\infty}^0 \psi_-(s)^\top D_{c}  h_-(\gamma_-(s);1/2,0)\,ds + \int_{0}^{\infty} \psi_+(s)^\top D_{c}  h_+(\gamma_+(s);1/2,0)\,ds\\
         =&        - \left(\int_{-\infty}^0 \psi_{-,1}(s) u(s)\,ds  +\int_{0}^{\infty} \psi_{+,1}(s) \tilde{u}(s)\,ds\right).
    \end{aligned}
\end{equation*}

In this case, observe that the sign of $\psi_{\pm,1}(s)$ does change across the shock, so we must inspect the integrands more closely. Using the fact that the horizontal components $\psi_{\pm,1}$ have a reflection antisymmetry, while the corresponding measure $\tilde{u}(s)\,ds$ is unevenly weighted toward the heteroclinic segment on $u > u_L$, we conclude that $D_{c} \Delta (\mu_0)<0$ as well. Hence, from \eqref{eq:pwleadingordercoeff}  we have that $b<0$. The proof is completed by noting that $c'(1/2) = b$.
\end{proof}

This result implies that invasion fronts (with $c>0$) emerge along a locally affine branch as $\alpha$ is varied below $\alpha = 1/2$, resp. evasion fronts ($c<0$) for $\alpha > 1/2$. By symmetry, a reflected connection from $u=u_-$ to $u=u_+$ coexists for $c=0$. An analogous Melnikov calculation gives a distinct local branch of bifurcations connecting $u=u_+$ to $u=u_-$.  The bifurcation structure is therefore locally similar to the branch crossing for the fast shock fronts depicted in Figure \ref{fig:bif-complete}. The invasion portion of one of these locally affine branches (i.e. $\alpha < 1/2,\,c >0$) is depicted in Figure \ref{fig:continuation}(a) for an example parameter set.

\begin{figure}[t]
\begin{subfigure}{0.5\textwidth}
\includegraphics[width=8.75cm]{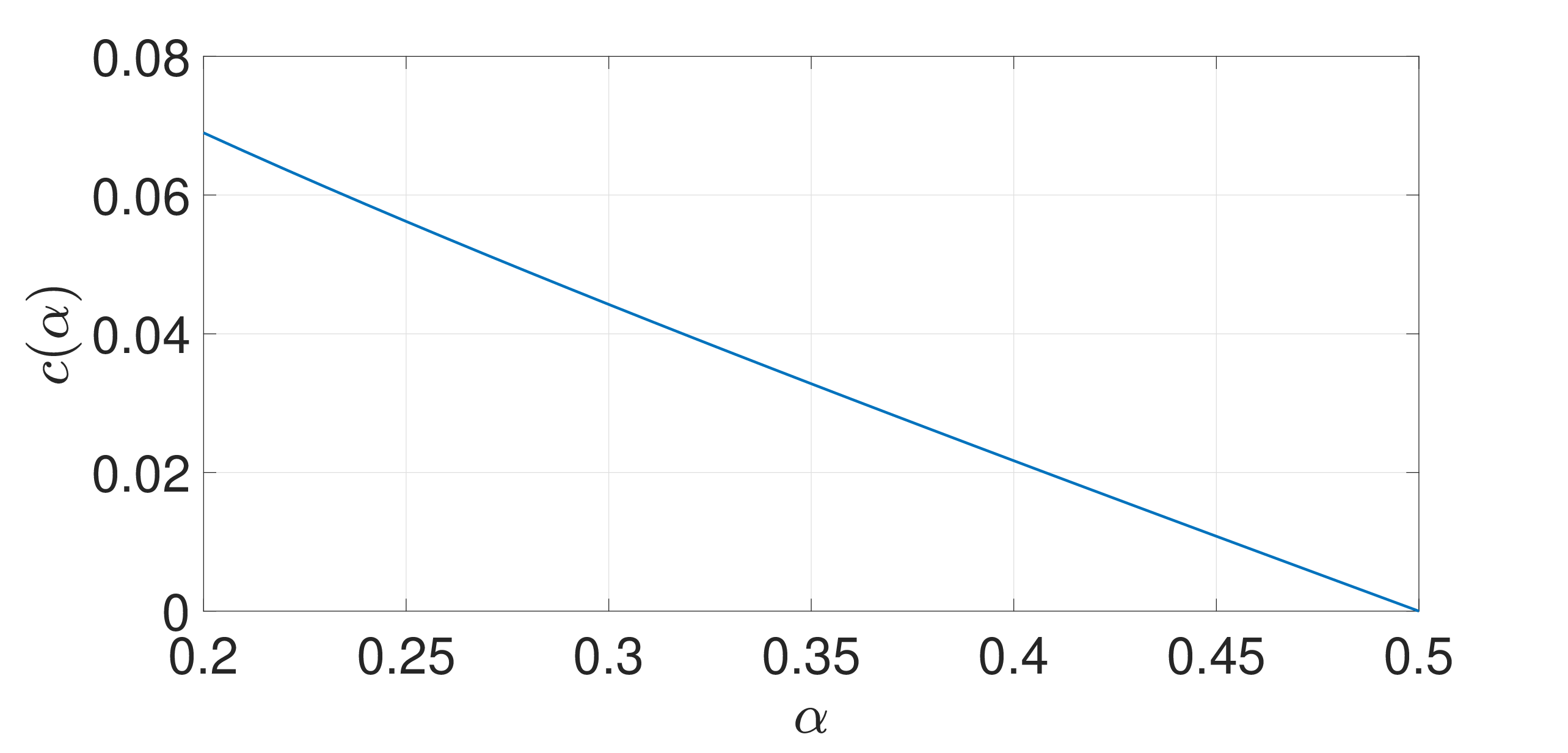} 
\caption{}
\end{subfigure}
\begin{subfigure}{0.5\textwidth}
\includegraphics[width=8.75cm]{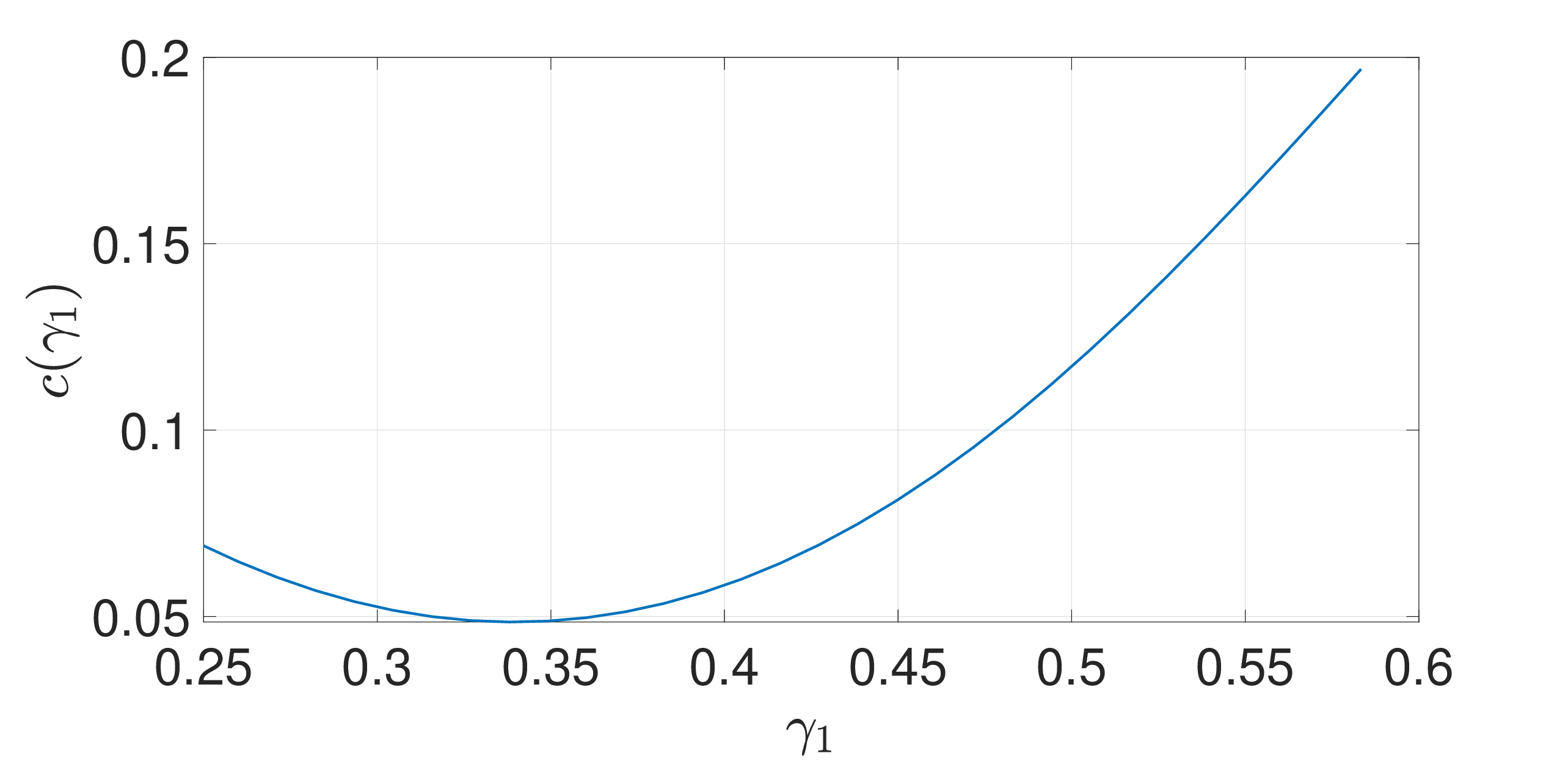}
\caption{}
\end{subfigure}
    \caption{One-parameter continuations of a singular heteroclinic orbit joining the saddle points at $u=0$ and $u=1$ for the desingularised slow  flow \eqref{eq:desing-1}, subject to shock selection given by the equal area rule. Initial parameter set: $\gamma_1 = 1/4,\,\gamma_2 = 1-\gamma_1 = 3/4,\,\alpha = 1/2,\,a=0,\,c=0,\,\kappa=5,\,\beta = 6$. (a) The continuation parameter $\alpha$ is varied from $\alpha = 0.5$ to $\alpha = 0.2$ with $c$ as the free parameter. (b) The continuation parameter $\gamma_1$ is then varied from $\gamma_1 = 1/4$ to $\gamma_1 = 7/12 = 0.58\bar{3}$, with $c$ as the free parameter. Note that $c(7/12) = c_* \approx 0.19686$.}
    \label{fig:continuation}
\end{figure}

\begin{comment}
\begin{figure}[t]
\centering
 (a)  \includegraphics[width=7.5cm]{c0vsalpha.eps}
 (b) \includegraphics[width=7.5cm]{c0vsgamma1.eps}
   % (b) \includegraphics[width=7.5cm]{interp-slowconcat.eps}
    \caption{One-parameter continuations of a singular heteroclinic orbit joining the saddle points at $u=0$ and $u=1$ for the desingularised slow  flow \eqref{eq:desing-1}, subject to shock selection given by the equal area rule. Initial parameter set: $\gamma_1 = 1/4,\,\gamma_2 = 1-\gamma_1 = 3/4,\,\alpha = 1/2,\,a=0,\,c=0,\,\kappa=5,\,\beta = 6$. (a) The continuation parameter $\alpha$ is varied from $\alpha = 0.5$ to $\alpha = 0.2$ with $c$ as the free parameter. (b) The continuation parameter $\gamma_1$ is then varied from $\gamma_1 = 1/4$ to $\gamma_1 = 7/12 = 0.58\bar{3}$, with $c$ as the free parameter. Note that $c(7/12) = c_* \approx 0.19686$.}
    \label{fig:continuation}
\end{figure}
\end{comment}

%%%%%%%%%%%%%%%%%%%%%%%%%%%%%%%%%%%%%%%%%%%%%%%%%%%%%%%%%%%%
\subsubsection{Continuation of singular heteroclinic orbits}
%\WM{continuation starts here}
This local patch of singular heteroclinic connections serves as a natural starting point for global continuation. For instance, we can demonstrate that the singular heteroclinic connection previously identified in the `Cahn-Hilliard'-type regularisation setting \cite{li2021,lizarraga-nonlocal}, corresponding to the parameter set 

\begin{equation*}
\gamma_1 = 7/12,\,\gamma_2 = 3/4,\,\alpha = 1/5,\,a = 0,\,c_* \approx 0.19686,\,\kappa = 5,\,\beta = 6,
\end{equation*}

is `accessible' from $\mathcal{L}$ via numerical continuation, i.e., they lie on the same path-connected component of the submanifold of heteroclinic connections. Concretely, we begin with the parameter set

\begin{equation*}
\gamma_1 = 1/4,\,\gamma_2 = 1-\gamma_1 = 3/4,\,\alpha = 1/2,\,a=0,\,c=0,\,\kappa=5,\,\beta = 6
\end{equation*}

on $\mathcal{L}$, giving rise to a symmetric standing wave (depicted in Figure \ref{fig:symmstanding}). We then perform a sequence of one-parameter continuations: first we vary $\alpha$ from $\alpha = 1/2$ to $\alpha = 1/5$ (Figure \ref{fig:continuation}(a)), and then we vary $\gamma_1$ from $\gamma_1  = 1/4$ to $\gamma_1 = 7/12$ (Figure \ref{fig:continuation}(b)), both times leaving $c$ free. As expected, the wavespeed $c=c_*$ found after these continuations agrees with the value computed in \cite{li2021,lizarraga-nonlocal}, and the corresponding singular connections are identical.

\begin{remark} \label{rem:standingwavecont}
We can also fix $c=0$ and choose some other free parameter (e.g., we can trace families of asymmetric standing waves). Note that the reduced problem \eqref{eq:desing-1} remains Hamiltonian in this case (see Remark \ref{rem:hamiltonian}); but in view of the jump condition, the disjoint slow portions of the singular heteroclinic connections that we seek will typically lie on different energy surfaces. These slow segments happen to lie on the same level set in the symmetric case (see, e.g., Figure \ref{fig:symmstanding}), but this scenario is exceptional.
\end{remark}

Altogether, the line of symmetric standing waves can be continued to a regular codimension-1 submanifold of singular heteroclinic bifurcations by, e.g., letting the first three parameters vary and leaving the wavespeed $c$ as a free parameter. \\

\subsubsection{Singular heteroclinic connections for $a > 0$}

\begin{figure}[t]
\begin{subfigure}{0.5\textwidth}
\includegraphics[width=8.75cm]{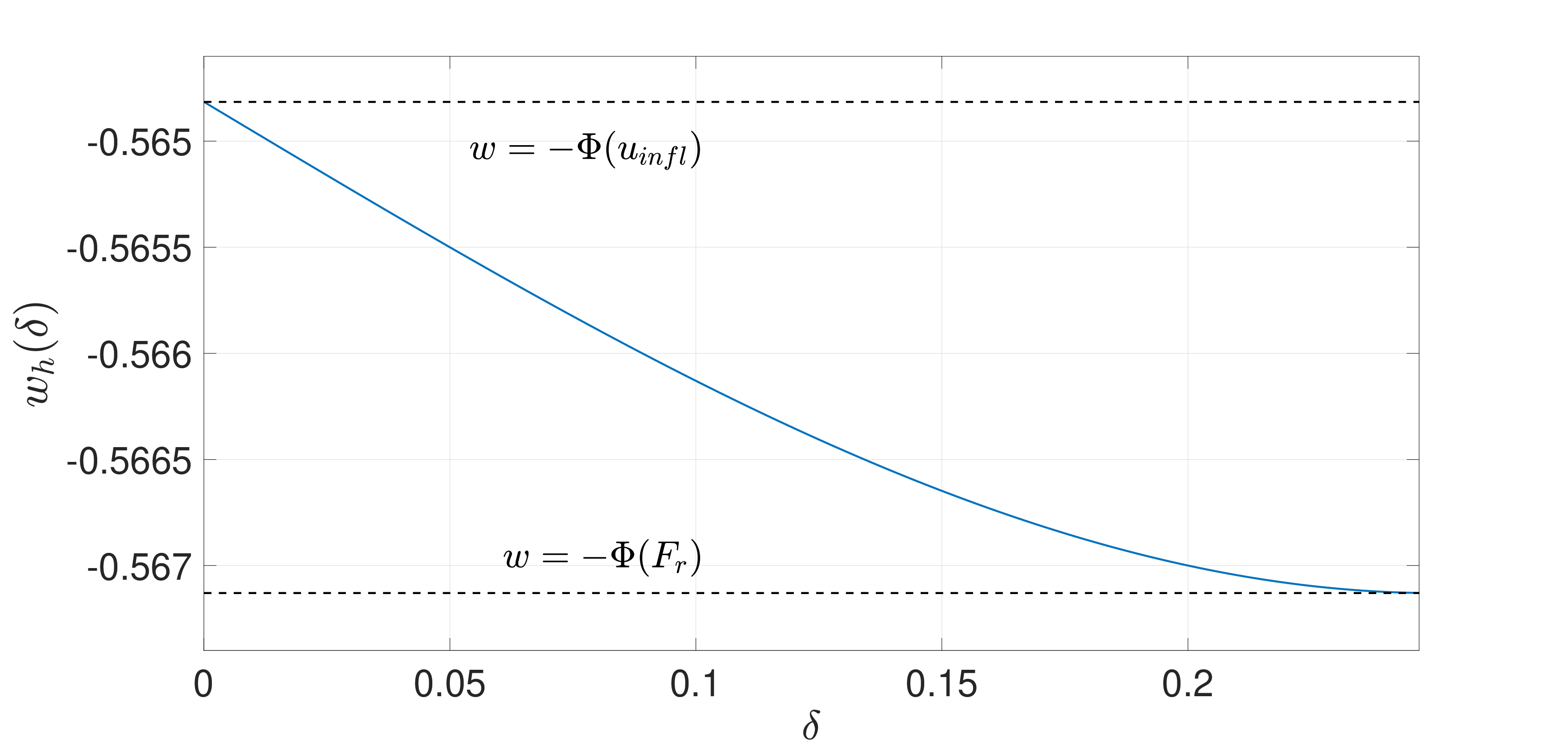} 
\caption{}
\end{subfigure}
\begin{subfigure}{0.5\textwidth}
\includegraphics[width=8.75cm]{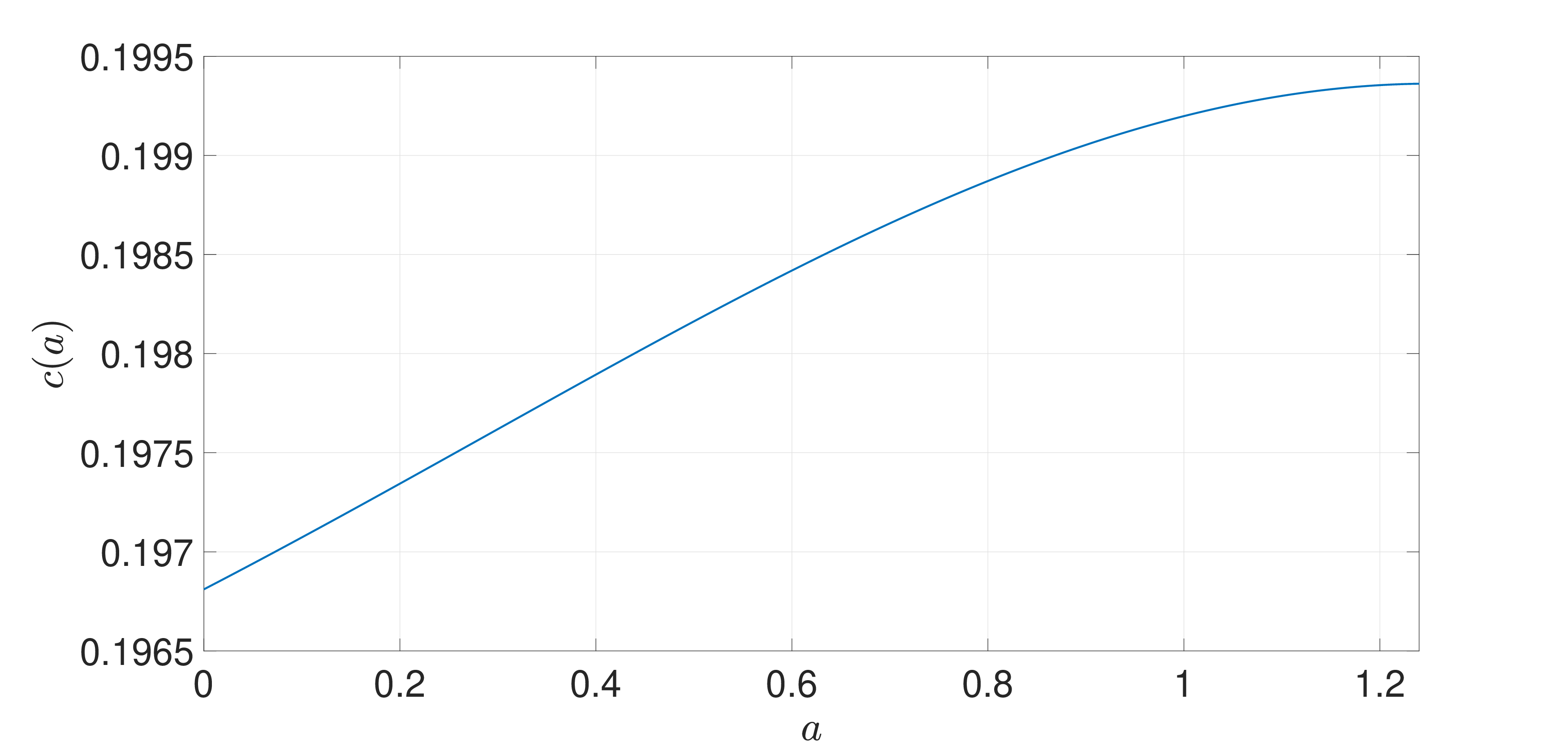}
\caption{}
\end{subfigure}
    \caption{(a) Bifurcation diagram in $(\delta,w)$-space for the shock height selection in the layer problem, corresponding to a segment of the branch $\Gamma_-$, as shown in Fig. \ref{fig:bif-complete}. Here, $\delta_m \approx 0.248$. (b) Bifurcation diagram in $(a,c)$-space for the singular heteroclinic orbits.  Parameter set: $\beta = 6$, $\gamma_1 = 7/12$, $\gamma_2 =3/4$, $\kappa = 5$, $\alpha = 1/5$.}
    \label{fig:slowbifdiag}
\end{figure}

\begin{comment}
\begin{figure}[t]
\centering
 (a)  \includegraphics[width=7.5cm]{whvsdelta3.eps}
 (b) \includegraphics[width=7.5cm]{cvsaparamyifei.eps}
   % (b) \includegraphics[width=7.5cm]{interp-slowconcat.eps}
    \caption{(a) Bifurcation diagram in $(\delta,w)$-space for the shock height selection in the layer problem, corresponding to a segment of the branch $\Gamma_-$, as shown in Fig. \ref{fig:bif-complete}. Here, $\delta_m \approx 0.248$. (b) Bifurcation diagram in $(a,c)$-space for the singular heteroclinic orbits.  Parameter set: $\beta = 6$, $\gamma_1 = 7/12$, $\gamma_2 =3/4$, $\kappa = 5$, $\alpha = 1/5$.}
    \label{fig:slowbifdiag}
\end{figure}
\end{comment}

Let us now extend the parameter space in the above analysis by the regularisation weighting parameter $a$. We demonstrate that the heteroclinic orbits in the previous section persist as transversal intersections of (un)stable manifolds of the saddle points in the extended parameter space.

Let $a \geq 0$ be specified; then as  $c$ varies, the corresponding extreme values $u = u_L(c),\,u_R(c)$ selected by the shock also vary continuously according to the height condition $w_h(\delta) = w_h(ac) = -\Phi(u)$ specified by the layer problem; see Figure \ref{fig:slowbifdiag}(a).\\

Varying $c$ simultaneously varies the (un)stable manifolds of the saddle points in the reduced problem. Suppose that the stable manifold $W^s(p_+,c)$ of the saddle point at $u_+ = 0$ has its first intersection with the cross-section $\{u = u_L(c)\}$ at a point $p_l(c)$, and similarly that the unstable manifold $W^u(p_-,c)$ of the saddle point at $p_- = 1$ has its first intersection with the cross-section $\{u = u_R(c)\}$ at $p_r(c)$. Under variation of $c$ we can then locate a locally unique value $c=c(a)$ at which the $v$-coordinates of $p_l(c)$ and $p_r(c)$ coincide; i.e., so that the shock simultaneously connects two slow trajectories that join the two saddle points.\\

Altogether, we have defined a bifurcation problem for the singular heteroclinic orbits with respect to the regularisation weighting parameter $a$; see Figure~\ref{fig:slowbifdiag}(b) for the resulting bifurcation diagram for our example parameter set. Figure~\ref{fig:singlehet} depicts an example of a singular heteroclinic connection formed under variation of the wavespeed parameter $c$ for fixed $a > 0$. We emphasize that each such invasion shock front formed within the parameter interval  $a\in [0,a_m]$, with $a_m \approx 1.2465$, satisfies a distinct generalised area rule! For the parameter set in the figure, a fixed wavespeed $c_m \approx 0.1994$ is selected for each  $a > a_m$, since this is the wavespeed at which the portions of the singular heteroclinic connection that lie on the critical manifold connect to a viscous shock at the fixed height specified by $w_h(\delta)=w_{sn}=-\Phi(u)$, which is in turn satisfied for each $\delta > \delta_m$ as depicted in Figure~\ref{fig:bif-complete}.

\begin{figure}
\centering
   \includegraphics[width=12cm]{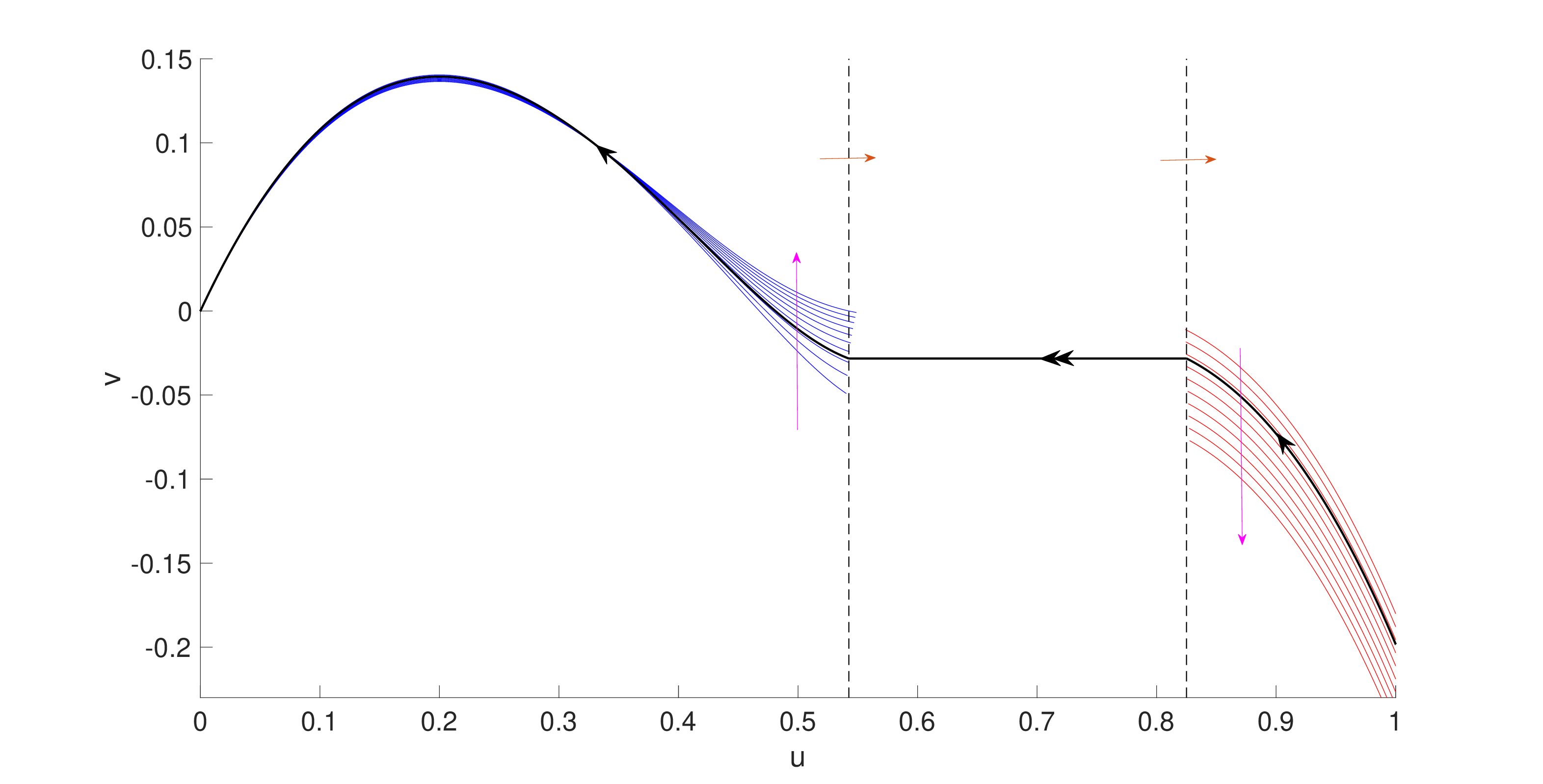}

   % (b) \includegraphics[width=7.5cm]{interp-slowconcat.eps}
    \caption{Singular heteroclinic orbit for $(a_*,c_*) \approx  (0.5182,0.19817)$, formed as a transversal intersection of the stable manifold $W^s(p_+,c)$ (blue curves) and unstable manifold $W^u(p_-,c)$ (red curves) as $c$ is increased within the interval $[0.18,0.25]$. Jump values $u = u_L(c_*),\,u_R(c_*)$ satisfying $-w_h(\delta_* = a_* c_*) = \Phi(u)$ denoted by dashed black lines. Vertical magenta arrows: direction of variation of (un)stable manifolds; horizontal orange arrows: variation of the endpoints of the shock, as $c$ increases. Parameter set: $\beta = 6$, $\gamma_1 = 7/12$, $\gamma_2 =3/4$, $\kappa = 5$, $\alpha = 1/5$.
    % \WM{Ian, I just noticed that you show here an example where the middle equilibrium is not on the middle branch. Do you have an example where it is in the middle branch?}
    }
    \label{fig:singlehet}
\end{figure}

\begin{figure}[th!]
\centering
  (a) \includegraphics[width=12cm]{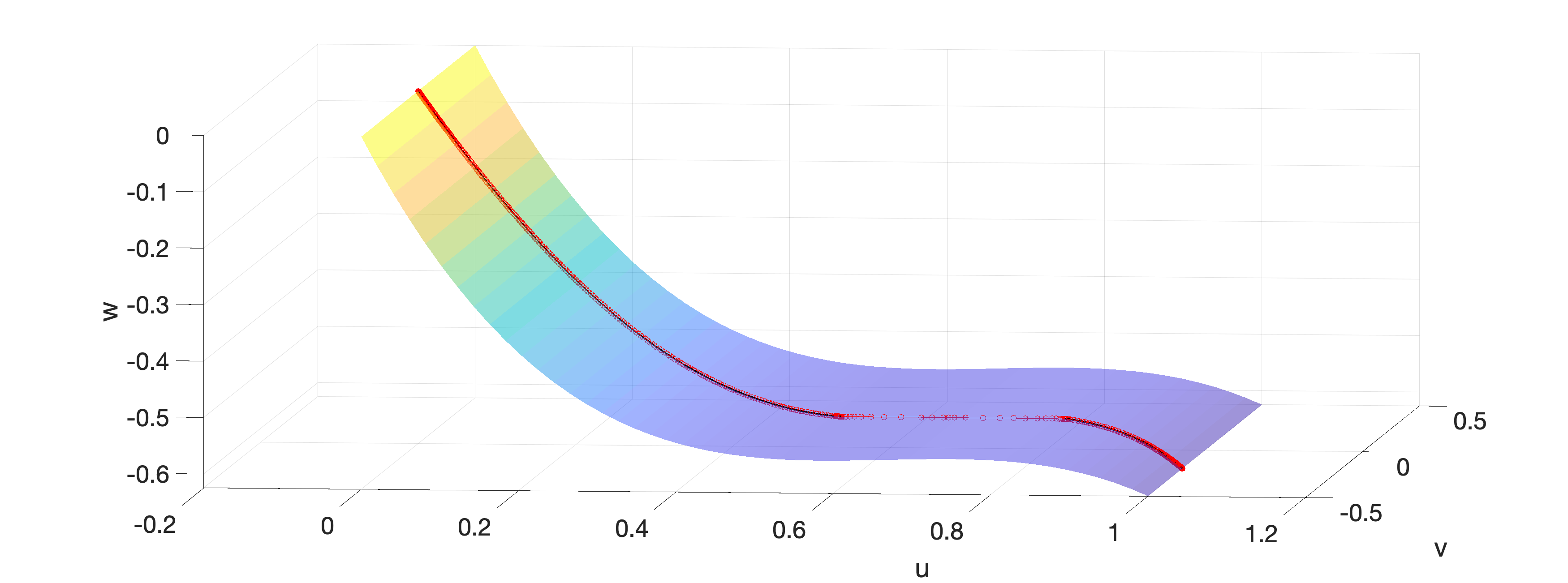}\\
  (b) \includegraphics[width=12cm]{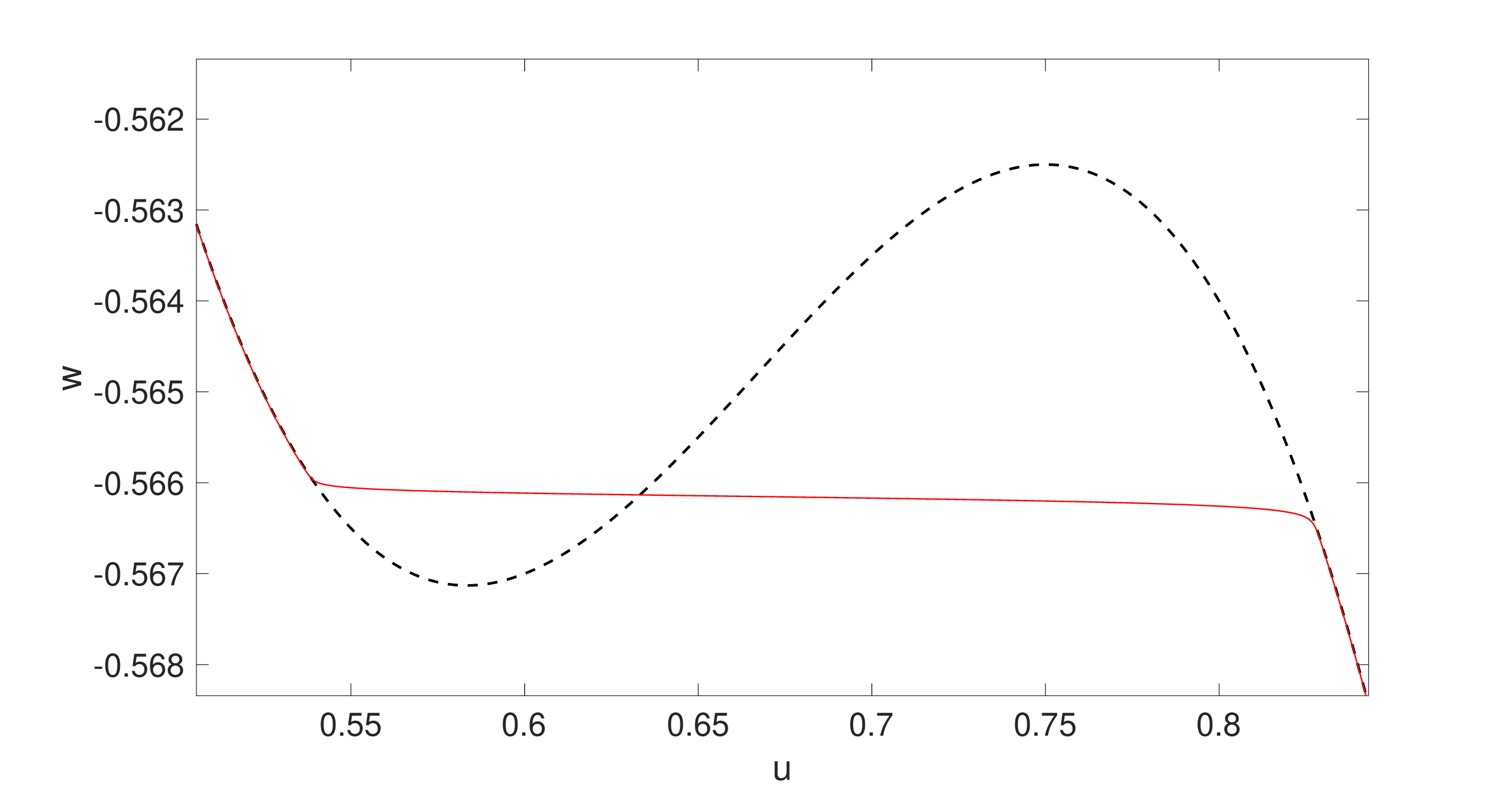}
   % (b) \includegraphics[width=7.5cm]{interp-slowconcat.eps}
    \caption{(a) A $(u,v,w)$-projection of a heteroclinic orbit of \eqref{eq:rnd-sys-1-fast} for $\varepsilon = 10^{-4}$, $a \approx 0.5182$, and $c = 0.19826$. Black curve segments on the critical manifold (obscured by red curve) denote the slow portions of the corresponding singular heteroclinic orbit depicted in Figure \ref{fig:singlehet}. (b) Side view of the heteroclinic orbit demonstrating that the connection satisfies the perturbed generalised area rule $w_h(\delta) = -\Phi(u)$. Parameter set: $\beta = 6$, $\gamma_1 = 7/12$, $\gamma_2 =3/4$, $\kappa = 5$, $\alpha = 1/5$. Computations performed with the \texttt{bvp4c} boundary value solver in MATLAB 2021a, with a relative error tolerance of $10^{-5}$.}
    \label{fig:fullhet}
\end{figure}

\subsection{Persistence of heteroclinics for $a \geq 0$ and $\varepsilon > 0$}

Our study of the existence of shock-fronted travelling waves for the RND PDE \eqref{eq:vcharnad} now culminates in a persistence analysis of the family of singular heteroclinic orbits we have constructed in the previous sections. In this section we sketch the relevant geometrical construction for fixed $a \geq 0$ and all sufficiently small values of $\varepsilon > 0$. The sketch for the `interpolated' shock case ($0 \leq |\delta| < \delta_m$) is slightly different from the viscous shock case ($|\delta| \geq \delta_m$), so we treat them separately. \\

\textit{The `interpolated' shock ($0 \leq |\delta| < \delta_m$) subcase.} We consider the extended five-dimensional dynamical system formed by appending the equation $c' = 0$ to the system \eqref{eq:rnd-sys-1-slow}. The (un)stable manifolds $W^s(u_-)$ and $W^u(u_+)$ in the four-dimensional system are now extended to three-dimensional center-stable and center-unstable manifolds $W^{cs}(u_-)$ and $W^{cu}(u_+)$. In the extended system, a dimension count readily verifies that intersections between these manifolds are generically transversal. Transversality in the singular limit can be verified using a Melnikov-type analysis along the singular layer connections and Fenichel theory (see e.g. Secs. 4 and 5 in \cite{sz-transverse}). The transversal connection is depicted in the extended system for an example parameter set in Fig. \ref{fig:singlehet}.\\

\textit{The viscous shock ($|\delta| \geq \delta_m$) subcase.} As before, we must verify transversality of the singular heteroclinic connection in the extended problem. In this case, Fenichel theory breaks down at the fold where the viscous-type shock lands. The technical tool to extend the relevant normally hyperbolic segments of the critical manifold across small neighbourhoods of the fold is geometric blow-up theory; see e.g. \cite{linwex13}.\\

Thus, our family of singular heteroclinic orbits defined for $\varepsilon = 0$ generically perturbs to a codimension-one manifold of heteroclinic bifurcations of \eqref{eq:rnd-sys-1-fast} in $(a,c,\varepsilon)$ parameter space. An example of such an orbit, which satisfies a generalised area rule for a nontrivial value of $a > 0$, is depicted in Figure~\ref{fig:fullhet}.

\begin{figure}[t]
\begin{subfigure}{0.5\textwidth}
\includegraphics[width=8.75cm]{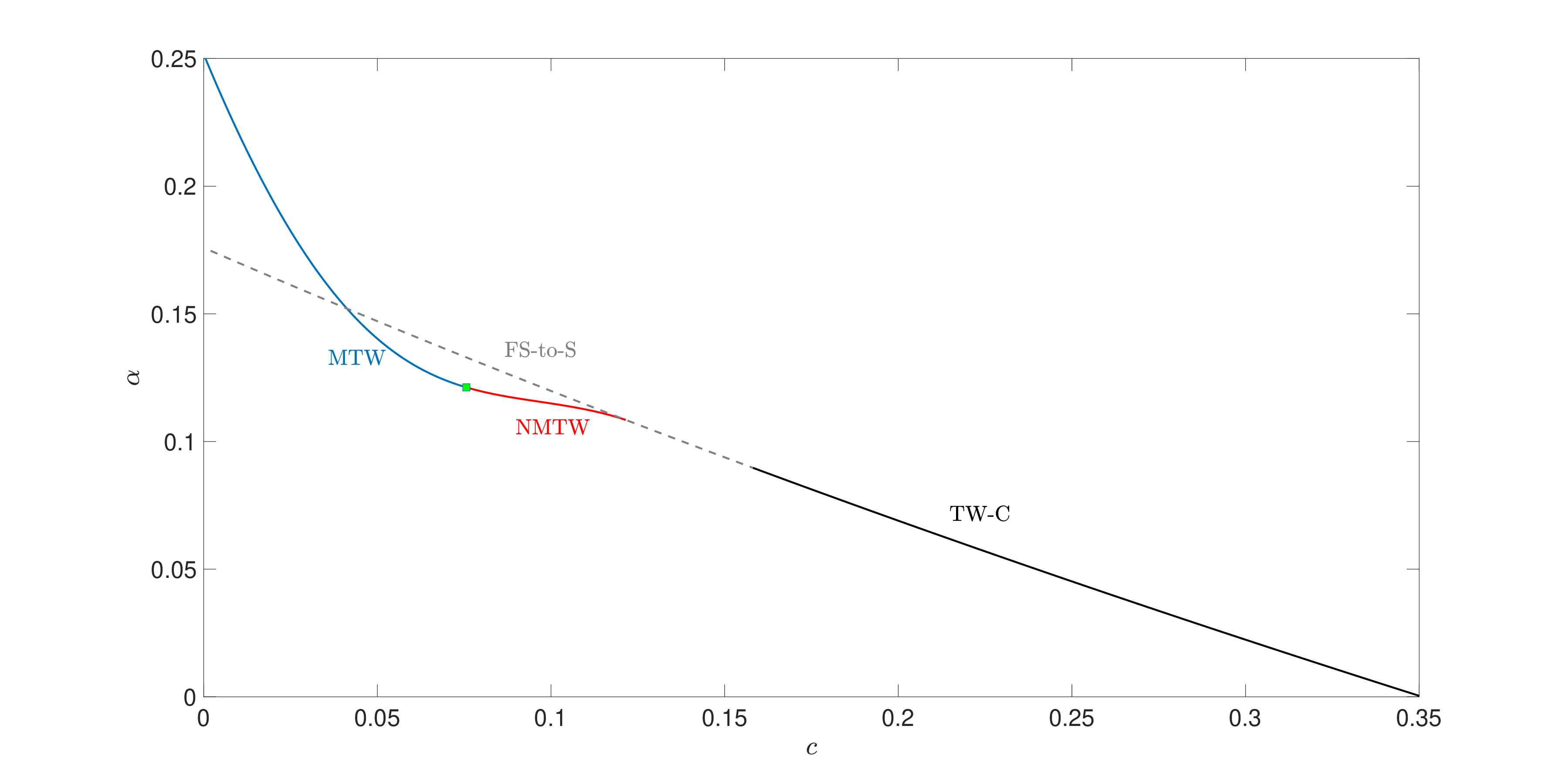} 
\caption{}
\end{subfigure}
\begin{subfigure}{0.5\textwidth}
\includegraphics[width=8.75cm]{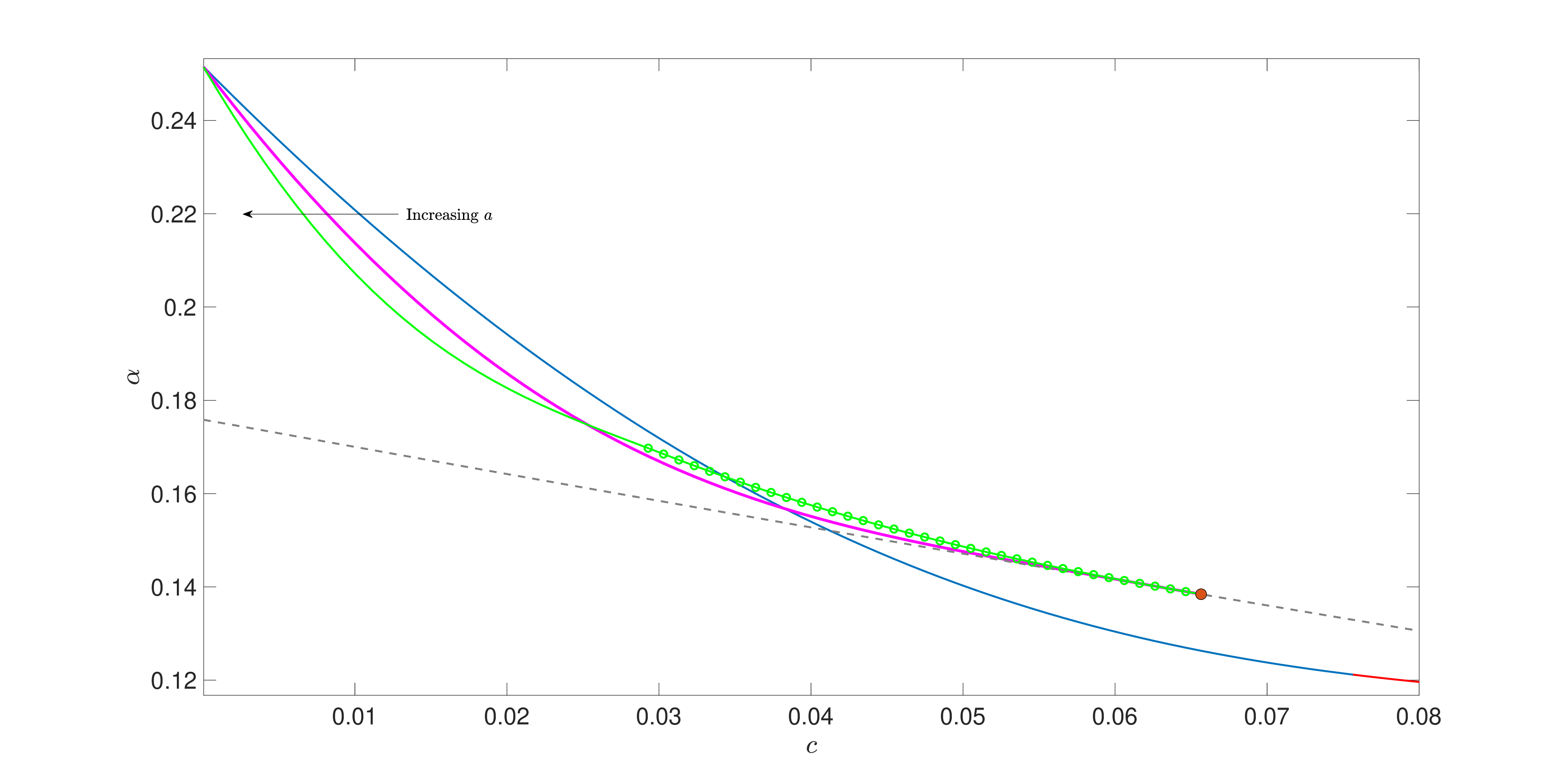}
\caption{}
\end{subfigure}
    \caption{(a) Singular bifurcation diagram depicting a monotone (blue, MTW) to nonmonotone (red, NMTW) shock-fronted travelling wave transition, via a codimension-two tangency bifurcation (green square) of $W^s(p_+)$ with the landing curve $\{u = u_l\}$. The family of nonmonotone waves terminates on the folded-saddle-to-saddle heteroclinic bifurcation (FS-to-S) curve (grey dashed curve). For each fixed $a>0$, the FS-to-S curve contains an open subset (black solid curve) where shock-fronted travelling waves with canard segments (TW-C) begin to appear.  Parameter set: $\beta = 1$, $\gamma_1 = 0.4$, $\gamma_2 = 0.75$, $\kappa = 3$, $a=0.5$. For this parameter set, we have $\delta_m \approx 0.2121$. (b) (Non)monotone wave bifurcation curves for three different values of $a$: (i) $a = 0.5$ (blue and red solid), (ii) $a = a_* \approx 3.2304$ (magenta solid), and (iii) $a = 7$ (green solid and with bubble markers). Solid red, blue, magenta, and green curves denote interpolated shock rules and the green curve with bubble markers denotes monotone waves with viscous shocks. Orange dot: codimension-three scenario for $(a_*,c_*,\alpha_*) \approx (3.2304,0.0657,0.1384)$ in which the monotone waves (magenta curve) terminate at an FS-to-S connection precisely as the shock rule changes from interpolated type to viscous type. Remaining parameters as in (a).
     }
     \label{fig:nonmbifdiag}
\end{figure}

\begin{comment}
 \begin{figure}[t]
\centering
   (a) \includegraphics[width=7.5cm]{bifdiag_nonmonotone2.eps}
   (b) \includegraphics[width=7.5cm]{hetfs_codim3.eps}
     \caption{(a) Singular bifurcation diagram depicting a monotone (blue, MTW) to nonmonotone (red, NMTW) shock-fronted travelling wave transition, via a codimension-two tangency bifurcation (green square) of $W^s(p_+)$ with the landing curve $\{u = u_l\}$. The family of nonmonotone waves terminates on the folded-saddle-to-saddle heteroclinic bifurcation (FS-to-S) curve (grey dashed curve). For each fixed $a>0$, the FS-to-S curve contains an open subset (black solid curve) where shock-fronted travelling waves with canard segments (TW-C) begin to appear.  Parameter set: $\beta = 1$, $\gamma_1 = 0.4$, $\gamma_2 = 0.75$, $\kappa = 3$, $a=0.5$. For this parameter set, we have $\delta_m \approx 0.2121$. (b) (Non)monotone wave bifurcation curves for three different values of $a$: (i) $a = 0.5$ (blue and red solid), (ii) $a = a_* \approx 3.2304$ (magenta solid), and (iii) $a = 7$ (green solid and with bubble markers). Solid red, blue, magenta, and green curves denote interpolated shock rules and the green curve with bubble markers denotes monotone waves with viscous shocks. Orange dot: codimension-three scenario for $(a_*,c_*,\alpha_*) \approx (3.2304,0.0657,0.1384)$ in which the monotone waves (magenta curve) terminate at an FS-to-S connection precisely as the shock rule changes from interpolated type to viscous type. Remaining parameters as in (a).
     }
     \label{fig:nonmbifdiag}
 \end{figure}
\end{comment}

\begin{figure}
\begin{subfigure}{0.5\textwidth}
\includegraphics[width=8.75cm]{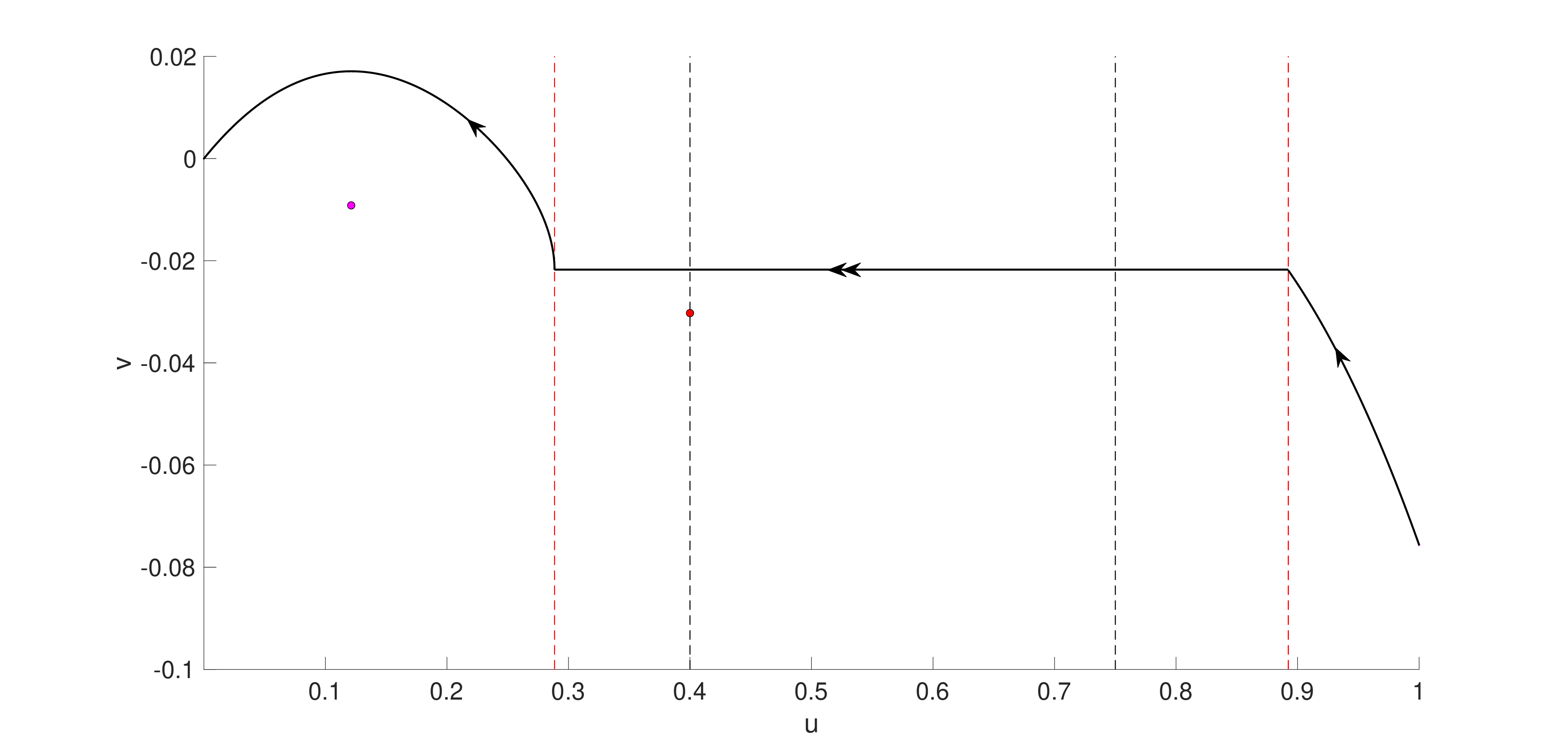} 
\caption{}
\end{subfigure}
\begin{subfigure}{0.5\textwidth}
\includegraphics[width=8.75cm]{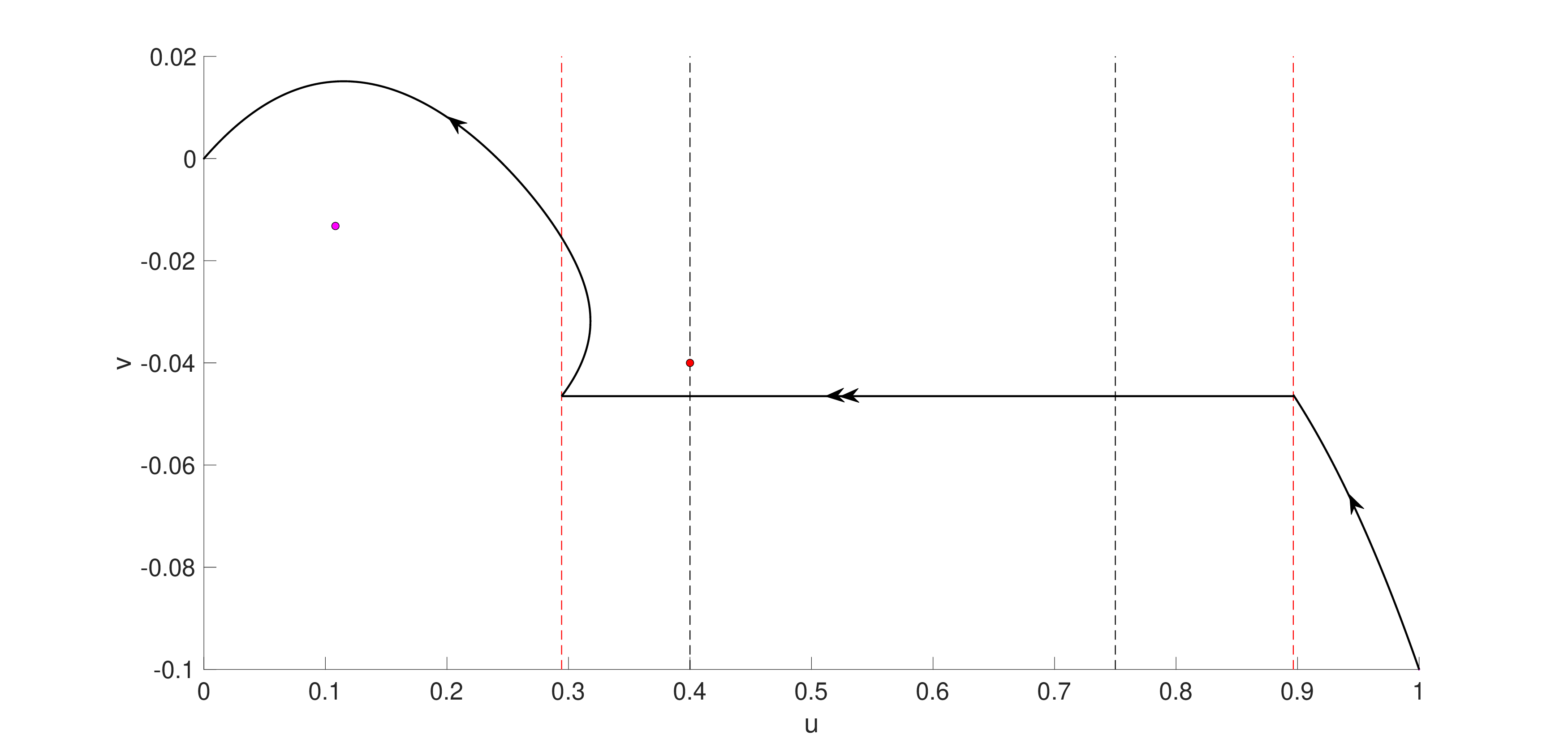}
\caption{}
\end{subfigure}
    \caption{(a) Travelling wave at the moment of tangency of $W^s(p_+)$ with the landing curve of the shock $\{u = u_{l}\}$. Parameters: $(c,\alpha) \approx (0.0756, 0.1212)$ (corresponding to the location of the green square in Figure \ref{fig:nonmbifdiag}(a)). (b) A (singular) nonmonotone travelling wave for $(c,\alpha) \approx (0.1,0.115)$.  Red point: folded saddle at $(u_{f_l},-c u_{f_l})$; magenta point: equilibrium at $(\alpha,-c \alpha)$. Remaining parameters: $\beta = 1$, $\gamma_1 = 2/5$, $\gamma_2 = 3/4$, $\kappa = 3$, $a=1/2$. 
     }
     \label{fig:nonmonotone}
\end{figure}

\begin{comment}
\begin{figure}
\centering
(a)    \includegraphics[width=7.5cm]{tangencywave.eps}
 (b)   \includegraphics[width=7.5cm]{nmwexample.eps}
     \caption{(a) Travelling wave at the moment of tangency of $W^s(p_+)$ with the landing curve of the shock $\{u = u_{l}\}$. Parameters: $(c,\alpha) \approx (0.0756, 0.1212)$ (corresponding to the location of the green square in Figure \ref{fig:nonmbifdiag}(a)). (b) A (singular) nonmonotone travelling wave for $(c,\alpha) \approx (0.1,0.115)$.  Red point: folded saddle at $(u_{f_l},-c u_{f_l})$; magenta point: equilibrium at $(\alpha,-c \alpha)$. Remaining parameters: $\beta = 1$, $\gamma_1 = 2/5$, $\gamma_2 = 3/4$, $\kappa = 3$, $a=1/2$. 
     }
     \label{fig:nonmonotone}
 \end{figure}
\end{comment} 

\subsection{(Non)monotone wave transition and termination, and shock-fronted travelling waves with tails in the aggregation regime} \label{sec:nonmwaves}

We now use numerical continuation to explore the eventual fate of the monotone\footnote{(Non)monotonicity always refers to the density variable $u$.} waves we have constructed. Let us return our focus to the singular limit $\varepsilon = 0$ and choose $\alpha$ (which specifies the middle root of the reaction term) as the continuation parameter, leaving the wavespeed parameter $c$ free. Varying $\alpha$ within the interval $0 < \alpha < \gamma_1$, the corresponding stable fixed point $(\alpha,-c\alpha)$ of the desingularised problem now plays a central role in the curving of the stable manifold $W^s(p_+)$ on $S^l_s$. At the same time, the folded singularity at $u = u_{f_l}$ is now of folded saddle (FS) type (see Table \ref{table:type-fs-bi}), allowing for the existence of canard solutions crossing from the stable middle branch (when $\delta > 0$) of the critical manifold into the saddle-type left branch.\\

The interplay between these equilibria, together with the relevant shock rule defined by $\delta = ac$, determines the termination of monotone waves, as well as the emergence of two new types of solutions: {\it nonmonotone} shock-fronted travelling waves, and shock-fronted travelling wave solutions containing nontrivial {\it slow passage through regions of aggregation/negative diffusion} (i.e., the wave trajectory now has a canard segment, visiting the slow manifold near the middle branch $S_m$ for $O(1)$ periods of `time' with respect to the travelling wave frame coordinate $z$). We now show how all of these new phenomena are organised and explained by global singular bifurcations. \\

Let us first continue the family of monotone shock-fronted travelling waves in the parameter $\alpha$, for small fixed values of $a>0$. For decreasing $\alpha$, the stable manifold $W^s(p_+)$ eventually develops a tangency to the jump curve $\{u = u_{l}\}$ on $S^l_s$ at the end of the shock; see the termination of the blue curve at the green square in Figure \ref{fig:nonmbifdiag}(a), and the corresponding tangency depicted in phase space in Figure \ref{fig:nonmonotone}(a). This global singular bifurcation heralds the termination of the monotone waves and the birth of nonmonotone shock-fronted travelling waves at increasing wavespeeds; see the example in Figure \ref{fig:nonmonotone}(b), corresponding to a point on the red bifurcation curve segment in Figure \ref{fig:nonmbifdiag}(a). This new family of nonmonotone waves can be continued in $\alpha$, where they too eventually terminate in a one-parameter family of folded saddle-to-saddle (FS-to-S) heteroclinic connections on $S^l_s$ connecting $W^s(p_+)$ to $W^u(u_{f_l})$ (the dashed grey curve in Figure \ref{fig:nonmbifdiag}(a)). \\

 We can investigate how this sequence of transitions morphs for different choices of $a$. In Figure \ref{fig:nonmbifdiag}(b), we have plotted the singular heteroclinic bifurcation curves for three different values of $a$.  As we fix larger values of $a>0$, the endpoints of the shock begin to move to the right relatively quickly as the wavespeed $c$ increases, as a result of the dependence of the shock selection rule on the parameter $\delta = ac$. Indeed, for sufficiently large $a>0$, the shock rule changes from `interpolated' type to `viscous' type along the monotone wave branch (i.e. we enter the parameter regime $\delta > \delta_m$ along the corresponding bifurcation curve) before a tangency of $W^s(p_+)$ with the jump curve $\{u = u_{l}\}$ has a chance to form. Therefore, no nonmonotone waves emerge; instead, the monotone family terminates directly on the FS-to-S branch; see the green curve in Figure \ref{fig:nonmbifdiag}(b) corresponding to a singular heteroclinic bifurcation curve for a large fixed value of $a$, and in particular the transition from interpolated shock rules (green solid) to viscous shock rules (green with bubble markers).\\

The transition between these two scenarios implies the existence of an intermediate critical value $a = a_*$, for which the family of monotone waves terminates on the FS-to-S branch at a point $(c_*,\alpha_*)$ precisely when the shock selection rule changes from `interpolated' to `viscous' type (i.e., we have $\delta_m = a_* c_*$). We have numerically identified this intermediate scenario in an example parameter set;  see the magenta curve terminating at the orange dot on the FS-to-S curve in Figure \ref{fig:nonmbifdiag}(b). This termination point can be interpreted as a codimension-three global singular bifurcation at $(a,c,\alpha) = (a_*,c_*,\alpha_*)$, whose unfolding includes the emergence of nonmonotone traveling waves.

\subsubsection{Shock-fronted travelling waves with canard segments}

 \begin{figure}[t]
\centering
    \includegraphics[width=14cm]{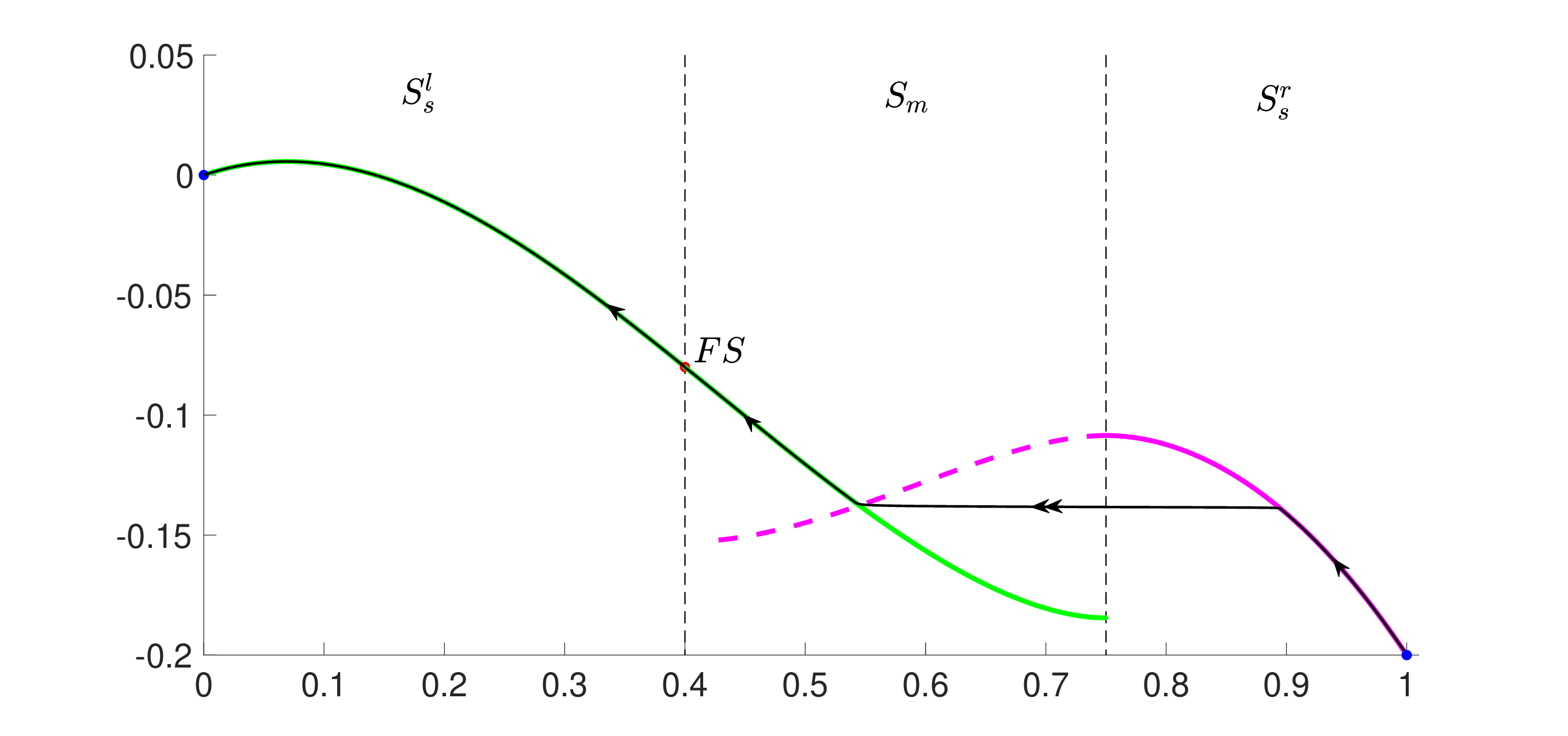}
     \caption{Singular heteroclinic orbit with a singular (vrai) canard segment through a folded saddle (FS) point. Green curve: stable manifold $W^s(p_+)$; magenta solid curve: unstable manifold $W^u(p_-)$; magenta dashed curve: segment of the projection of $W^u(p_-)$ onto $S_m$; black dashed lines: fold curves. Parameter set: $\beta = 1$, $\gamma_1 = 2/5$, $\gamma_2 = 3/4$, $\kappa = 3$, $\alpha \approx 0.068984$, $c = 0.2$.
     }
     \label{fig:interpolcanardshock}
 \end{figure}

We now turn to the existence of shock-fronted travelling wave solutions containing canard segments passing through the FS point. To find such solutions, it is natural to search along the codimension-one family of FS-to-S heteroclinic connections (the grey dashed curve in Figure \ref{fig:nonmbifdiag}(a)). Recalling that $S_m$ is attracting when $\delta > 0$, we search for parameter values along this branch for which the singular stable fast fibre bundle over the unstable manifold of the FS point (with respect to the orientation-reversed desingularised problem) on $S_m$ transversely intersects the unstable fast fibre bundle sitting over $W^u(p_-)$ (in $\mathbb{R}^4$).  A transverse intersection of this type is depicted in Figure \ref{fig:interpolcanardshock} by projecting the corresponding segment of $W^u(p_-)$ onto $S_m$ via the layer flow.\\

These transverse crossings along the FS-to-S curve give rise to shock-fronted travelling waves containing singular (slow) canard segments that connect $S_m$ to $S^l_s$, as shown in Figure \ref{fig:interpolcanardshock}. Such transverse intersections persist robustly under parameter variation, i.e., the codimension-one FS-to-S manifold will intersect the parameter set where such transverse crossings exist in an open neighbourhood. In other words, we still retain a well-defined codimension-one singular heteroclinic bifurcation problem. See the black solid curve segment (labeled TW-C) representing such a subset in Figure \ref{fig:nonmbifdiag}(a).\\

In this scenario, we are interested in the `shock selection rule' (i.e. the set of allowable layer connections) that specifies how $S^r_s$ connects to $S_m$.  The regularisation weighting parameter $a$ continues to play an important role: the composite parameter $\delta = ac$ must be {\it sufficiently large} so that a long enough segment of $W^u(p_-)$ can be projected onto $S_m$ via the layer flow to transversely intersect the unstable manifold of the FS point. This constraint arises due to the relative orientation with which the (un)stable manifolds break apart when crossing the heteroclinic bifurcation curve $\Gamma_-$ transversely in Figure \ref{fig:bif-complete}. In particular, suitable connections from $S^r_s$ to $S_m$ are found in the open region of parameter space to the \textit{top-right} of $\Gamma_-$ only. \\

 \begin{remark}
   Canards may also arise via slow passage through folded node (FN) points, which also appear in our model (see Table \ref{table:type-fs-bi}). However, the stability classification of the critical manifold and the `stability' of the FN singularities (relative to the desingularised problem) do not appear to be compatible with the existence of an invasion shock front (i.e., with $c>0$), which both connects $W^u(p_-)$ to $W^s(p_+)$ and simultaneously contains a canard segment passing through the FN point.
 \end{remark}

\section{Spectral stability of monotone shock-fronted travelling waves} \label{sec:stability}

We now assess the stability of the monotone travelling waves constructed in the previous section, leaving a stability analysis of the nonmonotone travelling waves as well as the travelling waves with canards for future work. Stability results for monotone waves of our RND PDE were initiated in \cite{lizarraga-nonlocal,lizarraga-viscous} for the viscous and `Cahn-Hilliard'-type regularisation limits; here we focus on extending these results to the case of composite regularisation. We adopt the standard approach of determining the spectral stability of the linearised operator $L$ associated with the PDE \eqref{eq:rnd} near the travelling wave. Our first step is to write down a suitable coordinate representation of $L$.\\

Let $\tilde{u}(z,t) = u(z) + \nu e^{\lambda t}p(z) + \mathcal{O}(\nu^2)$ denote a perturbation of a travelling wave $u(z)$ of \eqref{eq:rnd}, where $\lambda$ is the temporal eigenvalue parameter and the variables in the linear term are assumed to separate. Inserting this solution into \eqref{eq:rnd} and collecting terms of linear order in $\nu$, we obtain the equation

\begin{align}\label{eq:linearisedpde}
(f'(u) - \lambda)p &= -(cp_z + (D(u)p)_{zz} + \varepsilon a (\lambda p_{zz} - cp_{zzz})-\varepsilon^2 p_{zzzz}).
\end{align}

Defining linearised variables $y := (p,q,r,s)$ corresponding to the nested derivatives on the right (analogously to what is done for the travelling wave system \eqref{eq:rnd-sys-1-slow}), we arrive at the following nonautonomous linear system:
\begin{equation} \label{eq:eigsystemslow}
\begin{aligned}
\varepsilon \dot{p} &= q\\
\varepsilon \dot{q} &= (D(u)+\varepsilon a \lambda)p + s - \delta q\\
\dot{r} &= (f'(u)-\lambda)p\\
\dot{s} &= r + cp,
\end{aligned}
\end{equation}
or more compactly, 

\begin{align} \label{eq:eigsystemslowcompact}
\dot{y} &= M(z,\lambda,\varepsilon) y,
\end{align}

where $M(z,\lambda,\varepsilon)$ is a convenient matrix representation of $L$. We highlight the terms $\varepsilon a \lambda p$ and $-\delta q$ arising due to the viscous relaxation contribution.\\

By general theory (see e.g. \cite{kapitula1998stability}), the spectrum $\sigma(L)$ of $L$ (i.e. the set of $\lambda$ in the complex plane, where $L-\lambda$ is not invertible on $L^2(\mathbb{R})$) can be decomposed into its {\it point} and {\it essential} spectrum $\sigma(L) = \sigma_p(L) \cup \sigma_c(L)$. Our task is to ensure that the spectrum  is bounded within the left half complex plane, except for an eigenvalue at the origin that necessarily exists due to translational invariance. We must also check that this translational eigenvalue is simple. \\

\subsection{Essential spectrum and sectoriality}

We first give a brief overview of recent results for the essential spectrum.  In \cite{lizarraga-viscous}, it was shown that for each $a\geq 0$, the essential spectrum is bounded well inside the left-half plane. To briefly summarise the approach, the {\it Fredholm borders} of $\sigma_e(L)$ are computed by tracking changes in the Fredholm index of  the asymptotically constant matrices $M_{\pm}(\lambda,\varepsilon) := \lim_{z\to\pm \infty}M(z,\lambda,\varepsilon)$, which characterise the hyperbolic dynamics near the tails of the wave. We obtain the following parametrisations for the dispersion relations (with $k \in \mathbb{R}$):

 \begin{align}\label{eq:disp}
     \lambda_\pm(k) = \frac{f'(u_\pm) - D(u_\pm) k^2 - \varepsilon^2 k^4}{1 + a \varepsilon k^2} + i c k. 
 \end{align}

% \RM{I added the following paragraph, which I think is necessary to finish the argument. It doesn't have to go here though. } 

We first verify that the essential spectrum lies entirely within the left-half plane. This can be seen by noting that the denominator of the real part of $\lambda_\pm(k)$ in \eqref{eq:disp} is strictly positive, and then applying Descartes' rule of signs to the numerator, noting that $f'(u_\pm) <0$, and $D(u_\pm)$ and $\varepsilon$ are both positive. Since the numerator is an even polynomial, there are no real roots of the real part of $\lambda_\pm(k)$. Therefore, the essential spectrum is entirely contained in the left-half plane.\\

 Let us  recall the definition of sectoriality of a linear operator from Definition 1.3.1 in \cite{henry80}. A linear operator $A$ in a Banach space $(X,||\cdot|| )$ is {\it sectorial} if it is a closed, densely defined operator such that for some $\phi \in (0,\pi/2)$, some real $a$, and some $M\geq 1$, 
 \begin{enumerate}
     \item[(i)] the sector $S_{a,\phi}:= \{ \lambda \in \mathbb{C}: \phi \leq |\text{arg}(\lambda-a)| \leq \pi,\,\lambda \neq 0\}$ lies inside the resolvent set of $A$, and
     \item[(ii)] $||\lambda - A||^{-1}\leq M/|\lambda-a|$.
 \end{enumerate}
 
Sectoriality allows us to conclude \textit{nonlinear} stability from spectral stability under mild regularity conditions on the operator along travelling wave solutions, i.e. we can deduce the existence of a neighbourhood of initial conditions of the shock-fronted travelling wave tending to a translate of the wave exponentially quickly as time increases; see Ex. 6 in Sec. 5.1.1 of \cite{henry80}. Simultaneously, this property allows us to deduce the existence of a maximal compact contour $K$ in the complex plane containing all of the point spectrum inside it.\\

We showed that the essential spectrum is asymptotically vertical (obstructing sectoriality) in the viscous relaxation limit \eqref{eq:varnad} (see \cite{lizarraga-viscous}), whereas the linearised operator \textit{is} sectorial in the `Cahn-Hilliard' regularisation limit \eqref{eq:charnad}, corresponding to setting $a = 0$ (see \cite{lizarraga-nonlocal}).  In this section, we verify the sector estimate (i) in the list above for $a\geq 0$ by using a geometric compactification technique, i.e. a rescaling `at infinity,' as we now describe. 

\begin{remark} \label{rem:sectorissue}
The resolvent estimate (ii) in the list above was verified for the $a = 0$ case with the aid of spectral perturbation theorems; to summarise, sectoriality of a fourth-order spatial derivative operator was retained under perturbation by lower-order terms (see \cite{lizarraga-nonlocal}). In the present case, this argument is more difficult to adapt due to the presence of the mixed-derivative viscous term $u_{xxt}$ for $a > 0$, and so we defer a complete functional-analytic treatment of this issue for future work.  
\end{remark}

We follow the general approach of the proof of Prop 2.2 in Sec. 5-B of \cite{agj90}: we identify a suitable rescaling of the linearised variables in \eqref{eq:eigsystemslow} in powers of $|\lambda|$, such that the $|\lambda| \to \infty$ limiting system takes an especially simple form. We use this simple form to determine that  there is no unstable-to-stable connection made (i.e. there can be no spectrum for sufficiently large values of $|\lambda|$), as long as $\lambda$ lies inside a suitable sector specified by the constraint on $\text{arg}(\lambda)$.  \\

%\IL{Beginning of edits.}

It turns out that the appropriate choice of rescaling weights depends on whether $a = 0$ or $a > 0$. This dichotomy is not unexpected in view of the distinct asymptotic behaviour of the dispersion relations \eqref{eq:disp} in each case: the relations flair to the left with quartic growth when $a=0$,  but with only quadratic growth when $a>0$. \\

We first consider the case $a=0$, repeating the analysis in \cite{lizarraga-nonlocal} but with a focus on rescaling the representation \eqref{eq:eigsystemslow}. Define the rescaled quantities

\begin{equation*}
\tilde{p} = p,\hspace{0.3cm} \tilde{q} =q/|\lambda|,\hspace{0.3cm}\tilde{r} =r/|\lambda|^{3},\hspace{0.3cm} \tilde{s} =s/|\lambda|^{2},\hspace{0.3cm}\tilde{z} = z|\lambda|.
\end{equation*}

Writing the eigenvalue system \eqref{eq:eigsystemslow} in terms of the rescaled variables $(\tilde{p},\tilde{q},\tilde{r},\tilde{s})$ with rescaled `time' $\tilde{z}$, and then taking the limit $|\lambda| \to \infty$, we arrive at the hyperbolic, constant coefficient linear system

\begin{equation} \label{eq:rescaledeig}
    \begin{aligned}
        \varepsilon \dot{\tilde{p}} &= \tilde{q}\\
        \varepsilon \dot{\tilde{q}} &=  \tilde{s}\\
         \dot{\tilde{r}} &= -e^{i \text{arg}\,\lambda} \tilde{p}\\
        \dot{\tilde{s}} &= \tilde{r}.
    \end{aligned}
\end{equation}

\begin{figure}
    \centering
    \includegraphics[width=9cm]{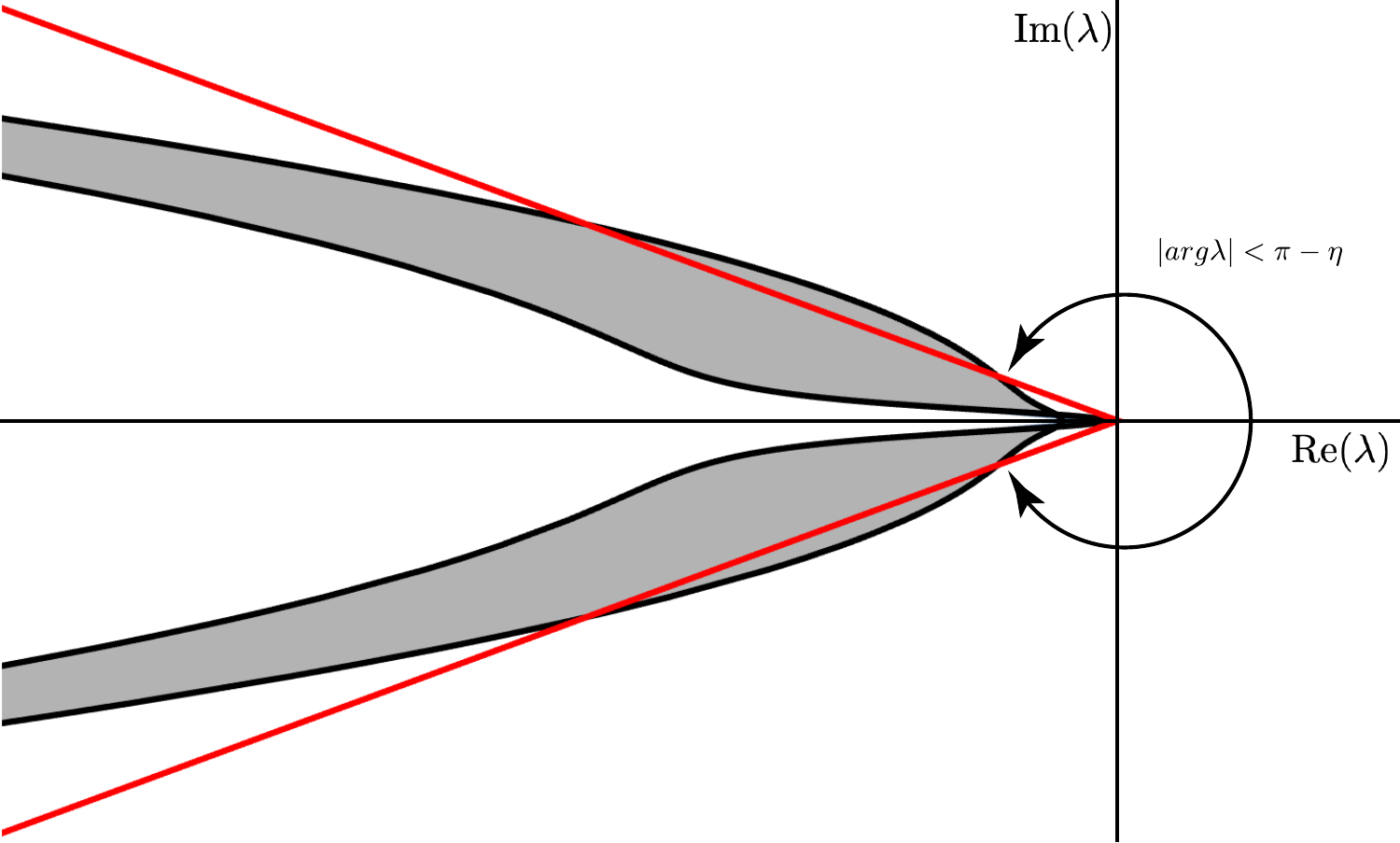}
    %(b) \includegraphics[width=7cm]{Sectorial-marked-numbers.pdf}
    
\caption{A schematic of the dispersion relations (black) from \eqref{eq:disp} enclosing the essential spectrum (grey), and an asymptotically bounding sector (red online). A bounding angle $\eta$ is chosen such that the viscous relaxation contribution causes the dispersion relations to `flair out' above the sector. The contributions from the Fourier transform of the nonlocal terms eventually take over for sufficiently large values of $|\lambda|$, and the dispersion relations return to being on the left of the bounding sector. } 
    \label{fig:sector}
\end{figure}

The (spatial) eigenvalues of \eqref{eq:rescaledeig} are obtained from the roots of the quartic characteristic polynomial 

\begin{equation*}
e^{i\, \text{arg}\,\lambda} + \varepsilon^2 \mu^4 = 0,
\end{equation*}

which can be solved explicitly to obtain the expressions

$$
\mu = \frac{e^{i (\text{arg}\,\lambda+m\pi)/4}}{\sqrt{\varepsilon}},\hspace{0.3cm} m=0,\,1,\,2,\,3.
$$

Note that this matches the `large-scale' eigenvalues in  \cite{lizarraga-nonlocal} (c.f. eq. (20)), which were calculated using a different choice of rescaling weights. We also highlight that the asymptotic eigenvalue problem is independent of the wavespeed parameter.\\

Choose any $\eta \in (0,\pi/2)$. Then for each fixed $\varepsilon > 0$, two of these eigenvalues have strictly positive real part and the remaining two have strictly negative real part whenever $|\text{arg}(\lambda)| < \pi - \eta$. Since \eqref{eq:rescaledeig} is autonomous, the corresponding unstable subbundle forms an attractor near to which solutions of the eigenvalue problem remain close by for all $y$, i.e. for sufficiently large $|\lambda|$ within the sector specified by the choice of $\eta$, there can be no contribution to the spectrum of the linearised operator. Furthermore, the eigenvalues remain well separated as $\varepsilon \to 0$.\\

Now we consider the case $a>0$. In order to control the additional $\lambda$-dependent term in \eqref{eq:eigsystemslow} when $|\lambda|$ grows large, we must also include $|\lambda|$ in the rescaling. We introduce an auxiliary rescaling parameter $\sigma>0    $ and write

\begin{align}\label{eq:scalings}
    \tilde{p} = p,\hspace{0.3cm} \tilde{q} =q\sigma,\hspace{0.3cm}\tilde{r} =r\sigma^3,\hspace{0.3cm} \tilde{s} =s\sigma^{2},\hspace{0.3cm}\tilde{z} = z/\sigma,\hspace{0.3cm} |\lambda| = 1/\sigma^{2}.
\end{align}

With respect to this scaling, we have

\begin{equation} \label{eq:rescaledeig2}
    \begin{aligned}
        \varepsilon \dot{\tilde{p}} &= \tilde{q}\\
        \varepsilon \dot{\tilde{q}} &= \left(\sigma^2  D(u) + \varepsilon a e^{i\,\text{arg}(\lambda)}\right)\tilde{p} + \tilde{s} -\sigma \delta  \tilde{q}\\
         \dot{\tilde{r}} &= (\sigma^4 f'(u) - \sigma^2 e^{i\,\text{arg}(\lambda)})\tilde{p}\\
        \dot{\tilde{s}} &= \tilde{r} + \sigma^3 c \tilde{p}.
    \end{aligned}
\end{equation}

We are concerned with the dynamics near the limit $\sigma\to 0$. As in the purely nonlocal case, the limiting linear system has constant coefficients, but it is now {\it nonhyperbolic}:

\begin{equation} \label{eq:rescaledeig2lim}
    \begin{aligned}
        \varepsilon \dot{\tilde{p}} &= \tilde{q}\\
        \varepsilon \dot{\tilde{q}} &= \varepsilon a e^{i\, \text{arg}\,\lambda}\tilde{p} +  \tilde{s}\\
         \dot{\tilde{r}} &= 0\\
        \dot{\tilde{s}} &= \tilde{r},
    \end{aligned}
\end{equation}

with eigenvalues

\begin{align} \label{eq:mu0th}
    \mu = 0,\,0,\,\pm \sqrt{\frac{a}{\varepsilon}} e^{i\, \text{arg}(\lambda)/2}.
\end{align}

The analysis here is more delicate than in the previous case: we do not have access to an attractor (an unstable 2-plane bundle) in the large $|\lambda|$ limit, and the weak (un)stable directions degenerate to a two-dimensional center subspace. We resort to standard perturbation theory. The eigenvalues of \eqref{eq:rescaledeig2} can be determined to arbitrary order in $\sigma$ (i.e. $1/|\lambda|^{1/2}$); we find that the pair of zero eigenvalues perturbs as

\begin{align}\label{eq:mu1st}
    \mu = \pm \sqrt{\frac{1}{a\varepsilon}}\, \frac{1}{|\lambda|^{1/2}} + \mathcal{O}\left( \frac{1}{|\lambda|}\right).
\end{align}

We remind the reader that the system \eqref{eq:rescaledeig2} is nonautonomous, but the relevant contributions from $D(u)$ and $f'(u)$ are bounded and do not affect the signs of the two smaller eigenvalues at leading order. An invariant attractor over the unstable subbundle can be constructed explicitly by projectivizing the system and then using the theory of relatively invariant sets for nonautonomous systems (see Sec. B in \cite{gardner-jones}), but we avoid these technical details here. The argument for the sector estimate (i) in terms of bounding angles $\eta$ then follows as in the `pure Cahn-Hilliard regularisation' case. See Fig. \ref{fig:sector} for a depiction of a bounding sector in relation to the essential spectrum.

\begin{remark}
The scaling weights $n_p,\,n_q,\,\dots$ in $\tilde{p} = p/|\lambda|^{n_p},\,\tilde{q} = q/|\lambda|^{n_q},$ etc. can be chosen such that the limiting system \eqref{eq:rescaledeig} is autonomous, and so that the exponent of $|\lambda|$ balances to zero in the slow equation. This rescaling procedure can be interpreted geometrically as an extended Poincar\'{e} compactification of the vector field \eqref{eq:eigsystemslow} `at infinity'; see \cite{Wechselberger_2002}.\\

Our choice of weights is also consistent with the scaling derived using the method of dominant balance in the WKB approximation of \eqref{eq:linearisedpde} (see \cite{benderorszag}). Furthermore, the corresponding eigenvalue expansions in \eqref{eq:mu0th}--\eqref{eq:mu1st} match the output of the WKB calculations.
\end{remark}

We have shown that the appropriate cone estimate required for sectoriality holds for each $a\geq 0$, extending the result in \cite{lizarraga-nonlocal}. Assuming that the resolvent estimate can also be verified for each $a \geq 0$ (see Remark \ref{rem:sectorissue}), this result has the following interesting implication: it is possible to retain sectoriality (and hence nonlinear stability) for travelling waves containing viscous shocks, when viscous relaxation is counterbalanced by `Cahn-Hilliard type' regularisation. This is in contrast to the `pure viscous relaxation' limit in \cite{lizarraga-viscous}, where the asymptotically vertical nature of the essential spectrum obstructs sectoriality of the operator.

\subsection{Computation of the point spectrum}

In both the viscous relaxation limit \eqref{eq:varnad} and the `pure Cahn-Hilliard' regularisation limit \eqref{eq:charnad}, there exist only two eigenvalues in the point spectrum for sufficiently small $\varepsilon>0$: the simple translational eigenvalue $\lambda_0 = 0$, and another simple real eigenvalue $\lambda_1 \in (\text{max}(R'(0),R'(1)),0)$ that does not destabilise the travelling wave (see \cite{lizarraga-nonlocal,lizarraga-viscous}). In this section, we augment these results by sampling the point spectrum of the corresponding linearised operator for $\delta > 0$ and small values of $\varepsilon > 0$. 

\begin{remark} \label{rem:spectralparams}
Throughout this section we fix the parameter set $\beta = 6$, $\gamma_1 = 7/12$, $\gamma_2 =3/4$, $\kappa = 5$, $\alpha = 1/5$ in order to make concrete computations; this parameter set is also used as a running example in \cite{li2021,lizarraga-nonlocal,lizarraga-viscous}. However, we emphasize that there is nothing particularly special about this parameter set in the spectral stability calculation, and the following approach that we develop can be applied in general.
\end{remark}
Let us outline the strategy to compute the point spectrum. For each $\lambda \in \mathbb{C}$ to the right of the essential spectrum, we can define an unstable complex $2$-plane bundle $\varphi_-(z,\lambda,\varepsilon)$ extending from the unstable subspace of the saddle point at $u=1$, resp. a stable complex $2$-plane bundle $\varphi_+(z,\lambda,\varepsilon)$ extending from the stable subspace of the saddle point at $u=0$, by using the eigenvalue problem \eqref{eq:eigsystemslow}.  A (spatial) eigenvalue $\lambda \in \sigma_p(L)$ is found whenever $\varphi_-$ and $\varphi_+$ have a nontrivial intersection at some value $z$; see \cite{agj90} for details.\\

In view of this geometric characterization for the spatial eigenvalues, we will use a {\it Riccati-Evans function} to compute these intersections. The eigenvalue problem \eqref{eq:eigsystemslow} induces a nonlinear flow on the Grassmannian $Gr(2,4)$ of complex $2$-planes in $\mathbb{C}^4$. On a suitable coordinate patch of $Gr(2,4)$, the nonlinear flow is defined using a {\it matrix Riccati equation} of the form
\begin{equation}
\label{eq:matrixriccati}
\begin{aligned}
W' &= C + DW -WA - WBW,
\end{aligned}
\end{equation} 

where $W$ is a complex-valued $2\times 2$ matrix variable defined using frame coordinates for the $2$-planes, and the $2\times 2$ matrices $A,B,C,D$ are defined via a block decomposition of the linear operator $M$ in \eqref{eq:eigsystemslowcompact}:
\begin{equation}
\label{eq:blockdecomp}
\begin{aligned}
M
 &=\left(
    \begin{array}{c|c}
      A & B\\
      \hline
      C & D
    \end{array}
    \right).
\end{aligned}
\end{equation} 

See \cite{harley2019spectral} and \cite{vlsmvt10} for the derivation of \eqref{eq:matrixriccati}, and \cite{lizarraga-nonlocal} for the specific construction in the case of `pure Cahn-Hilliard-type' regularisation \eqref{eq:charnad}. \\

In terms of the representation \eqref{eq:riccatievans} of the projectivised dynamics, $\varphi_-$ is equivalent to the unique trajectory of \eqref{eq:matrixriccati} that converges to the unstable subspace of the saddle point at $u=1$ as $z \to -\infty$; there is also an analogous characterisation of $\varphi_+$.   We can now formulate a shooting problem defined on a suitable cross section that intersects the travelling wave transversely, say $\Sigma = \{u = 0.7\}$, with the corresponding intersection point $z_0 \in \Sigma$. The unstable bundle $\varphi_-(z,\lambda,\varepsilon)$ is flowed forward from $u = 1$ and the stable bundle $\varphi_+(z,\lambda,\varepsilon)$ is flowed backward from $u=0$. Suppressing the notation for $\varepsilon$, the Riccati-Evans function $E(z_0,\lambda)$ is defined by

\begin{align} \label{eq:riccatievans}
E(z_0,\lambda)  = \det(\varphi_+(z_0,\lambda) - \varphi_-(z_0,\lambda)).
\end{align}

We can then find eigenvalues $\lambda \in \sigma_p(L)$ by locating zeroes of $E(z_0,\lambda)$; see \cite{harley2019spectral}. Using the argument principle, we locate these zeroes by computing the winding number of $E$ along suitably chosen contours in the complex plane and to the right of the essential spectrum.\\

\begin{figure}
\begin{subfigure}{0.5\textwidth}
\includegraphics[width=8.75cm]{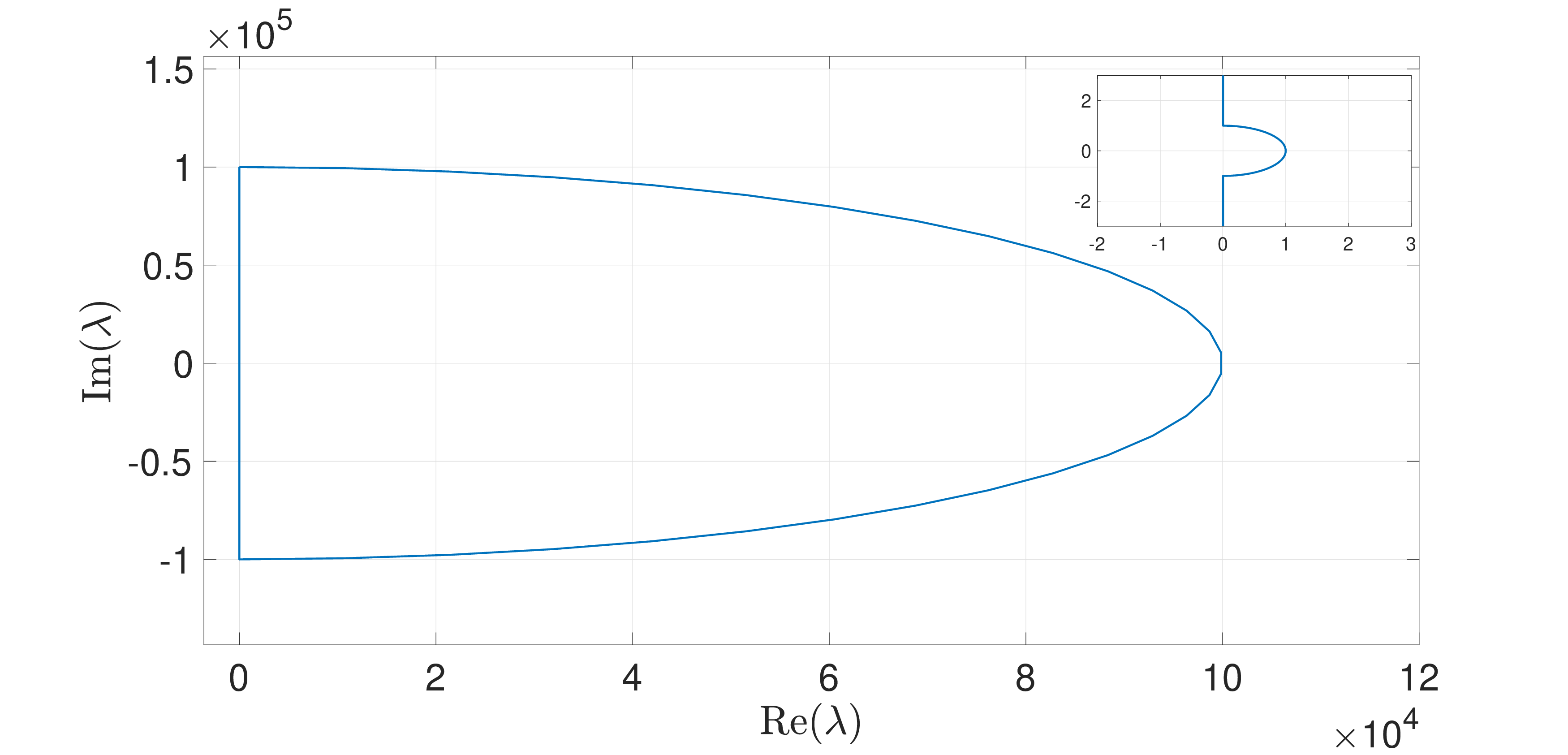} 
\caption{}
\end{subfigure}
\begin{subfigure}{0.5\textwidth}
\includegraphics[width=8.75cm]{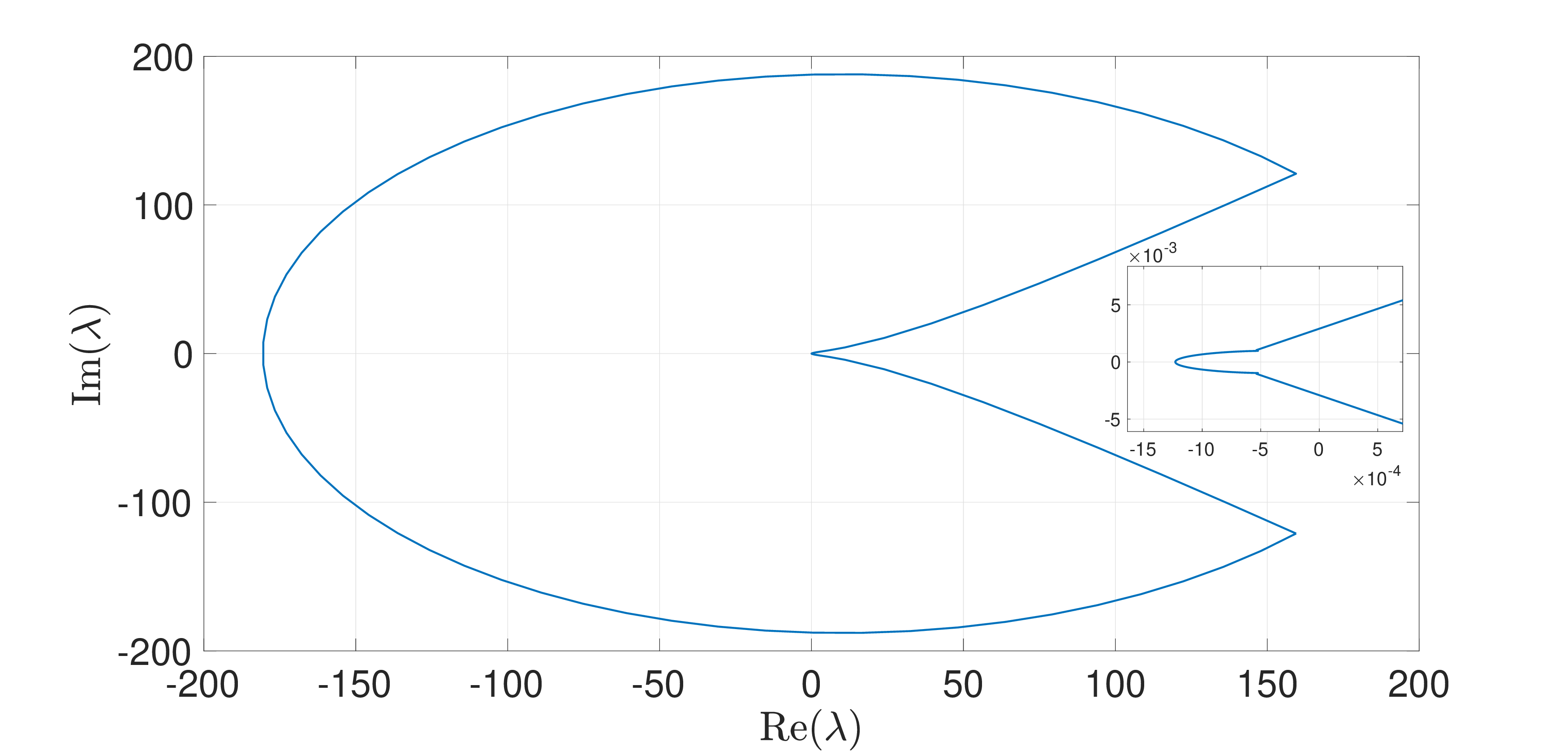}
\caption{}
\end{subfigure}

    \caption{(a) Semicircular contour $K$ of radius $10^5$ along which the Riccati-Evans function \eqref{eq:riccatievans} is evaluated. A small detour avoids the translational eigenvalue at the origin (inset).  (b) Image $E(z_0,K)$ in the complex plane. Note that the image does not wind around the origin (inset). Parameter set as in Fig. \ref{fig:fullhet}.} 
    \label{fig:riccatievans}
\end{figure}

We use a semicircular contour $K$ with increasingly large radii $R = 10^3,\,10^4,\,10^5$ opening in the right half complex plane,  with a small semicircular detour that avoids the translational eigenvalue at the origin (see Fig. \ref{fig:riccatievans}(a)). We sampled the interval $[0,a_m]$ with a grid of 100 equally spaced points, and we fix $\varepsilon = 10^{-4}$. For each $a$ in this sample, we evaluate $E(z_0,\lambda)$ around $K$ to find a winding number of 0 (see e.g. Fig. \ref{fig:riccatievans}(b)). Thus we obtain strong numerical evidence that there is no point spectrum within $K$, and hence that the corresponding family of shock-fronted travelling waves remains nonlinearly stable for each $a \in [0,a_m]$ and for each sufficiently small $\varepsilon>0$. \\

We also investigated the zeroes of $E(z_0,\lambda)$ along the real line. For each sampled parameter value $\delta$, we found evidence of only one simple translational eigenvalue $\lambda_0 = 0$ and one other simple real eigenvalue $\lambda_1 \approx -0.8$. We note that this result is entirely consistent with the calculation of the point spectrum in the `pure Cahn-Hilliard type' and `purely viscous' regularisation cases \cite{lizarraga-nonlocal,lizarraga-viscous}; indeed, these eigenvalues are accounted for by the corresponding reduced \textit{slow eigenvalue problem} defined on the critical manifold. As depicted in Fig. \ref{fig:slowbifdiag}(b), singular heteroclinic connections are formed over a small range of wavespeeds $c \in [0.1973,0.1993]$ as $a \geq 0$ is varied. This in turn slightly perturbs the dynamics of the reduced eigenvalue problem along the singular heteroclinic connection. As a consequence of these small variations, the (singular limit of the) secondary eigenvalue $\lambda_1$ moves continuously within the small interval $ [-0.81,-0.79]$ on the real line as $a$ is varied. The key point that we would like to emphasize is that a slow eigenvalue problem (defined for $\varepsilon = 0$) continues to approximate the eigenvalues of the `full' problem (defined for $0 < \varepsilon \ll 1$) when $a > 0$. \\

\begin{remark} 
It is interesting to ask whether the translational eigenvalue can bifurcate under variation of the regularisation parameter $a$ via e.g. transcritical crossings with the secondary eigenvalue $\lambda_1 = \lambda_1(a)$.  In other words, is it possible to destabilize a (regularised) shock-fronted travelling wave of \eqref{eq:vcharnad} just by re-weighting the regularisation?  \\

We argue that such a destabilization scenario is not possible for families of monotone travelling waves. We showed using comparison methods in \cite{lizarraga-viscous} that if a projectivised solution along a singular heteroclinic orbit of the reduced eigenvalue problem has no winds at $\lambda = \lambda_0 \in \mathbb{R}$, then no further winds are generated for each real $\lambda > \lambda_0$, and furthermore, there are no nontrivially complex eigenvalues---i.e. there are no more eigenvalues to the right of $\lambda_0$. Now suppose that the translational eigenvalue had a transcritical bifurcation for some critical value $a = a_c$ such that $\lambda_1(a) > 0$ for $a > a_c$. Then it must be true for nearby values of  $a> a_c$ that the variational solution winds around at least once, i.e., the singular heteroclinic of \eqref{eq:rnd-reduced-1} in $(u,v)$-space makes a full revolution. 
\end{remark}

\section{Concluding remarks}

    We have given a comprehensive description of shock-fronted travelling wave solutions arising in regularised RND PDEs of the form \eqref{eq:vcharnad}, focusing on a weighted composite of two well-known regularisations---viscous relaxation and a `Cahn-Hilliard type' fourth-order spatial regularisation. We highlight the role played by {\it symmetry}, in both the layer and reduced problems (see Sec. \ref{sec:singhetconstruction}), in enforcing the existence of a locus of monotone standing waves. This set serves as a starting point for the continuation of a codimension-one bifurcation manifold of singular heteroclinic connections in parameter space. \\
    
    We track these monotone connections until they terminate, identifying the global singular bifurcations that play a key role in their termination. This bifurcation analysis simultaneously explains their transition into new kinds of shock-fronted travelling wave solutions: nonmonotone waves and waves with canard segments. It is an interesting problem to determine parameter sets such that the (non)monotone travelling wave bifurcation branches intersect the travelling wave-with-canard branch in a singular bifurcation diagram (e.g. intersections of the red and/or blue (N)MTW curves with the solid black TW-C curve, in a diagram drawn as in Figure \ref{fig:nonmbifdiag}(a)). In such a scenario, these intersection points in parameter space would correspond to the simultaneous coexistence of two distinct types of travelling wave solutions in phase space, and could therefore be interpreted as codimension-two branch switching bifurcations. \\

    We highlight the role played by Melnikov theory in the construction of the heteroclinic bifurcation manifold---in particular, we provide a new adaptation to piecewise-smooth planar systems (see Lemma \ref{lem:pwmelnikov} and Sec. \ref{sec:pwmelnikov}). This new piecewise-smooth formula motivates the development of a more general piecewise-smooth Melnikov theory for autonomous dynamical systems in $\mathbb{R}^n$.\\

    We also investigated the spectral stability of the shock-fronted travelling wave solutions in our RND PDE, finding numerical evidence that our model admits spectrally stable monotone shockwaves. It is worthwhile to investigate the spectral stability of the nonmonotone and canard-type shockwave families we have identified. We conjecture that nonmonotonicity may be a geometric mechanism to destabilize these families. Let us highlight that while the canard-type shockwaves are apparently monotone near the tails (see Fig.  \ref{fig:interpolcanardshock}), fast connections from $S^r_s$ to $S_m$ can now introduce winding in the shock layer: since $S_m$ contains subsets where the nontrivial layer eigenvalues become complex, a shock departing from $S^r_s$ can potentially rotate infinitely often as it approaches $S_m$. A thorough investigation of this scenario, from the point-of-view of slow-fast splittings of the eigenvalue problem (as done in  \cite{lizarraga-nonlocal,lizarraga-viscous}, and this paper), is a topic of future work.\\

    Finally, we wish to point out that other, unrelated types of high-order regularisations appear in the shockwave literature. For example, numerical regularisations have been derived by using the technique of {\it modified PDE} analysis applied to finite difference schemes; see \cite{Anguige_2008}. We expect that GSPT will continue to provide a flexible and powerful approach to explain the dynamics of RND PDEs subject to a wide class of high-order regularisations.\\

\subsection*{Acknowledgement}

This work has been funded by the Australian Research Council DP200102130 grant.

\bibliographystyle{plain}
\bibliography{regbib2}

\begin{thebibliography}{10}

\bibitem{agj90}
J.~Alexander, R.~A. Gardner, and C.~K. R.~T. Jones.
\newblock A topological invariant arising in the stability analysis of
  traveling waves.
\newblock {\em J. Reine Angew. Math}, 410:167--212, 1990.

\bibitem{Anguige_2008}
K.~Anguige and C.~Schmeiser.
\newblock A one-dimensional model of cell diffusion and aggregation,
  incorporating volume filling and cell-to-cell adhesion.
\newblock {\em Journal of Mathematical Biology}, 58(3):395--427, jun 2008.

\bibitem{benderorszag}
C.M. Bender and S.A. Orszag.
\newblock {\em Advanced Mathematical Methods for Scientists and Engineers I}.
\newblock Springer New York, 1999.

\bibitem{Browning2020}
Alexander~P. Browning, Wang Jin, Michael~J. Plank, and Matthew~J. Simpson.
\newblock Identifying density-dependent interactions in collective cell
  behaviour.
\newblock {\em Journal of The Royal Society Interface}, 17(165):20200143, 2020.

\bibitem{cahn61}
J.~W. Cahn.
\newblock On spinoidal decomposition.
\newblock {\em Acta. Metall.}, 9(795-801), 1961.

\bibitem{CohenMurray81}
D.S. Cohen and J.D Murray.
\newblock A generalized diffusion model for growth and dispersal in a
  population.
\newblock {\em J. Math. Biology}, 12:237--249, 1981.

\bibitem{fenichel79}
N.~Fenichel.
\newblock Geometric singular perturbation theory.
\newblock {\em J Differential Equations}, 31:53--98, 1979.

\bibitem{landetal10}
A.~E. Fernando, K.~A. Landman, and M.~J. Simpson.
\newblock Nonlinear diffusion and exclusion processes with contact
  interactions.
\newblock {\em Physical Review E}, 81, January 2010.

\bibitem{Fife00}
Paul~C. Fife.
\newblock Models for phase separation and their mathematics.
\newblock 2000.

\bibitem{fisher37}
R.~A. Fisher.
\newblock The wave of advance of advantageous genes.
\newblock {\em Annals of Eugenics}, 7:355--369, 1937.

\bibitem{gardner-jones}
R.~Gardner and C.K.R.T. Jones.
\newblock {Stability of Travelling Wave Solutions of Diffusive Predator-Prey
  Systems}.
\newblock {\em Transactions of the AMS}, 327:465--524, 1991.

\bibitem{granados}
A.~Granados, S.J. Hogan, and T.M. Seara.
\newblock The scattering map in two coupled piecewise-smooth systems, with
  numerical application to rocking blocks.
\newblock {\em Physica D}, (269):1--20, 2014.

\bibitem{harley2019spectral}
K.~E. Harley, P.~van Heijster, R.~Marangell, G.~J. Pettet, T.~V. Roberts, and
  M.~Wechselberger.
\newblock (in)stability of travelling waves in a model of haptotaxis.
\newblock {\em submitted. 20 pages}, 2019.

\bibitem{hvhmpw13}
K.~E. Harley, P.~van Heijster, R.~Marangell, G.~J. Pettet, and
  M.~Wechselberger.
\newblock {Existence of Travelling Wave Solutions for a Model of Tumour
  Invasion}.
\newblock {\em SIAM Journal of Applied Dyn Sys}, 13(1):366--396, 2014.

\bibitem{hvhmpw14}
K.~E. Harley, P.~van Heijster, R.~Marangell, G.~J. Pettet, and
  M.~Wechselberger.
\newblock {Novel solutions for a model of wound healing angiogenesis}.
\newblock {\em Nonlinearity}, 27:2975--3003, 2014.

\bibitem{henry80}
D.~Henry.
\newblock {\em Geometric Theory of Semilinear Parabolic Equations}.
\newblock Number 840 in Lecture Notes in Mathematics. Springer--Verlag, New
  York, 1980.

\bibitem{Hillen2009}
T.~Hillen and K.~Painter.
\newblock A user's guide to pde models for chemotaxis.
\newblock {\em Journal of Mathematical Biology}, 58(1--2):183--217, 2009.

\bibitem{hilliard70}
J.~Hilliard.
\newblock Spinodal decomposition. {\em in} phase transformations.
\newblock {\em Am. Soc. Metals}, pages 497--560, 1970.

\bibitem{hollig1983existence}
Klaus H{\"o}llig.
\newblock Existence of infinitely many solutions for a forward backward heat
  equation.
\newblock {\em Transactions of the American Mathematical Society},
  278(1):299--316, 1983.

\bibitem{Johnston2017}
S.~T. Johnston, R.~E. Baker, D.~L.~S. McElwain, and M.~J. Simpson.
\newblock Co-operation, competition and crowding: a discrete framework linking
  allee kinetics, nonlinear diffusion, shocks and sharp-frontec travelling
  waves.
\newblock {\em Scientific Reports}, 7(42134), 2017.

\bibitem{jonesgsp95}
C.~K. R.~T. Jones.
\newblock Geometric singular perturbation theory.
\newblock {\em in Dynamical Systems, Springer Lecture Notes Math.},
  1609:44--120, 1995.

\bibitem{kapitula1998stability}
T.~Kapitula and B.~Sandstede.
\newblock Stability of bright solitary-wave solutions to perturbed nonlinear
  schr{\"o}dinger equations.
\newblock {\em Physica D: Nonlinear Phenomena}, 124(1):58--103, 1998.

\bibitem{keenersneyd}
J.~Keener and J.~Sneyd.
\newblock {\em Mathematical Physiology}, volume~8 of {\em Interdisciplinary
  Applied Mathematics: Mathemarical Biology}.
\newblock Springer, Hong Kong, 1998.

\bibitem{Keller1971}
E.~F. Keller and L.~A. Segel.
\newblock Travelling bands of chemotactic bacteria: A theoretical analysis.
\newblock {\em J. theor. Biol.}, 30:235--248, 1971.

\bibitem{kukucka}
P.~Kuku\v{c}ka.
\newblock Melnikov method for discontinuous planar systems.
\newblock {\em Nonlinear Analysis}, (66):2698--2719, 2006.

\bibitem{landmanetal03}
K.~A. Landman, G.~J. Pettet, and D.~F. Newgreen.
\newblock Chemotactic cellular migration: Smooth and discontinuous travelling
  wave solutions.
\newblock {\em SIAM Journal of Applied Mathematics}, 63:1666--1681, 2003.

\bibitem{landmanetal08}
K.~A. Landman, M.~J. Simpson, and G.~J. Pettet.
\newblock Tactically-driven nonmonotone travelling waves.
\newblock {\em Physica D}, 237:678--691, 2008.

\bibitem{vlsmvt10}
V.~Ledoux, S.~Malham, and V.~Th\"ummler.
\newblock Grassmannian spectral shooting.
\newblock {\em Mathematics of Computation}, 79:1585--1619, 2010.

\bibitem{li2021}
Y.~Li, P.~van Heijster, M.~J. Simpson, and M.~Wechselberger.
\newblock Shock-fronted travelling waves in a reaction–diffusion model with
  nonlinear forward–backward–forward diffusion.
\newblock {\em Physica D}, 423(132916), 2021.

\bibitem{linwex13}
X.-B. Lin and M.~Wechselberger.
\newblock Transonic evaporation waves in a spherically symmetric nozzle.
\newblock {\em SIAM Journal on Mathematical Analysis}, 46:1472--1504, 2014.

\bibitem{lizarraga-nonlocal}
I.~Lizarraga and R.~Marangell.
\newblock Nonlocal stability of shock-fronted travelling waves under nonlocal
  regularization.
\newblock {\em arXiv:2211.07824 (Preprint)}, 2022.

\bibitem{lizarraga-viscous}
I.~Lizarraga and R.~Marangell.
\newblock Spectral stability of shock-fronted travelling waves under viscous
  relaxation.
\newblock {\em Journal of Nonlinear Science}, 33(82), 2023.

\bibitem{mainietal07}
P.~K. Maini, L.~Malaguti, C.~Marcelli, and S.~Matucci.
\newblock Aggregative movement and front propagation for bi-stable population
  models.
\newblock {\em Mathematical Models and Methods in Applied Sciences},
  17(9):1351--1368, 2007.

\bibitem{mainietal10}
P.~K. Maini, F.~Sanchez-Garduno, and J.~Perex-Velazques.
\newblock A nonlinear degenerate equation for direct aggreation and traveling
  wave dynamics.
\newblock {\em Discrete and Continuous Dynamical Systems Series B}, 13(2),
  March 2010.

\bibitem{marchantetal01}
B.~P. Marchant, J.~Norbury, and J.~A. Sherratt.
\newblock Travelling wave solutions to a haptotaxis-dominated model of
  malignant invasion.
\newblock {\em Nonlinearity}, 14:1653--1671, 2001.

\bibitem{murray02}
J.D. Murray.
\newblock {\em Mathematical Biology I: An introduction}.
\newblock Number~17 in Interdisciplinary Applied Mathematics. Springer, Hong
  Kong, 3rd edition, 2002.

\bibitem{novpego91}
A.~Novick-Cohen and R.~L. Pego.
\newblock Stable patterns in a viscous diffusion equation.
\newblock {\em Transactions of the American Mathematical Society}, 324(1),
  March 1991.

\bibitem{padron03}
V.~Padron.
\newblock Effect of aggregation on population recovery modelled by a
  forward-backward pseudoparabolic equation.
\newblock {\em Transactions of the American Mathematical Society},
  356(7):2739--2756, 2003.

\bibitem{pego89}
R.~L. Pego.
\newblock Front migration in the nonlinear {C}ahn--{H}illiard equation.
\newblock {\em Proc. R. Soc. Lond. A}, 422:261--278, 1989.

\bibitem{penington11}
C.~J. Penington, B.~D. Hughes, and K.~A. Landman.
\newblock Building macroscale models from microscale probablistic models: A
  general probablistic approach for nonlinear diffusion and multispecies
  phenomena.
\newblock {\em Physical Review E}, 84, October 2011.

\bibitem{P2007}
Mathieu Poujade, Erwan Grasland-Mongrain, A~Hertzog, J~Jouanneau, Philippe
  Chavrier, Beno{\^\i}t Ladoux, Axel Buguin, and Pascal Silberzan.
\newblock Collective migration of an epithelial monolayer in response to a
  model wound.
\newblock {\em Proceedings of the National Academy of Sciences},
  104(41):15988--15993, 2007.

\bibitem{smoller83}
J.~Smoller.
\newblock {\em Shock waves and reaction diffusion equations}, volume 258 of
  {\em Grundlehren der mathematischen Wissenshaften}.
\newblock Springer-Verlag, 1982.

\bibitem{sz-transverse}
P.~Szmolyan.
\newblock Transversal heteroclinic and homoclinic orbits in singular
  perturbation problems.
\newblock {\em J Differential Equations}, 92:252--281, 1991.

\bibitem{szmwex2001}
P.~Szmolyan and M.~Wechselberger.
\newblock Canards in $\mathbb{R}^3$.
\newblock {\em J Differential Equations}, 177:419--453, 2001.

\bibitem{Vanderbauwhede_1992}
A.~Vanderbauwhede.
\newblock Bifurcation of degenerate homoclinics.
\newblock {\em Results in Mathematics}, 21(1-2):211--223, mar 1992.

\bibitem{Wechselberger_2002}
M.~Wechselberger.
\newblock Extending {M}elnikov theory to invariant manifolds on non-compact
  domains.
\newblock {\em Dynamical Systems}, 17(3):215--233, aug 2002.

\bibitem{wex20}
M.~Wechselberger.
\newblock {\em Geoemtric Singular Perturbation Theory beyond the standard
  form}.
\newblock Number~7 in Frontiers in Applied Dynamical Systems: Reviews and
  Tutorials. Springer, Cham, 2020.

\bibitem{wechpet10}
M.~Wechselberger and G.~J. Pettet.
\newblock Folds, canards and shocks in advection-reaction-diffusion models.
\newblock {\em Nonlinearity}, 23:1949--1969, 2010.

\bibitem{whitham1999}
G.B. Whitham.
\newblock {\em Linear and Nonlinear Waves}.
\newblock Wiley, New York, 1st edition, 1999.

\bibitem{WITELSKI199527}
T.P. Witelski.
\newblock Shocks in nonlinear diffusion.
\newblock {\em Applied Mathematics Letters}, 8(5):27--32, 1995.

\bibitem{Witelski96}
T.P. Witelski.
\newblock The structure of internal layers for unstable nonlinear diffusion
  equations.
\newblock {\em Studies in Applied Mathematics}, 96:277--300, 1996.

\end{thebibliography}

%%%%%%%%%%%%%%%%%%%%%%%%
%APPENDIX
%%%%%%%%%%%%%%%%%%%%%%%%
\newpage
\appendix

\section{Melnikov theory for autonomous vector fields} 

One of the main analytical tools that deals with the existence and bifurcation of heteroclinic orbits in dynamical systems is known as {\em Melnikov theory}. We follow the treatment of Vanderbauwhede \cite{Vanderbauwhede_1992} for sufficiently smooth autonomous dynamical systems in arbitrary dimensions; see also the extension to heteroclinic problems on non-compact domains in \cite{Wechselberger_2002}. Here we provide a succinct summary of this theory for autonomous problems, tailored towards the heteroclinic orbit analysis of the layer problem in section~\ref{sec:layer-het}. \\

We then adapt Melnikov theory to the case of planar piecewise-smooth dynamical systems, which is needed to show the existence of  singular heteroclinic orbits in section~\ref{sec:concat}. The Melnikov method has been adapted to the piecewise-smooth setting in different nonsmooth contexts, with an emphasis on time-periodic perturbations of planar vector fields with either Hamiltonian or trace-free structure (see e.g. \cite{granados,kukucka}); however, the following Vanderbauwhede-style presentation for piecewise smooth autonomous vector fields is new, to the best of our knowledge.

\subsection{The smooth case} \label{appendix:smooth}
\label{sec:smoothmelnikov}
We work with a system of the form
\begin{align} \label{eq:melnikoveq}
x' &= h(x,\mu)
\end{align}

with $x = x(t)\in \mathbb{R}^n$, $n\ge 2$, $h$ sufficiently smooth, and the parameters are denoted by $\mu\in \mathbb{R}^{m}$, $m\ge 1$. 
We assume the existence of a unique heteroclinic connection $\Gamma = \{\gamma(t) \in \mathbb{R}^n: t \in \mathbb{R}\}$ between saddle equilibria $p_-$ and $p_+$ for some $\mu=\mu_0$, i.e.,  $\Gamma = W^u_{\mu_0}(p_-) \cap W^s_{\mu_0}(p_+)$ is the one-dimensional intersection of the unstable manifold $W^u_{\mu_0}(p_-)$ of $p_-$ and the stable manifold $W^s_{\mu_0}(p_+)$ of $p_+$ with $\dim W^u_{\mu_0}(p_-)=l_u$, $1\le l_u<n$, and  $\dim W^s_{\mu_0}(p_+)=l_s$, $1\le l_s<n$.\footnote{For the general setup with possible higher dimensional heteroclinic connections we refer to \cite{Vanderbauwhede_1992,Wechselberger_2002}.} We further assume that $l_u+l_s-1<n$. Otherwise, the intersection $W^u_{\mu_0}(p_-) \cap W^s_{\mu_0}(p_+)$ is transverse implying the persistence of this heteroclinic orbit for nearby $\mu$-values.\footnote{Note that these stable and unstable manifolds $W^u_{\mu_0}(p_-)$ and $W^s_{\mu_0}(p_+)$ persist for nearby $\mu$-values due to the robustness of the hyperbolic saddle equilibria. With a slight abuse of notation, we will denote these perturbed saddle equilibria by $p_\pm$ independent of $\mu$.}\\

We define a suitable cross section $\Sigma$ of the unique heteroclinic orbit $\Gamma$. Without loss of generality we assume the intersection of $\Gamma$ with $\Sigma$ occurs at $t=0$.

The main task is to measure the distance  between the stable and unstable manifolds $W^u_{\mu}(p_-)$ and $ W^s_{\mu}(p_+)$ in this cross section $\Sigma$ for nearby $\mu$-values. Melnikov theory defines a corresponding distance function\footnote{In higher dimensional problems, the distance function is vector-valued.} $\Delta = \Delta(\mu)$, noting that $\Delta(\mu_0)=0$. In the following, we derive computable formulas for $\Delta = \Delta(\mu)$.\\

Let $x(t)=\gamma(t) + X(t)$, with $X\in\mathbb{R}^n$, which transforms \eqref{eq:melnikoveq} to the non-autonomous problem

$$X'=A(t)X +g(X,t,\mu)$$

with the non-autonomous matrix
$A(t):= D_x h(\gamma(t);\mu_0)$
and the nonlinear remainder
$g(X,t;\mu)= h(\gamma+X;\mu)-h(\gamma;\mu_0) - A(t)X\,.$
The linear equation
\begin{equation}\label{EQ:VAREQ}
X'=A(t)X
\end{equation}
is the {\it variational equation} along $\gamma(t)$. The corresponding {\it adjoint equation} is given by

$$\Psi' +A^\top(t)\Psi=0\,.$$

Note that solutions of the variational and adjoint equations preserve a constant angle along $\gamma$, i.e.,
$(\partial/\partial t)(\Psi^\top(t) X(t))=0, \forall t\in\mathbb{R}\,.$
We can use this fact to define a splitting of the vector space $\mathbb{R}^n$ along $\gamma(t)$.
Following Vanderbauwhede in \cite{Vanderbauwhede_1992}, we define a splitting of the vector space $\mathbb{R}^n$ at $t=0$ in the following way:

$$\mathbb{R}^n=\mbox{span}\,\{h(\gamma(0),\mu_0)\} \oplus Y$$

where $Y$ is a complementary subspace that admits a further sub-splitting into dynamically distinct subspaces. In view of our assumption on $\Gamma$, we have a splitting

$$Y = V_s \oplus V_u \oplus W,$$

where $V_{s,u}$ denote subspaces complementary to $T_{\gamma(0)}W^u_{\mu_0}(p_-) \cap T_{\gamma(0)}W^s_{\mu_0}(p_+)$ inside the respective (un)stable tangent spaces $T_{\gamma(0)}W^{s,u}$ and $W$ denotes the orthogonal complement of $T_{\gamma(0)}W^{s} + T_{\gamma(0)}W^{u}$.\footnote{When $\dim (T_{\gamma(0)}W^u_{\mu_0}(p_-) \cap T_{\gamma(0)}W^s_{\mu_0}(p_+)) > 1$, the splitting for $Y$ generalises to incorporate an additional subspace $Y =  U \oplus V_s \oplus V_u \oplus W$. In this case, $U$ can be chosen complementary to $\mbox{span}\,\{h(\gamma(0),\mu_0)\}$ inside this intersection; see \cite{Vanderbauwhede_1992} for details.}\\

Based on this setup, we now specify $\Sigma=\gamma(0) + Y$ as a suitable cross-section to perform computations. A key observation of Vanderbauwhede in \cite{Vanderbauwhede_1992} is that---by means of projection maps $P_-(t)$ (for $t\leq 0$) and $Q_+(t)$ (for $t \geq 0)$, defined using properties of exponential dichotomies, together with the variation of parameters formula---the (un)stable manifolds of the saddle points can be characterised analytically as the set of initial conditions whose (forward respectively backward) trajectories have bounded norms in appropriately chosen function spaces. Indeed, $W^s_{\mu}(p_+)$ and $W^u_{\mu}(p_-)$ are locally expressible as graphs over the images of the corresponding projection maps near the intersection point $\gamma(0) \in \Sigma$. For a suitable neighbourhood $\Omega$ containing $\gamma(0)$, we have

 \begin{align} \label{eq:stablemangraph}
W_{\mu}^s(p_+) \cap \Omega &= \{\xi +\beta_+(\xi,\mu): \xi \in \omega_+ \subset \text{Im}(P_+(0))\}
\end{align}

and 

 \begin{align} \label{eq:unstablemangraph}
W_{\mu}^u(p_-) \cap \Omega &= \{\eta +\beta_-(\eta,\mu): \eta \in \omega_- \subset \text{Im}(Q_-(0))\},
\end{align}
 
 where the graphs $h_{\pm}$ are defined using the transition matrix $\Phi(t,s)$ of \eqref{EQ:VAREQ}, the projection operators $P_-(0)$ and $Q_+(0)$, and the corresponding (un)stable manifold segments  $\gamma_\pm$ of the saddle equilibria, i.e.,
 
\begin{align}  \label{eq:stablemangraph2}
\beta_+(\xi,\mu) &:= -Q_+(0) \int_{0}^{\infty} \Phi(0,s)g(s,\gamma_+(\xi,\mu)(s),\mu)ds
\end{align}

 and 
 
 \begin{align}\label{eq:unstablemangraph2}
\beta_-(\eta,\mu) &:= P_-(0) \int_{-\infty}^{0} \Phi(0,s)g(s,\gamma_-(\eta,\mu)(s),\mu)ds.
\end{align}

Note that the splitting defines a local coordinate system in the neighbourhood $\Omega$. The cross-section is given in local coordinates by $\xi = \eta = 0$. We may therefore define a suitable distance function on $\Sigma$ by

\begin{align} \label{eq:distancecoords}
\Delta(\mu) &:= \beta_-(0,\mu)-\beta_+(0,\mu).
\end{align}

We can specify an orthonormal basis $\{\psi^0_1, \cdots, \psi^0_k\}$ of $W$, $k=n+1-(l_u+l_s)$,\footnote{Based on our assumptions, $\dim W= n+1 -(l_u+l_s)\ge 1$.} so that each solution $\psi_i(t)$ of the adjoint equation with initial conditon $\psi_i(0) = \psi^0_i$ decays exponentially for $t\to\pm \infty$.\footnote{Such solutions of the adjoint quation with exponential decay  for $t\to\pm\infty$ exist due to the exponential dichotomy properties induced by the saddle endpoints of this heteroclinic orbit $\Gamma$.}  Melnikov theory then establishes the following formula for the distance function, expressed in components
%which is more practical for computation 
(see e.g. \cite{Vanderbauwhede_1992,Wechselberger_2002}):

\begin{comment}
\begin{equation}\label{equ:distance-fct2}
\begin{aligned}
\Delta(\mu) & =\sum_{i=1}^k \psi_i^0 
\left( \int_{-\infty}^0\psi_i(s)^\top g(\gamma_-(\mu)(s),s;\mu)\,ds +
\int_{0}^{\infty}
\psi_i(s)^\top g(\gamma_+(\mu)(s),s;\mu)\,ds\right)\,\\
&=:\sum_{i=1}^k\psi_i^0 \int_{-\infty}^\infty \psi_i(s)^\top g(\gamma(\mu)(s),s;\mu)\,ds,
\end{aligned}
\end{equation}
\end{comment}

%\begin{comment}
\begin{equation}\label{equ:distance-fct2}
\begin{aligned}
\Delta_i(\mu) & =
\int_{-\infty}^0\psi_i(s)^\top g(\gamma_-(\mu)(s),s;\mu)\,ds +
\int_{0}^{\infty}
\psi_i(s)^\top g(\gamma_+(\mu)(s),s;\mu)\,ds\,\\
&=:\int_{-\infty}^\infty \psi_i(s)^\top g(\gamma(\mu)(s),s;\mu)\,ds,\qquad i=1,\ldots,k\,,
\end{aligned}
\end{equation}
%\end{comment}

where $\gamma$ in the final line should be interpreted as the representative of $\gamma_{\pm}$ in the relevant domain. 

\begin{remark}
This distance function (vector) $\Delta:\mathbb{R}^m\to\mathbb{R}^k$ is well-defined since the vector $\psi(s)$ decays exponentially in forward and backward time and $\gamma_\pm$ is bounded.
Solving $\Delta(\mu)=0$ is well defined for $m\ge k$. For $m=k$ and $\mbox{rk}\, (D_\mu\Delta)=k$, $\mu=\mu_0$ is an isolated zero. For $m>k$, we expect the solution set to be a submanifold of codimension $k$.
\end{remark}
\begin{remark}
The distance formula has a particularly simple representation which is independent of the projection operators, because we have some freedom in choosing $P_-(t)$ and $Q_+(t)$; indeed, we can choose them so that their kernels are orthogonal to the adjoint solution $\psi(t)$ at $t=0$. This is relevant for the derivation of the piecewise-smooth Melnikov function in the next section.
\end{remark}

In general, one cannot solve $\Delta(\mu)=0$ explicitly, but one can calculate its leading order Taylor series expansion.
%i.e., one can write down suitable formulas for the Melnikov integrals in \eqref{eq:melnikovintegrals}. 
Taking partial derivatives of \eqref{equ:distance-fct2} with respect to $\mu_j$ and referring to the definition of the nonlinear remainder $g(X,t;\mu)$, we obtain the following formulas, defined component-wise:

\begin{align} \label{eq:melnikovintegral2}
D_{\mu_j} \Delta_i (\mu_0) &= \int_{-\infty}^\infty \psi_i(s)^\top \frac{\partial h}{\partial \mu_j}(\gamma_0(s),s;\mu_0)\,ds,\qquad j = 1,\ldots m\,,\quad i = 1,\ldots k\,.
\end{align}

\begin{definition}
The first-order derivative terms \eqref{eq:melnikovintegral2} are known as first-order \emph{Melnikov integrals}.    
\end{definition}

\begin{remark}
In the planar case that we consider, we have $k = 1$ and $Y=W$. A suitable adjoint solution $\psi(t)$ can be chosen so that $\psi(0)$ is orthogonal to $h(\gamma(0),\mu_0)$ at $t=0$. 
\end{remark}

\begin{comment}
 We seek to show the existence of a path $\mu = \mu(s)=(\mu_1(s),\ldots,\mu_m(s))$ in parameter space containing $\mu_0$ with $\Delta(\mu(s)) = 0$, i.e. the existence of nearby heteroclinic connections along this path.  Suppose that $D_{\mu_j} \Delta (\mu_0)\neq 0$ for some $j \in\{ 1,\cdots,m\}$, and let $\delta = \mu_j - \mu_0_{j}$ where $\mu_0_{j}$ denotes the $j$-th component of $\mu_0$. Then by the implicit function theorem we have that $\mu = (\mu_1(\mu_j),\mu_2(\mu_j),\cdots,\mu_j,\cdots,\mu_m(\mu_j)),$ given by
    $$\mu_i=\mu_i(\mu_j)=\mu_0_{i} + b_i \delta +O(\delta^2)$$ 
    
  for $i \neq j$,  solves $\Delta(\mu)=0$ for $\delta\in (-\delta_0,+\delta_0)$. In the planar setting, the distance $\Delta$ may be identified with a scalar function, in which case the leading order parameters are then given by 
 \begin{align} \label{eq:melnikovintegrals}
b_i &=-\frac{D_{\mu_i} \Delta (\mu_0)}{D_{\mu_j} \Delta (\mu_0)}.
\end{align}

 The first-order derivative terms in the $b_i$ formulas are known as first-order \emph{Melnikov integrals}. We now turn to a derivation of computable formulas for the distance $\Delta = \Delta(\mu)$ and for these Melnikov integrals.\\
\end{comment}

%%%%%%%%%%%%%%%%%%%%%%%%%%%%%%%%%%%%%%%%%%%%%%%%%%%%%%%%%%%%%%%%%%%%%%%%%%%%%%%%%

\subsection{The piecewise-smooth planar case}\label{sec:pwmelnikov}

We now adapt the Melnikov method developed in Appendix~\ref{sec:smoothmelnikov} to a subclass of piecewise-smooth planar vector fields, thereby proving Lemma \ref{lem:pwmelnikov}. Fix $\delta > 0$ and $x_{\Sigma} \in \mathbb{R}$, and consider a pair of planar vector fields $h_\pm(x,\mu)=(h_\pm^1(x,\mu),h_\pm^2(x,\mu))$, such that $h_-$ is smoothly defined on a compact subset $V_- \subset \{x_1 < x_{\Sigma} + \delta\}$ and $h_+$ is smoothly defined on a compact subset $V_+ \subset \{x_1 > x_{\Sigma}-\delta\}$. Both vector fields $h_\pm$ are smoothly dependent on the parameters $\mu \in \mathbb{R}^{m}$.\\

We assume that $V_- \cap V_+$ is nonempty and contains a (vertical) cross-section $\Sigma \subset \{x_1 =  x_{\Sigma}\}$ of both vector fields $h_\pm$ for each $\mu \in P\subset\mathbb{R}^{m}$ for a non-trivial subset $P$ in parameter space. In particular, both vector-fields are assumed to cross (the interior of) $\Sigma$ unidirectionally, e.g., here from left to right, for each $\mu \in P\subset\mathbb{R}^{m}$.
%We define $U_- = V_- \cap \{x_1 < x_{\Sigma}\}$ and $U_+ = V_+ \cap \{x_1 > x_{\Sigma}\}$. 
With this setup, we may define a piecewise smooth vector field on 
$U_- \cup U_+$, with $U_-:= V_- \cap \{x_1 < x_{\Sigma}\}$ and $U_+:=V_+ \cap \{x_1 > x_{\Sigma}\}$, as follows:

\begin{align} \label{eq:piecewiseeqns}
\dot{x} &= \begin{cases}
h_-(x,\mu)& \text{ if } x \in U_-\\
h_+(x,\mu)& \text{ if } x \in U_+.
\end{cases}
\end{align}
%

\begin{comment}
We focus on vector fields crossing $\Sigma$ transversally in one direction---for concreteness, we assume that the $x_1$-components of each vector field satisfy 
\begin{equation}
\label{eq:pwtransversality}
\begin{aligned}
h_-(x,\mu)_{1}|_{\Sigma} &> 0 \text{ and}\\
h_+(x,\mu)_{1}|_{\Sigma} &> 0
\end{aligned}
\end{equation}
for each $\mu \in P$. 
\end{comment}

We assume the existence of a {\it piecewise-smooth heteroclinic orbit} which connects a saddle equilibrium $p_- \in U_-$ to one at $p_+ \in U_+ $ for some $\mu = \mu_0$ in (the interior of) $P$ as follows:
%\footnote{is smoothness in time necessary? if yes, then shortc explanation needed how to obtain (49) with that property. \textit{No, I believe for our setup that it is sufficient to have the desired smoothness separately on the overlapping sets $V_{\pm}$.}} 
specify an unstable manifold segment $W^u(p_-) \subset U_- \cup \Sigma$, defined for $t \leq 0$ by a trajectory $\gamma_-(t)$ satisfying $\dot{x} = h_-(x,\mu_0)$, and a stable manifold segment $W^s(p_+) \subset U_+ \cup \Sigma$ defined for $t \geq 0 $ by a trajectory $\gamma_+(t)$ satisfying $\dot{x} = h_+(x,\mu_0)$, such that $\gamma(0) := \gamma_-(0) = \gamma_+(0) \in \Sigma$. The piecewise-smooth heteroclinic orbit $\gamma_0$ is then defined by the set

$$\gamma_0 = \left(\bigcup_{t < 0} \gamma_-(t)\right) \cup \gamma(0) \cup  \left(\bigcup_{t > 0} \gamma_+(t)\right).$$ 

See Figure~\ref{fig:pwheteroclinic} for a depiction of such a piecewise-defined $\gamma_0$. As in the smooth case, we seek to determine the persistence of $\gamma_0$ under parameter variation via a Melnikov calculation on $\Sigma$. \\

\begin{figure}
    \centering
    \includegraphics[width=10cm]{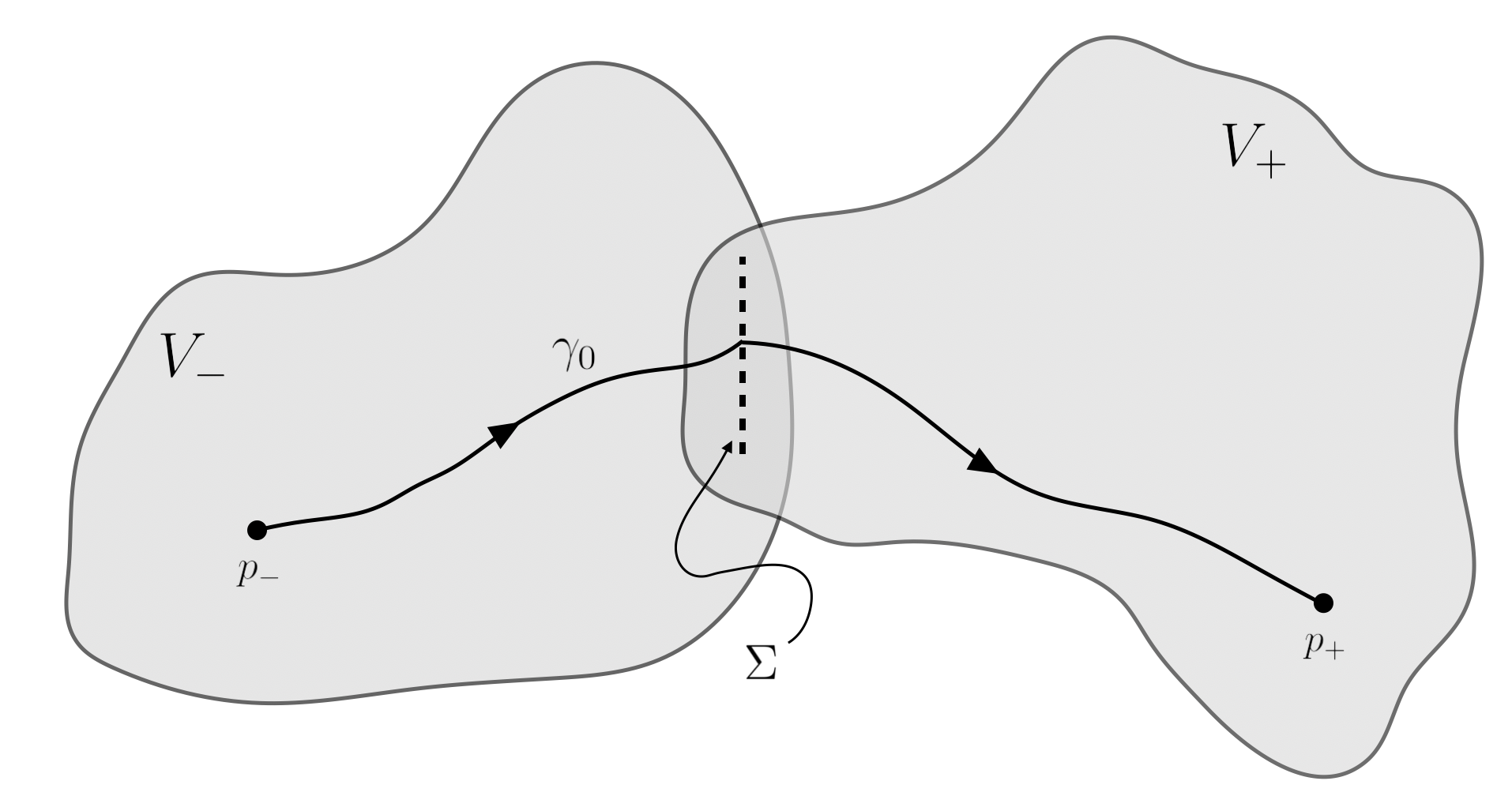}    
    \caption{A sketch of a piecewise-smooth heteroclinic connection $\gamma_0$.} 
    \label{fig:pwheteroclinic}
\end{figure}

Much of the setup in Appendix~\ref{sec:smoothmelnikov} carries over for the problems $\dot{x} = h_{\pm}(x,\mu)$ defined separately on the subsets $V_{\pm}$. On $V_-$, we can specify a section $\Sigma_-$ through $\gamma(0)$ which is aligned with a backward exponentially decaying adjoint solution $\psi_-(t)$ at $t = 0$. The associated graph representation of the unstable manifolds in a neighbourhood of $\gamma(0)$ is given by

 \begin{align}\label{eq:unstablemangraph2pw}
\beta_-(\eta,\mu) &:= P_-(0) \int_{-\infty}^{0} \Phi_-(0,s)g_-(s,\gamma_-(\eta,\mu)(s),\mu)ds,
\end{align}

\begin{figure}
    \centering
    \includegraphics[width=10cm]{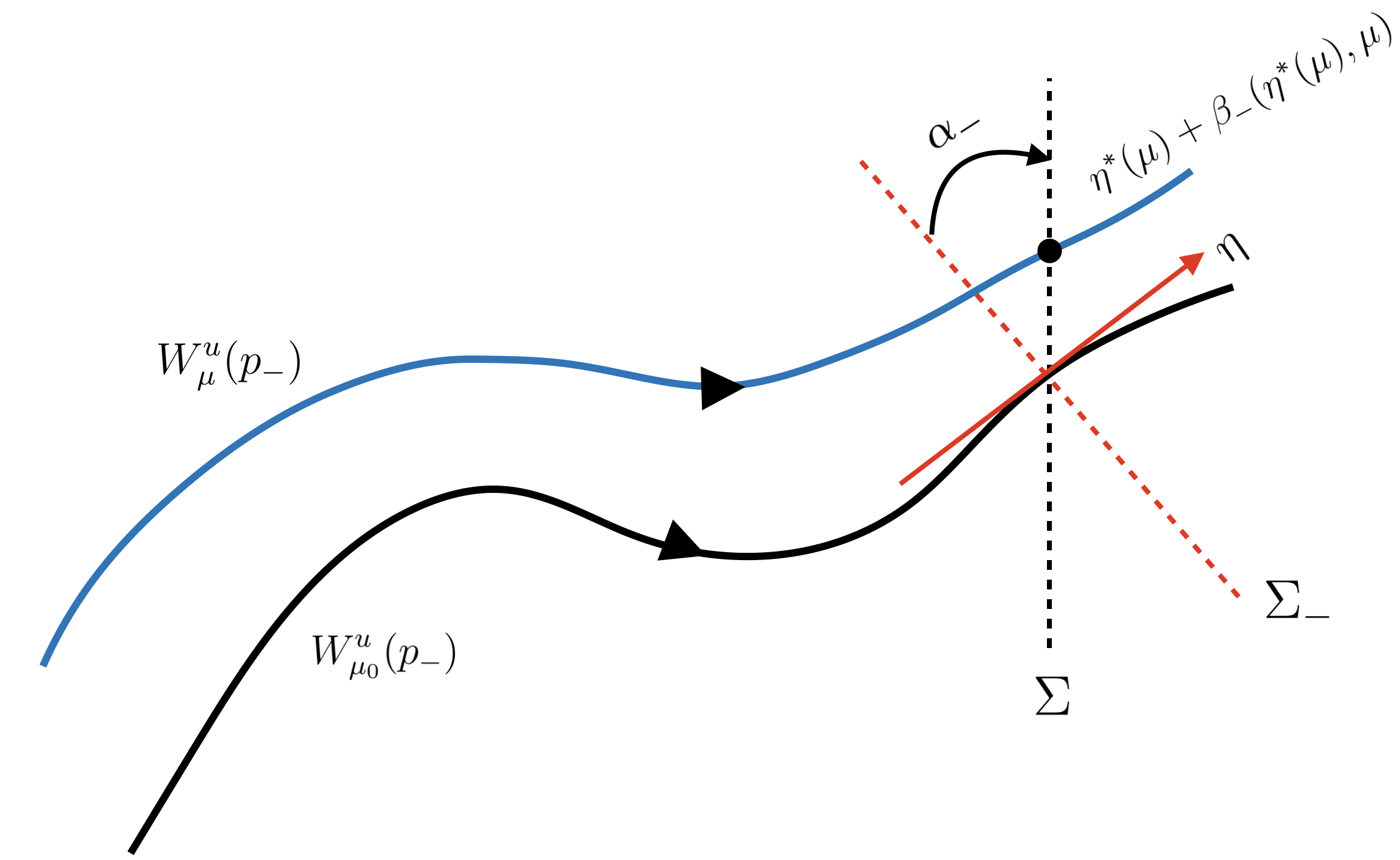}    
    \caption{Local coordinates near $\Sigma$ on $V_-$.} 
    \label{fig:slopecoords}
\end{figure}

where we denote the appropriate objects in $V_-$ by the $(-)$ subscript. Similarly, near an appropriately chosen cross-section $\Sigma_+$ which is aligned with a forward exponentially decaying adjoint solution $\psi_+(t)$ at $t = 0$ there is a graph representation of the stable manifolds on $V_+$, given by 

\begin{align}  \label{eq:stablemangraph2pw}
\beta_+(\xi,\mu) &:= -Q_+(0) \int_{0}^{\infty} \Phi_+(0,s)g_+(s,\gamma_+(\xi,\mu)(s),\mu)ds.
\end{align}

The main conceptual difference in deriving the distance function in the piecewise-smooth case  is that the local coordinates used to specify the (un)stable manifold segments near $\Sigma$ are adapted to the sections $\Sigma_-$ and $\Sigma_+$, neither of which is aligned with $\Sigma$ in general (e.g. see Fig. \ref{fig:slopecoords}). This introduces extra terms into the separation function when measuring the distance between $W^u(p_-)$ and $W^s(p_+)$ on $\Sigma$, as we now show.  \\

We denote by $\alpha_\pm\in (-\pi/2,\pi/2)$ the acute angle of $\psi_\pm(0)$ relative to the vertical cross section $\Sigma$ spanned by $e_2$, or, equivalently, the acute angle of $\dot\gamma_\pm(0)$ relative to the horizontal coordinate axis spanned by $e_1$. Since $\dot\gamma_\pm(0)$ lies in the orthogonal complement of the linear subspaces ${\cal W}^\pm$, i.e., in the span of $\psi^\perp_\pm(0)$, we define the corresponding orthonormal basis vectors in $\mathbb{R}^2|_{\gamma(0)}=
\mbox{span}\,\{\psi^\perp_\pm(0)\} \oplus \mbox{span}\,\{\psi_\pm(0)\}$ 
with respect to the angle $\alpha_\pm$:

$$
\psi_\pm^\perp(0)=
\begin{pmatrix}
\cos(\alpha_\pm)\\
\sin(\alpha_\pm)
\end{pmatrix}
\,,\qquad
\psi_\pm(0)=
\begin{pmatrix}
-\sin(\alpha_\pm)\\
\cos(\alpha_\pm)
\end{pmatrix}
\,.
$$

Hence, the local (un)stable manifolds $W^{s/u}_\pm (\mu)$ near $\gamma(0)$ described in this basis are given by

\begin{equation}
%\label{eq:distancecoords2}
\begin{aligned}
W^{u}_-(\mu) &=\{\eta+\beta_-(\eta,\mu)=\psi_-^\perp(0)\rho_- + \psi_-(0)\sigma_- (\mu)\},\\
W^{s}_+(\mu) &=\{\xi+\beta_+(\xi,\mu)=\psi_+^\perp(0)\rho_+ + \psi_+(0)\sigma_+ (\mu)\}
\end{aligned}
\end{equation} 

with local coordinates $\rho_\pm\in\mathbb{R}$ and $\sigma_\pm (\mu)\in\mathbb{R}$,

$$
\sigma_\pm (\mu) =\psi_\pm^\top(0)\beta_\pm = \int_{\pm\infty}^0 \psi_\pm^\top(s)g_\pm(s,\gamma_\pm(s),\mu)\,ds\,.
$$

 To measure the distance of these manifolds $W^{s/u}_\pm (\mu)$ in $\Sigma$, we identify the unique intersection points $W^{s/u}_\pm (\mu)\cap \Sigma$ which are given by the coordinate relationship 
$\rho_\pm^\ast :=\tan (\alpha_\pm)\sigma_\pm (\mu)$ in the basis defined above, i.e., this coordinate relationship guarantees that the horizontal component vanishes in the local Euclidean coordinate frame. From the implicit function theorem it follows that this coordinate relationship can be indeed solved locally as a function of the parameter $\mu$ as shown.  See Figure~\ref{fig:slopecoords} for a sketch of the local coordinates on the patch $V_-$. A similar sketch can be drawn for the patch $V_+$ to measure the location of the stable manifold $W^s_{\mu}(p_+)$ for $\mu$ sufficiently close to $\mu_0$.\\

The distance function on $\Sigma$ can now be defined as
\begin{equation}
%\label{eq:distancecoords2}
\begin{aligned}
\Delta(\mu) &= \eta^*(\mu) - \xi^*(\mu) + \beta_-(\eta^*(\mu),\mu) - \beta_+(\xi^*(\mu),\mu)\\
&= \psi_-^\perp(0) \tan (\alpha_-)\sigma_- (\mu) - \psi_+^\perp(0) \tan (\alpha_+)\sigma_+ (\mu)
+ \psi_-(0)\sigma_- (\mu) - \psi_+(0)\sigma_+ (\mu)\,.
\end{aligned}
\end{equation} 
Since $\Delta(\mu)\in\Sigma$, we are only interested in the vertical component of this vector which is given by

\begin{equation} \label{eq:distancecoords2}
\begin{aligned}
e_2^\top \Delta(\mu)&=
\sin (\alpha_-) \tan (\alpha_-)\sigma_- (\mu) - \sin (\alpha_+) \tan (\alpha_+)\sigma_+ (\mu)
+ \cos (\alpha_-)\sigma_- (\mu) - \cos (\alpha_+)\sigma_+ (\mu)
\\
&=\frac{1}{\cos (\alpha_-)}\sigma_- (\mu) - \frac{1}{\cos (\alpha_+)}\sigma_+ (\mu)\\
&= \frac{1}{\cos (\alpha_-)}\int_{-\infty}^0 \psi_-^\top(s)g_-(s,\gamma_-(s),\mu)\,ds 
+ \frac{1}{\cos (\alpha_+)}\int^{\infty}_0 \psi_+^\top(s)g_+(s,\gamma_+(s),\mu)\,ds\,.
\end{aligned}
\end{equation}

\begin{remark}
The formula \eqref{eq:distancecoords2} also generalises the classical distance formula \eqref{eq:distancecoords} for sections $\Sigma$ that are nonorthogonal to the flow in the smooth case. 
\end{remark}

Let $u_{\pm} = (v_{1,\pm},v_{2,\pm})^T$ denote unit basis vectors obtained from normalising the vector fields $h_{\pm}(x,\mu_0)$ evaluated at $\gamma(0)$. An algebra calculation shows that $\cos(\alpha_{\pm}) = v_{1,\pm}$ and we may choose $\psi_{\pm}(0) = (-v_{2,\pm},v_{1,\pm})^T$. Then  the corresponding first-order Melnikov integrals, %\eqref{eq:melnikovintegralspw}
obtained as in the classical case by differentiating \eqref{eq:distancecoords2} with respect to $\mu_i$ and evaluating at $\mu = \mu_0$, have the particularly convenient form
\begin{equation}
\label{eq:melnikovintegralspw2}
\begin{aligned}
D_{\mu_i}G(\mu_0) &= \frac{1}{v_1^-}\int_{-\infty}^0 \psi_-(s)^{\top}  \frac{\partial h_-}{\partial \mu_i}(\gamma_-(s),s;\mu_0)\,ds
+ \frac{1}{v_1^+}\int_{0}^{\infty} \psi_+(s)^{\top}  \frac{\partial h_+}{\partial \mu_i}(\gamma_+(s),s;\mu_0)\,ds.
\end{aligned}
\end{equation}

 The formula \eqref{eq:melnikovintegralspw2} gives the promised generalisation of \eqref{eq:melnikovintegral2} in the piecewise-smooth  planar setting. Note that if $u_{-} = u_{+}$ (so that the heteroclinic $\gamma$ connects smoothly across the discontinuity surface $\Sigma$), then \eqref{eq:melnikovintegralspw2} is identical to \eqref{eq:melnikovintegral2} up to a positive prefactor (with the improper integral appropriately interpreted as a sum of two integrals evaluated separately on either side of the discontinuity).\\
 
In fact, we also recover the classical Melnikov integral (up to a positive prefactor) if $v_1^- = v_1^+$ only. This observation is useful when the piecewise-smooth problem inherits a reflection symmetry in the vertical component from a smooth problem; see Sec. \ref{sec:melnikovreduced}.

\end{document}